\documentclass{lmcs}
\pdfoutput=1

\usepackage{lastpage}
\lmcsdoi{21}{1}{14}
\lmcsheading{}{\pageref{LastPage}}{}{}%
{Sep.~22,~2022}{Feb.~06,~2025}{}

\usepackage[utf8]{inputenc}
\usepackage[T1]{fontenc}
\usepackage[english]{babel}

\usepackage{microtype}

\usepackage{lineno}

\usepackage[inline]{enumitem}
\usepackage{mathtools}
\usepackage{xspace}
\usepackage{xcolor}
\usepackage{comment}

\usepackage{etoolbox}
\usepackage{xifthen} %

\usepackage{tikz-cd}
\usepackage{tikz}
\usetikzlibrary{decorations.pathmorphing}

\usepackage{amssymb}
\usepackage{amsmath}
\usepackage{amsthm}
\usepackage{thmtools}
\usepackage{bussproofs}

\usepackage[hypertexnames=false]{hyperref}

\usepackage[capitalise]{cleveref}

\usepackage[draft]{optional}
\usepackage[colorinlistoftodos,prependcaption,textsize=tiny]{todonotes}%

\newcommand{\plan}[1]{}
\newcommand{\BA}[1]{}
\newcommand{\PN}[1]{}
\newcommand{\ER}[1]{}
\newcommand{\JE}[1]{}
\opt{draft}{
   \renewcommand{\plan}[1]{\todo[color=blue!30]{Plan: #1}\PackageWarning{TODO}{Plan: #1}}
   \renewcommand{\BA}[1]{\todo[color=orange!30]{BA: #1} \PackageWarning{TODO}{BA: #1}}
   \renewcommand{\ER}[1]{\todo[color=green!30]{ER: #1}\PackageWarning{TODO}{ER: #1}}
   \renewcommand{\JE}[1]{\todo[color=red!30]{JE: #1}\PackageWarning{TODO}{JE: #1} }
   \renewcommand{\PN}[1]{\todo[color=violet!30]{PN: #1} \PackageWarning{TODO}{PN: #1}}
}

\tikzset{fib/.style={->>,font=\scriptsize}}

\theoremstyle{plain}
\theoremstyle{definition}
  \Crefname{defi}{Definition}{Definitions}
  \Crefname{rem}{Remark}{Remarks}
  \Crefname{rems}{Remark}{Remarks}
  \Crefname{prob}{Problem}{Problems}
\Crefname{cor}{Corollary}{Corollaries}
\Crefname{lem}{Lemma}{Lemmas}
\Crefname{exa}{Example}{Examples}
\Crefname{egs}{Examples}{Examples}
\Crefname{thm}{Theorem}{Theorems}
\Crefname{prop}{Proposition}{Propositions}

\newtheorem{constrInternal}[thm]{Construction}
\newenvironment{construction}[2][]
{\pushQED{\qed}\begin{constrInternal}[{%
                                       for~\cref{#2}%
                                      }%
                                     ]
  }
  {\popQED\end{constrInternal}}

\Crefname{problem}{Problem}{Problems}
\Crefname{constrInternal}{Construction}{Constructions}

\newcommand{\define}[1]{\textbf{#1}}
\newcommand{\ft}[1][]{\mathsf{ft}_{#1}}
\newcommand{\bd}[1][]{\partial_{#1}}
\newcommand{\jdeq}{=}
\newcommand{\comp}{\circ}

\newcommand{\catof}[1]{\mathbf{#1}}
\newcommand{\opcat}{^{\mathrm{op}}}
\newcommand{\rttr}{\catof{RtTr}}
\newcommand{\bfr}{\catof{Bfr}}
\newcommand{\grph}{\catof{Grph}}
\newcommand{\Cat}{\catof{Cat}}
\newcommand{\Set}{\catof{Set}}
\newcommand{\catid}[1]{\mathrm{id}_{#1}}
\newcommand{\cat}[1]{\mathcal{#1}}
\newcommand{\C}{\cat{C}}
\newcommand{\D}{\cat{D}}
\newcommand{\F}{\cat{F}}
\newcommand{\converts}{\equiv}
\newcommand{\type}{\xspace \mathsf{type} \xspace}
\newcommand{\term}[1]{\tilde{#1}}
\newcommand{\N}{\mathbb{N}}
\newcommand{\ob}[1]{\mathrm{Ob}\inpar1{#1}}
\newcommand{\length}{{\ell}}

\newcommand{\defeq}{:=}
\newcommand{\U}{\mathcal{U}}
\newcommand{\tU}{\tilde{\mathcal{U}}}
\newcommand{\epsi}{\varepsilon}
\renewcommand{\phi}{\varphi}
\newcommand{\Mor}[1]{\mathrm{Mor}({#1})}

\newcommand{\te}{\mathsf{term}}
\newcommand{\ty}{\mathsf{type}}
\newcommand{\ff}[1]{{[#1]}}

\newcommand{\wt}{\widetilde}

\newcommand{\blank}{\ensuremath{-}}
\newcommand{\ctxext}[2]{{#1}.{#2}}
\newcommand{\inpar}[2]{%
	\relax\ifcase#1\relax{#2}
	\else{\left(#2\right)}\fi
}
\newcommand{\sys}[1]{\ensuremath{\mathbb{#1}}}
\newcommand{\assys}[2]{%
	\relax\ifcase#1\relax{#2}
	\else{\sys{#2}}\fi
}
\newcommand{\ptdcat}[2]{\inpar{#1}{#2}_{*}}
\newcommand{\strcat}[1]{\catof{#1_s}}
\newcommand{\strCat}{\strcat{Cat}}

\newcommand{\indecarr}{indecomposable\xspace}

\newcommand{\Bsys}{\catof{Bsys}}
\newcommand{\Bfr}{\catof{Bfr}}
\newcommand{\Slist}[2]{S_{#1}^{#2}}
\newcommand{\SlistT}[2]{\tilde{S}_{#1}^{#2}}

\newcommand{\ctxwk}[2]{\langle{#1}\rangle{#2}}
\newcommand{\subst}[2]{{#2}[{#1}]}
\newcommand{\Esys}{\catof{Esys}}
\newcommand{\strEsys}{\strcat{\Esys}}
\newcommand{\rEsys}{\catof{Esys}}%
\newcommand{\Esysptd}{\ptdcat0{\Esys}}
\newcommand{\Efam}[1]{\cat{F}_{\sys{#1}}}
\newcommand{\Eroot}[1]{[~]_{#1}}
\newcommand{\Erootsys}[1]{\Eroot{\sys{#1}}}
\newcommand{\EtrmCat}[2]{\cat{C}_{\sys{#1}}\!\left(#2\right)}
\newcommand{\EtrmCatcl}[1]{\cat{C}_{\sys{#1}}}
\newcommand{\Etrmfun}[2]{I_{\sys{#1}}^{#2}}
\newcommand{\Etrmfuncl}[1]{I_{\sys{#1}}}
\newcommand{\Efcmp}[2]{\ctxext{#1}{#2}}
\newcommand{\Epcmp}[4]{
	\inpar{#1}{#3}^*\!\inpar{#2}{#4}
}
\newcommand{\Epcmpf}[2]{\Epcmp{#1}{0}{#2}{\,}}
\newcommand{\Epcin}[4]{
	\inpar{#2}{#4} \cdot \inpar{#1}{#3}
}

\newcommand{\Ehtrm}[2]{#1_{#2}}

\newcommand{\tfid}[1]{\mathsf{idtm}_{#1}}
\newcommand{\thom}[2]{\mathrm{thom}(#1,#2)}
\newcommand{\thomd}[3]{\mathrm{thom}_{#1}(#2,#3)}
\newcommand{\jcomp}[3]{\Epcin00{#2}{#3}}
\newcommand{\tmext}[2]{\ctxext{#1}{#2}}
\newcommand{\cprojfstf}[2]{\mathrm{pr}_0^{#1,#2}}
\newcommand{\cprojsndf}[2]{\mathrm{pr}_1^{#1,#2}}
\newcommand{\cprojfst}[3]{\subst{#3}{\cprojfstf{#1}{#2}}}
\newcommand{\cprojsnd}[3]{\subst{#3}{\cprojsndf{#1}{#2}}}
\newcommand{\jhom}[4]{{#4} \in \thom{#2}{#3}}
\newcommand{\jhomd}[7]{#7 \in \thomd{#4}{#5}{#6}}
\newcommand{\jvcomp}[3]{#2 \ltimes #3}
\newcommand{\jfcomp}[5]{#5 \bullet #4}

\newcommand{\CEsys}{\catof{CEsys}}
\newcommand{\rCEsys}{\catof{rCEsys}}
\newcommand{\strCEsys}{\strcat{\CEsys}}
\newcommand{\strrCEsys}{\strcat{\rCEsys}}
\newcommand{\CEroot}[2]{\mathsf{1}_{\assys{#1}{#2}}}
\newcommand{\CEctxext}[2]{\ctxext{#1}{#2}}
\newcommand{\CEfcmp}[2]{\ctxext{#1}{#2}}
\newcommand{\CEpb}[4]{%
	\inpar{#1}{#3}^*\!\inpar{#2}{#4}
}
\newcommand{\CEpbvar}[4]{%
	\inpar{#1}{#3}^{\hat{*}}\!\inpar{#2}{#4}
}
\newcommand{\CEpbf}[2]{\CEpb{#1}{0}{#2}{\,}}
\newcommand{\CEqar}[2]{\pi_2\left(#1,#2\right)}
\newcommand{\CEqarvar}[2]{\hat{\pi_2}\left(#1,#2\right)}
\newcommand{\CEcat}[2]{\cat{C}_{\assys{#1}{#2}}}
\newcommand{\CEfam}[2]{\cat{F}_{\assys{#1}{#2}}}
\newcommand{\CEfun}[2]{I_{\assys{#1}{#2}}}
\newcommand{\CEsl}[2]{\sys{#1}/{#2}}
\newcommand{\CEslcat}[3]{\CEcat{#1}{#2}\!\left(#3\right)}
\newcommand{\CEslfun}[1]{I_{#1}}
\newcommand{\hCEfam}[1]{#1_{\cat{F}}}
\newcommand{\hCEcat}[1]{#1_{\cat{C}}}

\newcommand{\p}{\mathsf{p}}
\newcommand{\q}{\mathsf{q}}
\newcommand{\Csys}{\catof{Csys}}
\newcommand{\Cterminal}{\mathsf{1}}

\newcommand{\isfctr}[1]{\mathbf{#1}}
\newcommand{\Ffctr}{\isfctr{F}}
\newcommand{\Gfctr}{\isfctr{G}}
\newcommand{\Rfctr}{\isfctr{R}}
\newcommand{\Ifctr}{\isfctr{I}}
\newcommand{\rEtoEp}{\isfctr{E2E_*}}
\newcommand{\CEtorE}{\isfctr{CE2E}}

\newcommand{\EptoCE}{\isfctr{E_*2CE}}
\newcommand{\rEtoCE}{\isfctr{E2CE}}

\newcommand{\CtoCE}{\isfctr{CE}}
\newcommand{\CEtoC}{\isfctr{C}}
\newcommand{\Ctorttr}{\isfctr{C2RtTr}}
\newcommand{\Btorttr}{\isfctr{R}}
\newcommand{\BtoTC}{\isfctr{T}}
\newcommand{\BtoE}{\isfctr{B2E}}
\newcommand{\EtoB}{\isfctr{E2B}}
\newcommand{\BtoC}{\isfctr{B2C}}
\newcommand{\CtoB}{\isfctr{C2B}}

\newcommand{\FF}{\ensuremath{\mathbb{F}}}
\newcommand{\fin}[1]{\ensuremath{[#1]}}
\newcommand{\std}[1]{\ensuremath{\mathsf{std}(#1)}}

\newcommand{\ie}{\textit{i.e.}\xspace}
\newcommand{\scat}{strict category\xspace}
\newcommand{\scats}{strict categories\xspace}

\newcommand{\sslcat}{strict slice category\xspace}

\newcommand{\rtdCE}{rooted\xspace}


\begin{document}

\title{Algebraic presentations of type dependency}
\author[B.~Ahrens]{Benedikt Ahrens\lmcsorcid{0000-0002-6786-4538}}[a]
\author[J.~Emmenegger]{Jacopo Emmenegger\lmcsorcid{0000-0003-1383-2415}}[b]
\author[P.R.~North]{Paige Randall North\lmcsorcid{0000-0001-7876-0956}}[c]
\author[E.~Rijke]{Egbert Rijke\lmcsorcid{0000-0002-5272-6175}}[d]

\address{Delft University of Technology, Netherlands, and University of Birmingham, UK}
\email{B.P.Ahrens@tudelft.nl}
\address{University of Genoa, Italy}
\email{jacopo.emmenegger@unige.it}
\address{Utrecht University, Netherlands}
\email{p.r.north@uu.nl}
\address{Johns Hopkins University, USA, and University of Ljubljana, Slovenia}
\email{erijke1@jh.edu}
\thanks{%
  This work was partially funded by EPSRC grant EP/T000252/1 and ``PNRR - Young Researchers'' grant SOE-0000071 of the Italian Ministry of University and Research. This material is based upon work supported by the Air Force Office of Scientific Research under award numbers FA9550-21-1-0334 and FA9550-21-1-0024.
}

\begin{abstract}

  C-systems were defined by Cartmell as the algebraic structures that correspond exactly to generalised algebraic theories.
  B-systems were defined by Voevodsky in his quest to formulate and prove an initiality conjecture for type theories.
  They play a crucial role in Voevodsky's construction of a syntactic C-system from a term monad.

  In this work, we construct an equivalence between the category of C-systems and the category of B-systems, thus proving a conjecture by Voevodsky.
  We construct this equivalence as the restriction of an equivalence between more general structures, called CE-systems and E-systems, respectively.
  To this end, we identify C-systems and B-systems as ``stratified'' CE-systems and E-systems, respectively; that is, systems whose contexts are built iteratively via context extension, starting from the empty context.

\end{abstract}

\maketitle

\section{Introduction}

In his unfinished and only partially published \cite{MR3402489,Voevodsky_relative,MR3584698,MR3475277,MR3607209} research programme on type theories, Voevodsky aimed to develop a mathematical theory of type theories, similar to the theory of groups or rings.
In particular, he aimed to state and prove rigorously an ``Initiality Conjecture'' for type theories, in line with the initial semantics approach to the syntax of (programming) languages (cf.~\Cref{sec:initial-semantics}).

One aspect of this Initiality Conjecture is to construct, from the types and terms of a programming language, a ``model'', that is, a mathematical object (which is supposed to be an initial object in a category of models and their morphisms).
To help with this endeavour in the context of initial semantics for type theories, Voevodsky introduced the essentially-algebraic theory of \emph{B-systems}.
The models of this theory, he conjectured in \cite{VV_B-systems}, are constructively equivalent to the well-known \emph{C-systems} or \emph{contextual categories} first introduced by Cartmell~\cite{DBLP:journals/apal/Cartmell86}.
Furthermore, in his Templeton grant application~\cite{TempletonProposal}, Voevodsky writes:

\begin{quote}
  The theory of B-systems is conjecturally equivalent to the theory of C-systems that were introduced by John Cartmell under the name ``contextual categories'' in [2],[3]. Proving this equivalence is among the first goals of the proposed research.
\end{quote}

The precise role of B-systems in Voevodsky's programme is described in~\cite{MR3475277}; we give an overview in~\cref{sec:context} below.

In this present work, we construct an equivalence of categories between C-systems and B-systems, each equipped with a suitable notion of homomorphism.
Our construction is entirely constructive, in the sense that it does not rely on the law of excluded middle or the axiom of choice.

C- and B-systems are ``stratified'', in a sense that will be defined later (in \cref{ssec:c2ce,ssec:b2e}, respectively).
In this work, we also introduce unstratified structures, under the name of E-system and CE-system, respectively.
We construct an adjunction between these structures, and obtain the equivalence between B- and C-systems via an equivalence of suitable subcategories.
The construction is summarized in the following diagram, in which maps are annotated with the respective section numbers where they are constructed:

\begin{linenomath*}
\[
  \begin{tikzcd}[ampersand replacement=\&,column sep=4em,row sep=4em]
    \Esys
    \ar[rrr,shift right=1ex,"{\rEtoCE, \S\ref{ssec:e2ce}}"{swap,inner sep=1ex}]
    \ar[rrr,phantom,"{\scriptstyle\top}"]
    \ar[from=dr,hook',start anchor={[yshift=-.5ex,xshift=.5ex]north west}]
    \&
    \&
    \&
    \CEsys	\ar[lll,shift right=1ex,"{\CEtorE, \S\ref{ssec:ce2e}}"{swap,inner sep=1ex}]
    \\
    \Bsys	\ar[r,"\simeq",leftrightarrow,"\text{\S\ref{ssec:b2e}}"'] \ar[u,hook]
    \&
    \overset{\mathbf{stratified}}{\Esys}
    \ar[r,"\simeq",leftrightarrow,"\text{\S\ref{ssec:eqv-b-c}}"']
    \&
    \overset{\mathbf{stratified}}{\CEsys}
    \ar[ur,start anchor={[yshift=-.5ex]north east},hook]
    \&
    \Csys
    \ar[l,"\simeq"',leftrightarrow, "\text{\S\ref{ssec:c2ce}}"] \ar[u,hook]
  \end{tikzcd}
\]
\end{linenomath*}

\noindent
The unstratified structures are of interest in their own right:
they will serve, in a follow-up work, to relate C-systems and B-systems to other, well-established, unstratified categorical structures for the interpretation of type theories, such as
categories with families~\cite{Dybjer1996} and natural models~\cite{DBLP:journals/mscs/Awodey18}, categories with attributes~\cite{cartmell_phd,Hofmann_syntax_semantics},%
\footnote{Hofmann \cite[\S\S3.1, 3.2]{Hofmann_syntax_semantics} also compares categories with families and categories with attributes in a set-theoretic setting, and a comparison between these notions in a univalent setting is given in~\cite{DBLP:journals/lmcs/AhrensLV18}.}
and display map categories~\cite{Taylor1999,DBLP:journals/mscs/North19}.

\subsection{Initial Semantics}\label{sec:initial-semantics}

The ``template'' for initial semantics is as follows:
One starts by defining a suitable notion of \emph{signature}---an abstract specification device describing the (types and) terms of a language.
To any signature, one then associates a category of \emph{models} of that signature, in such a way that the\footnote{We are working modulo isomorphism in a category.} initial object in that category---if it exists---deserves to be called the \emph{syntax generated by the signature}.
Finally, one aims to construct such initial objects, or identify sufficient criteria for a signature to admit initial objects.

A particularly simple example of initial semantics is the following: consider the category an object of which is given by a triple $(X,x,s)$ where $X$ is a set, $x \in X$, and $s : X \to X$.
Then the initial object in that category is given by $(\mathbb{N}, 0, (+1))$, and the structure of being initial provides the well-known iteration principle: to define a map $\mathbb{N} \to X$, it suffices to specify $x \in X$ (the image of $0$) and an endomap $s : X \to X$ (the recursive image of $(+1)$).
That is, no explicit application of recursion or induction principles on $\mathbb{N}$ is required once it is established that $(\mathbb{N}, 0, (+1))$ is an initial object; instead, the initiality property provides an interface to these black-boxed principles.

For ``simple'' programming languages (e.g., for untyped or simply-typed lambda calculi), notions of signature, and initial semantics for such signatures, have been constructed; see, e.g., \cite{lamiaux2024introduction} for an overview.

For some specific dependently-typed languages, Streicher~\cite{streicher-semantics-of-tt}, and, more recently, De Boer, Brunerie, Lumsdaine, and Mörtberg~\cite{boer-brunerie-lumsdaine-mortberg}, have constructed initial models.
Voevodsky aimed at developing a general notion of signature for dependently-typed languages, and an initial semantics result for such signatures.
In \Cref{sec:context} we sketch Voevodsky's approach towards a theory of type theories, and the role of C- and B-systems therein.

Meanwhile, Uemura \cite[Section~5]{uemura_phd} has also developed a notion of signature for dependently-typed theories, and an initial semantics result for them.

\subsection{Voevodsky's approach towards a theory of type theories}
\label{sec:context}

In this \lcnamecref{sec:context}, we sketch Voevodsky's plan for an initial semantics result for type theories.
Voevodsky's Bonn lectures~\cite{vv-bonn} served as the main source for this overview.

\subsubsection{Setting the scene}
\label{sec:setting-scene}

In \cite{Voevodsky_relative}, Voevodsky opens with the following statement:
\begin{quote}
\textit{The first few steps in all approaches to the set-theoretic semantics of dependent
type theories remain insufficiently understood.}
\end{quote}
According to him, constructions and theorems about type theories are currently assumed by \textbf{analogy}.
Instead, they should be proved by \textbf{specialization} of a general theorem.

Voevodsky aimed to build his theory on top of
the notion of C-system, introduced by Cartmell~\cite{DBLP:journals/apal/Cartmell86} under the name of \emph{contextual category}.
Voevodsky calls a C-system equipped with extra operations corresponding to the inference rules of a type theory a \textbf{C-system model}---or just \textbf{model}---of type theory.
To give semantics of type theory, Voevodsky aimed to build two C-system models:
\begin{enumerate*}[label=(\roman*)]
\item one from the formulas and derivations of some type theory, and
\item one from a category of abstract mathematical objects.
\end{enumerate*}
Furthermore, one should construct an \textbf{interpretation} (a functor) from the first to the second.

Such an interpretation typically needs to be constructed by \emph{recursion} over the derivations of the type theory.
As explained in \cref{sec:initial-semantics}, the recursive pattern can be encapsulated in an initiality result; the methodology of initial semantics thus suggests the following approach:
\begin{enumerate}
\item Show that the term model is initial in a suitable category.
\item Then, any model yields automatically a (unique) interpretation from the term model.
\end{enumerate}

Now, for the construction of the two desired models, syntactic and semantic, respectively, Voevodsky developed different methodologies.
For the construction of \emph{semantic} models, Voevodsky exhibited several constructions of C-systems from universe categories~\cite{MR3402489}.
He also sketched a strictification from categories with families to C-systems.
For the construction of \emph{syntactic} (or \emph{term}) models, Voevodsky developed a framework outlined across several papers.
We summarize the ingredients involved here:
\begin{enumerate}
\item Restricted 2-sorted binding signatures (cf.~\cite[Section~1]{Voevodsky_relative}) with sorts for terms and types are used as abstract specificiation devices for pretypes and preterms.
\item From a restricted 2-sorted binding signature, a ``term'' monad $R : \Set \to \Set$ and a ``type'' module $LM : \Set \to \Set$ over $R$ are constructed
  (cf.~\cite[Section~1]{Voevodsky_relative}).
\item Any monad $R$ on $\Set$ gives rise to a C-system $C(R)$, corresponding to the mono-typed (or untyped) syntax of $R$, cf.~\cite[Section~4.2]{Voevodsky_relative}.
\item The \emph{presheaf extension} of $C(R)$ by the module $LM$ over $R$, called $C(R)[LM]$, constitutes the C-system of pretypes and preterms---but without any typing relation yet, cf.~\cite[Section~4.2]{Voevodsky_relative}.
\item Finally, Voevodsky's theory of \emph{sub-C-systems and regular quotients of C-systems}~\cite{MR3475277} allows one to carve out C-systems of types and well-typed terms modulo a regular congruence relation.
\end{enumerate}

In the following, we discuss some of these ingredients in slightly more detail, but without any rigorous definitions.

A ``restricted 2-sorted binding signature'' is a signature that specifies a 2-sorted language.
We can think of these two sorts as a sort $\ty$ of ``types'' and a sort $\te$ of ``terms'', respectively.
The signatures are ``restricted'' in the sense that constructors can bind variables of sort $\te$ but not of sort $\ty$.

We do not dwell on the notion of signature, but refer instead to \cite[Section~1]{Voevodsky_relative} for details; here, we give an example of a language specified by such a signature.
\begin{exa}
  An example of a syntax generated by a 2-sorted binding signature is the syntax of the Calculus of Constructions, adapted from Streicher's \emph{Semantics of Type Theory}~\cite{streicher-semantics-of-tt}:
  
\begin{linenomath*}
  \[
    \begin{array}{lcll}
      A,B    & ::= & \Pi(A,x.B) & \hspace{1cm} \text{Product of types} \\
             & |  & Prop & \hspace{1cm} \text{Type of propositions} \\
             & |  & Proof(t) & \hspace{1cm} \text{Type of proofs of proposition }t \\[1ex]
      t,u    & ::= & x & \hspace{1cm} \text{Variable} \\
             & |  & \lambda(A,x.t) & \hspace{1cm} \text{Function abstraction} \\
             & |  & App(A,x.B,t,u) & \hspace{1cm} \text{Function application} \\
             & |  & \forall(A,x.t) & \hspace{1cm} \text{Universal quant.\ over propositions }t
    \end{array}
  \]
\end{linenomath*}
\end{exa}

\noindent
This signature specifies a language with two sorts, the sort $\ty$ of ``types'' and the sort $\te$ of ``terms''.
It is restricted because there is no binding of variables of sort $\ty$, only of variables of sort $\te$.
Such a signature yields a monad $T : \Set\times\Set \to \Set\times\Set$,
\begin{linenomath*}
  \[   
    (X,Y) \mapsto (\ty(X,Y),\te(X,Y))
  \]
\end{linenomath*}
where $\ty(X,Y)$ is the set of expressions of sort $\ty$ with variables of kind $\ty$ in $X$  and of kind $\te$ in $Y$
and similarly for $\te(X,Y)$.
From such a monad on $\Set\times\Set$, Voevodsky~\cite{Voevodsky_relative} constructs, by fixing a set of type variables, a monad $R=\te$ on $\Set$, and a module $LM = \ty$ over $R$.
Here, the action of the module $LM$ is substitution of term expressions in type expressions.
From $R$ and $LM$, in turn, Voevodsky~\cite{Voevodsky_relative} constructs two C-systems, called $C(R)$ and $C(R)[LM]$, respectively.
The C-system $C(R)$ corresponds to a mono-typed syntax of just terms---in detail:
\begin{enumerate}
\item Objects are natural numbers (untyped contexts).
\item Morphisms $m \to n$ are maps $\ff{n} \to R(\ff{m})$, where $\ff{k}$ is the standard finite set associated to $k \in \N$.
\item The category thus obtained is the opposite of the Kleisli category on $R$ restricted to natural numbers.%
  \footnote{Put differently, it is the Kleisli category of the $Jf$-relative monad induced by the monad $R$, as indicated by the title of Voevodsky's article~\cite{Voevodsky_relative}.}
\item The morphism $\p_n : n+1 \to n$ is given by the composition $\ff{n} \xrightarrow{\iota} \ff{n+1} \xrightarrow{\eta} R(\ff{n+1})$.
\item Given a morphism $f : m \to n$, that is, a function $f : \ff{n} \to R(\ff{m})$, the pullback of $\p_n : n+1 \to n$ along $f$ is the morphism $\p_m : m+1 \to m$.
  The morphism $m+1 \to n+1$ required to complete the pullback square is the morphism $q(f) : n+1 \to R(\ff{m+1})$ induced by the morphisms $n \xrightarrow{f} R(\ff{m}) \xrightarrow{R(\iota)} R(\ff{m+1})$ and $1 \xrightarrow{\iota} \ff{m+1} \xrightarrow{\eta} R(\ff{m+1})$;
  intuitively, $q(f)$ extends the substitution $f$ by one variable. See also \cite[Lemma~4.2.2]{Voevodsky_relative}.
\end{enumerate}
The C-system $C(R)[LM]$, in turn, looks as follows:
\begin{enumerate}
\item $C(R)[LM]$ has, as contexts, finite sequences of types (with a suitable number of free variables).
\item Pullback is given by substitution of terms in type expressions.
\item There is no typing relationship yet: $C(R)[LM]$ is a C-system of pretypes and preterms.
\end{enumerate}
In order to build, from $C(R)[LM]$, a C-system of types and well-formed terms, with the intended typing relation, Voevodsky devised
\begin{enumerate*}[label=(\roman*)]
\item sub-C-systems (for eliminating ill-formed pretypes and preterms),
\item quotients of C-systems (for considering terms and types modulo judgemental equality).
\end{enumerate*}
To construct such subsystems and quotients, Voevodsky devised the theory of B-systems.

\subsubsection{B-systems for the construction of C-systems}
\label{sec:b-syst-constr}

Intuitively, the idea is to use the C-system $C(R)[LM]$ to obtain the pretypes and preterms to formulate \textbf{judgements}:
\begin{itemize}
\item A statement $\Gamma \vdash$ is an element of
  \begin{linenomath*}
  \begin{equation}\label{eq:B-sets-of-module}
    B(R,LM) \defeq \coprod_{n\ge 0} \prod_{i=0}^{n-1}LM(\ff{i})
  \end{equation}
\end{linenomath*}
\item A statement $\Gamma \vdash t : T$ is an element of
\begin{linenomath*}
  \begin{equation}\label{eq:Btilde-sets-of-module}
    \wt{B}(R,LM) \defeq \coprod_{n\ge 0} \left(\prod_{i=0}^{n-1}LM(\ff{i})\times R(\ff{n})\times LM(\ff{n})\right)
  \end{equation}
\end{linenomath*}
\end{itemize}
Voevodsky~\cite{VV_B-systems} defines eight operations on $B$ and $\wt{B}$, corresponding to structural rules of type theory.
The resulting mathematical structure is captured by the notion of \textbf{B-system}, illustrated in more detail in \Cref{sec:models-of-tt} and studied in detail in \Cref{sec:b-sys}.

Given a C-system $C$, we call $B(C)$ and $\wt{B}(C)$ the B-sets associated to $C$.
Voevodsky~\cite{MR3475277} constructed a bijection between
\begin{enumerate}
\item Sub-C-systems of a given C-system
\item Subsets of $(B,\wt{B}(C))$ that are closed under the eight operations
\end{enumerate}
and similar, but more complicated, for quotients.
This bijection is used by Voevodsky to construct suitable C-systems; Voevodsky himself~\cite{VV_B-systems} positions B-systems as follows:
 \begin{quote}
  B-systems are algebras (models) of an essentially algebraic theory that is expected to be
constructively equivalent to the essentially algebraic theory of C-systems which is, in turn,
constructively equivalent to the theory of contextual categories. The theory of B-systems is
closer in its form to the structures directly modeled by contexts and typing judgements of
(dependent) type theories and further away from categories than contextual categories and
C-systems.
 \end{quote}

\noindent
This concludes our overview of the use of B-systems in Voevodsky's research program.
In the remainder of the introduction, we provide more intuition for the notions of B-system and C-system, before giving rigorous definitions and constructions.

\subsection{Models of Type Theory}\label{sec:models-of-tt}

When studying type theories mathematically, one question to answer is: what is the appropriate mathematical structure that captures the essential behaviour of type theories? Technically speaking: what are the objects in the category of models of a type theory?

Many different answers have been given to this question.
The purpose of this \lcnamecref{sec:models-of-tt} is to present the two contenders studied and compared in this work, and to relate them to other notions of ``model''.

\subsubsection{Contextual categories and C-systems}

Contextual categories were defined, by Cartmell~\cite[\S14]{DBLP:journals/apal/Cartmell86}, as a mathematical structure for the interpretation of generalized algebraic theories and of the judgements of Martin-Löf type theory.
A contextual category comes with a tree structure, in particular, a partial ordering, on its objects; think of the objects of $\C$ as ``contexts'', and $\Gamma \leq \Delta$ stating that $\Gamma$ can be obtained from $\Delta$ by truncation.
Furthermore, there is a special class of morphisms, closed under pullback along arbitrary morphisms---thought of as substitution by that morphism.
In his PhD dissertation~\cite[Section~2.4]{cartmell_phd}, Cartmell shows that the category of contextual categories and homomorphisms between them is equivalent to the category of generalized algebraic theories and (equivalence classes of) interpretations between them.

Voevodsky defined C-systems as equivalent to contextual categories:
a C-system is a category coming, in particular, with a length function and a compatible ``father'' function on objects of the category, signifying truncation of contexts.
Again, we have a class of morphisms closed under pullback along arbitrary morphisms.
Voevodsky rejected the name ``contextual category'' for these mathematical object, for the reason that the extra structure on top of the underlying category cannot be transported along equivalence of categories and is thus not ``categorical'' in nature.
As an example, consider the terminal category: it can be equipped with exactly one C-system structure.
However, there is no C-system structure on any category with more than one, but finitely many, objects.

More recently, Cartmell~\cite{Cartmell_GAAxiom} gave two Generalized Algebraic axiomatizations of contextual categories, one of which is using Voevodsky's $s$-operator~\cite[Definition~2.3]{MR3475277} for pullbacks.

\subsubsection{B-systems}

Voevodsky's definition of B-systems \cite{VV_B-systems} is inspired by the presentation of type theories in terms of \emph{inference rules}.
Specifically, type theories ``of Martin-Löf genus'' are given by sets of five kinds of judgements:

\begin{description}
\item[Well-formed context]
\begin{linenomath*}  \[   \Gamma \vdash \] \end{linenomath*}
  
\item[Well-formed type in some context]
\begin{linenomath*}  \[  \Gamma \vdash A~\type \]\end{linenomath*}
  
\item[Well-formed term of some type in some context]
\begin{linenomath*}  \[ \Gamma \vdash a : A \]\end{linenomath*}
  
\item[Equality of types]
\begin{linenomath*}  \[   \Gamma \vdash A \converts B \]\end{linenomath*}
  
\item[Equality of terms]
\begin{linenomath*}  \[   \Gamma \vdash a \converts b : A \] \end{linenomath*}
\end{description}

Interpreting equality of types and terms as actual equality, and expressing $\Gamma \vdash A$ instead as $\Gamma, A \vdash$, lead Voevodsky to defining a B-system to consist of families of sets $(B_n)_{n\in \N}$ and $(\term{B}_n)_{n \in \N_{> 0}}$, intuitively denoting, for any $n \in \N$, contexts of length $n$ and terms in a context of length $n-1$, together with their types, respectively.
Furthermore, any B-system has various operations on $B$ and $\term{B}$, such as maps $\bd[n] : \term{B}_{n+1} \to B_{n+1}$ specifying, intuitively, for each ``term'' $t \in \term{B}_{n+1}$, the context $\bd[n](t) \in B_{n+1}$ in which $t$ lives.

Voevodsky's B-systems are very similar to the algebras of the theory \emph{MetaGAT} defined by  John Cartmell~\cite{Cartmell_metaGATs}, and to the algebras of a monad studied by Richard Garner~\cite{GARNER20151885}.
The intention is that these are all equivalent notions of structure.
Below, we will indicate more precise connections to Garner's work.

Compared to other semantics for dependent type theories, B-systems appear the closest to syntax.
For this reason it seems easier to describe extensions of the structural rules by type constructors or modal operators for B-systems than it is for, say, C-systems (\ie~contextual categories).
For the same reason, B-systems seem also more suitable than other semantics to describe notions of substructural dependent type theories and, more generally, variations on the syntax.

\subsubsection{Other Notions of Model}

There are many other mathematical structures for the interpretation of type theory.
Here, we give some pointers to related literature.

Voevodsky sketched a relation between C-systems and categories with families in his Lectures in the Max Planck Institute in Bonn \cite[Lecture~5]{vv-bonn}, identifying C-systems as categories with families with a particular property.
In the present work, we introduce and study unstratified categorical structures, in the form of E-systems and CE-systems, which we anticipate will be useful in giving a precise construction for Voevodsky's conjecture.

Categories with families, in turn, are related to categories with attributes (a.k.a.\ split type categories) in \cite{blanco} (in a categorical setting) and in \cite{DBLP:journals/lmcs/AhrensLV18} (in the univalent setting).
Composing these characterizations with the equivalence presented here provides a comparison between B-systems and other mathematical structures for type theory.

Garner~\cite{GARNER20151885} studies and compares two structures related to Voevodsky's B-systems: Generalized Algebraic Theories (GATs) and algebras for a monad on the category of type-and-term structures (see also \cref{eg:b-frames-equiv-type-term,eg:b-frames-structure}).

\begin{rem}
  Garner's and Cartmell's works, taken together, also point to another possible way to constructing an equivalence between C-systems and B-systems:
Cartmell~\cite[Section~2.4]{cartmell_phd} constructs an equivalence of categories between the category of contextual categories and homomorphisms between them, and the category of GATs and (equivalence classes of) interpretations between them.
Garner~\cite{GARNER20151885} constructs an equivalence of categories between the category of B-frames and the category of $\emptyset$-GATs (GATs without structural rules) (see also~\cref{eg:b-frames-equiv-type-term}).
Garner's equivalence looks like it could be ``upgraded'' to an equivalence between the category of B-systems and the category of GATs (see also~\cref{eg:b-frames-structure}).
Constructing an equivalence between B- and C-systems in this way is, however, conceptually circular --- if not in actuality, then at least in spirit.
After all, B- and C-systems were studied by Voevodsky in the context of the Initiality Conjecture; one purpose of the Initiality Conjecture is to give a \textbf{specification of dependently-typed syntax}. We believe that such a specification is best given without recourse to such syntax itself.
\end{rem}

\subsection{About the present work}

The main result of this paper is the construction of an equivalence of categories,
between the category of C-systems and the category of B-systems.
The existence of such an equivalence was conjectured by Voevodsky.

We construct this equivalence as a restriction of an equivalence between more general, unstratified structures introduced in this paper, called CE-systems and E-systems, respectively.
While it is not \emph{necessary} to pass via E-systems and CE-systems to construct an equivalence between B-systems and C-systems, it seems \emph{desirable} to us for two reasons:
\begin{enumerate}
\item The definitions and constructions are automatically more modular, isolating structure on either side that corresponds to each other.
\item The study of unstratified structures is useful in connecting B-systems and C-systems to other unstratified structures, such as categories with families~\cite{Dybjer1996}.
  Work on constructing a suitable comparison is already underway.
\end{enumerate}

\noindent
This paper is organized as follows.
In~\cref{sec:preliminaries} we discuss some prerequisites that we build upon in later sections.
In~\cref{sec:c-sys} we review the definition of C-systems given by Voevodsky in~\cite[Def.~2.1]{MR3475277}, itself an equivalent formulation of Cartmell's definition of contextual categories~\cite[\S14]{DBLP:journals/apal/Cartmell86}.
Here we also introduce CE-systems.
In~\cref{sec:b-sys}, we give Voevodsky's definition of B-system~\cite{VV_B-systems} and introduce E-systems.
In~\cref{sec:equiv-b-c} we construct our equivalence of categories between B-systems and C-systems.

\subsubsection{Foundations}
\label{sec:foundations}

The work described in this result can be read to take place in intuitionistic set theory (IZF) or extensional type theory, i.e., a type theory with equality reflection.
In particular, we do not make use of classical reasoning principles such as an axiom of choice or excluded middle.
We consider in this work categories built from algebraic structures (which sometimes are themselves categories with structure, but see \cref{sec:categories}).
Implicitly, we take these algebraic structures to be built from sets (or types) from a universe $\U_1$. The categories of such structures are hence categories built from sets (or types) of objects and morphisms of a universe $\U_2$.
In the following, we leave the universe levels implicit.

\subsubsection{About our use of categories}
\label{sec:categories}

In this work, categories are used on two different levels.

Firstly, we use categories as algebraic structures, as the basis for C-systems and CE-systems.
This use of categories is somewhat ``accidental'', and our constructions on these categories are not invariant under equivalence of categories.
In particular, we liberally reason about equality of objects in such categories.
Consequently, we avoid the unadorned word ``category'' for these structures, and call them \emph{\scats} instead.
We denote by $\Cat$ the category of \scats and functors between them.

Secondly, we use categories to compare different mathematical structures to each other, by considering a suitable category of such structures and their homomorphisms.
Here, we never consider equality, but only isomorphism, of such mathematical structures;
our reasoning on that level is entirely categorical.
We reserve the word ``category'' for such uses of the concept.

We use different fonts for \scats and categories, respectively:
calligraphic font, such as $\cat{C}$, is used for \scats;
boldface, such as $\grph$, is used for categories.
We use the same notation for arrows in strict categories and in categories.
We write either $g \comp f$ or $gf$ for the composition of $f \colon a \to b$ and $g \colon b \to c$.

\subsection{Version history}

We first reported on the construction of an equivalence of categories between B-systems and C-systems in a short summary paper \cite{b-c-systems}.
The present work expands on that previous work in the following ways:
\begin{enumerate}
\item An expanded introduction summarizes the role of B-systems and C-systems in Voevodsky's unfinished theory of type theories.
\item We consider here more general ``unstratified'' variants of B-systems and C-systems, called E-systems and CE-systems, respectively. We construct an equivalence of categories between E-systems and CE-systems. 
\item We then construct the equivalence between B-systems and C-systems as a restriction, to the respective subcategories of stratified objects, of the aforementioned equivalence.
\item We give full details of all the constructions in this paper.
\end{enumerate}

\subsection{Acknowledgements}

We thank Steve Awodey for feedback on a draft of this paper and the anonymous referees for helpful remarks and suggestions.
The research described in this paper was presented at the Seminar on Contextual Categories in Ljubljana and online in May 2021, at the TYPES conference in Leiden and online in June 2021, and at the first meeting of WG6 of the COST action CA20111 ``EuroProofNet'' in Stockholm in May 2022.
We thank the organisers and participants of the three events for valuable discussions.

\section{Preliminaries: stratification of categories}
\label{sec:preliminaries}

In this section we collect definitions and results related to \textbf{stratification of \scats}
and morphisms between them.
A stratification (see~\cref{def:stratification}) associates, to any object of a \scat a natural number, its ``length'',
and to any length-decreasing morphism a factorization of this morphism into morphisms ``of length 1''.
Such a stratification can equivalently be described as a rooted tree, see~\cref{sec:rooted-trees}.

\subsection{Stratification of \scats}

\begin{defi}[Stratified strict categories, stratified functors]
  \label{def:stratification}
  Let $\cat{C}$ be a \scat with terminal object 1.
  A \define{stratification} for $\cat{C}$ consists of a
  \emph{stratification functor}
  \begin{linenomath*}
  \begin{equation*}
  L : \cat{C}\to (\mathbb{N},\geq)
  \end{equation*}
  \end{linenomath*}
  such that
  \begin{enumerate}
  \item\label{def:stratification:term}
  $L(X)\jdeq 0$ if and only if $X$ is the chosen terminal object 1,
  \item\label{def:stratification:setlev}
  for any $f:X\to Y$ we have
  $L(X)\jdeq L(Y)$ if and only if $X\jdeq Y$ and $f\jdeq\catid{X}$, and
  \item\label{def:stratification:fact}
  every morphism $f:X\to Y$ in $\cat{C}$, where $L(X)\jdeq
  n+m+1$ and $L(Y)\jdeq n$, has a unique factorization 
 \begin{linenomath*}
  \begin{equation*}
  \begin{tikzcd}
  X \jdeq X_{m+1} \arrow[r,"f_m"] & X_m \arrow[r,"f_{m-1}"] & \cdots \arrow[r,"f_1"] & X_1 \arrow[r,"f_0"] & X_0 \jdeq Y
  \end{tikzcd}
\end{equation*}
\end{linenomath*}
  where $L(X_i)\jdeq n+i$.
  \end{enumerate}
  
  A functor $F:\cat{C}\to\cat{D}$ between \scats with stratifications $L_\C$ and $L_\D$, respectively, is said to be \define{stratified} 
  if $L_{\cat{C}}\jdeq L_{\cat{D}}\circ F$.
\end{defi}

\begin{rem}
  We emphasize that stratifications do not transport along equivalence of categories.
  For instance, there is a (necessarily unique, see~\cref{prop:strat-is-prop}) stratification on the terminal strict category, but no stratification on the strict chaotic category with two objects and a choice of terminal object.
\end{rem}

\begin{rem}
  Those readers familiar with Conduché functors might note that a stratified category $(\C, L)$ is equivalently a Conduché functor $L : \C \to \mathbb N$ with discrete fibers which takes the chosen terminal object of $\C$ to $0$.
\end{rem}

\begin{defi}\label{def:individual-ar}
Let $\C$ be a category with a terminal object and $\ell \colon \ob{C} \to \mathbb{N}$ a function.
An arrow $f \colon X \to Y$ in $\C$ is \define{\indecarr}
if $\ell(X) \jdeq \ell(Y) + 1$.
\end{defi}

\begin{rems}\label{rmk:stratified}\hfill
\begin{enumerate}
\item\label{rmk:stratified:uqterm}
In a stratified category, there is a unique terminal object 1.
More generally, if there is an arrow $1 \to X$, then $X \jdeq 1$.
 \item\label{rmk:stratified:fact}
 The factorisation of an arrow $f \colon X \to Y$ such that $L(X)-L(Y) \jdeq m+1 > 0$
 in~\ref{def:stratification}.\ref{def:stratification:fact}
 consists of $m+1$ \indecarr arrows.
\item\label{rmk:stratified:funct}
A stratified functor is determined by its action on \indecarr arrows.
\end{enumerate}
\end{rems}

\begin{lem}\label{lem:strat-alt}
Let $\C$ be a category with a terminal object $1$.
A function $\ell \colon \ob{\C} \to \mathbb{N}$ extends to a stratification
$L \colon \C \to (\mathbb{N},\geq)$ of $\C$ if and only if the following three conditions hold:
\begin{enumerate}[label=(\textit{\roman*}),ref=\textit{\roman*}]
\item\label{lem:strat-alt:term}
$\ell(1) \jdeq 0$,
\item\label{lem:strat-alt:down}
for every object $X$ and $k \leq \ell(X)$, the set
\begin{linenomath*} \[\coprod_{Y\mathop{\mid}\ell(Y)\jdeq k}\C(X,Y)\] \end{linenomath*}
is a singleton, \ie there is a unique arrow $x_k \colon X \to X_k$
such that $\ell(X_k) \jdeq k$, and
\item\label{lem:strat-alt:up}
for every $X$ and $k > \ell(X)$, the set
\begin{linenomath*}\[\coprod_{Y \mathop{\mid} \ell(Y)\jdeq k}\C(X,Y)\]\end{linenomath*}
is empty, \ie there are no arrows $X \to Y$ such that $\ell(X) < \ell(Y)$.
\end{enumerate}
\end{lem}

\begin{proof}
If $\C$ is stratified by $L$ such that $L(X) \jdeq \ell(X)$,
condition~\eqref{lem:strat-alt:term} follows from~\ref{def:stratification}.\ref{def:stratification:term}.
To show condition~\eqref{lem:strat-alt:down}, note that every arrow $X \to 1$ factors uniquely into $l \jdeq \ell(X)$ \indecarr arrows
\begin{linenomath*}
\[\begin{tikzcd}[column sep=3em]
X	\ar[r,"x_l"]	&	X_{l-1}	\ar[r,"x_{l-1}"]	&	\cdots	\ar[r,"x_2"]	&	X_1	\ar[r,"x_1"]	&	1.
\end{tikzcd}\]
\end{linenomath*}
In particular, for every $n \leq \ell(X)$, the composite $x_{n+1}\cdots x_l \colon X \to X_n$ is such that $\ell(X_n) \jdeq n$.
If $f \colon X \to Y$ is also such that $\ell(Y) \jdeq n$, then $f$ factors into $l-n$ \indecarr arrows $(f_i)_{i\jdeq1}^{l-n}$.
Let $y_1\cdots y_n$ be the factorisation of $Y \to 1$ into \indecarr arrows.
The composite $y_1\cdots y_n f_0 \cdots f_{l-n-1}$ is a factorisation of $X \to 1$ into \indecarr arrows.
It follows by uniqueness of such factorisations that
\begin{linenomath*}
\[
x_1 \jdeq y_1,\ x_2 \jdeq y_2,\ \dots\ ,\ x_n \jdeq y_n,\
x_{n+1} \jdeq f_1,\ \dots\ ,\ x_l \jdeq f_{l-n}.
\]
\end{linenomath*}
In particular, $f \jdeq x_{n+1}\cdots x_l$ as required.

Since $(\mathbb{N},\geq)$ is a poset, condition \eqref{lem:strat-alt:up} is equivalent to the fact that
the function $\ell$ extends uniquely to a functor $L \colon \C \to (\mathbb{N},\geq)$.

Suppose now that conditions~(\ref{lem:strat-alt:term}--\ref{lem:strat-alt:up}) above hold.
In particular, the function $\ell$ extends to a functor $L$.

\eqref{def:stratification:term}
Let $X$ be such that $\ell(X)\jdeq0$.
Then $X \to 1$ and $\catid{X}\colon X \to X$ are both such that $\ell(X)\jdeq0\jdeq\ell(1)$.
Hence $X \jdeq 1$ and the object $X$ is terminal.
Conversely, let $X$ be terminal. Then there is $1 \to X$ and thus $0 \geq \ell(X)$.

\eqref{def:stratification:setlev}
Let $f  \colon X \to Y$ and suppose $\ell(X)\jdeq \ell(Y)$,
then $Y \jdeq X$ and $f\jdeq\catid{X}$ by \eqref{lem:strat-alt:down} with $n\jdeq\ell(X)$.

\eqref{def:stratification:fact}
For every $X$ such that $n+1 \jdeq \ell(X) > 0$, let $\overline{X} \colon X \to X'$
be the unique arrow such that $\ell(X') \jdeq n$ given by~\eqref{lem:strat-alt:down}.
For every $k \leq \ell(X)$, we have a composite $x_k$ of $k$ \indecarr arrows
\begin{linenomath*}
\begin{equation}\label{lem:strat-alt:fact}
\begin{tikzcd}[column sep=4em]
X	\ar[r,"\overline{X}"] \ar[rrrr,bend right=1.5em,"x_k"]	&	X'	\ar[r,"\overline{X'}"]
&	\cdots	\ar[r,"\overline{X^{(k-2)}}"]	&	X^{(k-1)}	\ar[r,"\overline{X^{(k-1)}}"]	&	X^{(k)}
\end{tikzcd}
\end{equation}
\end{linenomath*}
where $\ell(X^{(k)}) \jdeq \ell(X)-k$,
which is the unique arrow $X \to Y$ such that $\ell(Y) \jdeq \ell(X)-k$ by~\eqref{lem:strat-alt:down}.
Let us show that \eqref{lem:strat-alt:fact} is also
the unique factorisation of $x_k$ into \indecarr arrows,
for every $0 < k \leq \ell(X)$.
We proceed by induction on $n$.
If $n\jdeq0$, then factorisations consist of only one \indecarr arrow
and uniqueness follow from~\eqref{lem:strat-alt:down}.
For $n>0$, let $0 < k \leq n+1$ and consider a factorisation
$X \xrightarrow{g_0} Z_1 \xrightarrow{g_1} \cdots \xrightarrow{g_{k-2}} Z_{k-1} \xrightarrow{g_{k-1}} X^{(k)}$
of $x_k$ into \indecarr arrows.
Then $\ell(Z_1) \jdeq \ell(X)-1 \jdeq \ell(X')$, and so $g_0 \jdeq \overline{X}$ by~\eqref{lem:strat-alt:down}.
Again, $g_{k-1}\cdots g_1 \jdeq x'_{k-1} \colon X' \to (X')^{(k-1)}$ by~\eqref{lem:strat-alt:down} and,
by inductive hypothesis, $g_i \jdeq \overline{X^{(i)}}$ for $0<i<k$.
Therefore~\eqref{lem:strat-alt:fact} is the unique factorisation of $x_k$
into \indecarr arrows.

Given an arrow $f \colon X \to Y$ such that $n \jdeq \ell(Y) < \ell(X) \jdeq m+n+1$,
it must be $Y \jdeq X^{(m+1)}$ and $f \jdeq x_{m+1}$ by~\eqref{lem:strat-alt:down}.
It follows that $f$ factors uniquely into $m+1$ \indecarr arrows
$\overline{X^{(m)}}\cdots\overline{X'}\,\overline{X}$.
\end{proof}

\begin{rem}
Condition~\eqref{lem:strat-alt:down} in \cref{lem:strat-alt} is equivalent to requiring that,
for every object $X$:
\begin{enumerate}[label=(\textit{\ref*{lem:strat-alt:down}.\alph*}),%
ref=\textit{\ref*{lem:strat-alt:down}.\alph*}]
\item
for every $n \leq \ell(X)$ there is at most one arrow $f \colon X \to Y$ such that $\ell(Y) \jdeq n$, and
\item
there is an \indecarr arrow $\overline{X} \colon X \to X'$.
\end{enumerate}
One direction is clear. For the converse it is enough to show that
for every $n < \ell(X)$ there is $f \colon X \to Y$ such that $\ell(Y) \jdeq n$.
Such an arrow is given as the composite of $\ell(X)-n$ \indecarr arrows
as in~\cref{lem:strat-alt:fact} above.
\end{rem}

\begin{prop}\label{prop:strat-is-prop}
Any category can be stratified in at most one way. 
\end{prop}
\begin{proof}
Consider a category $\C$ with two stratifications $L,M: \C \to \mathbb N$.
By \cref*{lem:strat-alt}.\ref{lem:strat-alt:down},
\begin{linenomath*}
\[
L^{-1}(n+1) \jdeq \left\{ X \mid \exists\ Y \in L^{-1}(n) \text{ and an \indecarr } f: X \to Y \right\}
\]
\end{linenomath*}
and similarly for $M^{-1}(n+1)$.
Thus, if $L^{-1} (n) \jdeq M^{-1} (n)$, we find that $L^{-1} (n+1) \jdeq M^{-1} (n+1)$,
and the claim follows by induction
since $L^{-1} (0) \jdeq M^{-1} (0)$ by~\cref{def:stratification}.\ref{def:stratification:term}.
\end{proof}

Uniqueness of stratification justifies the following definition:

\begin{defi}\label{def:stratcat}
We define $\strCat$ to be the subcategory of $\Cat$ consisting of stratified \scats and stratified functors between them.
\end{defi}  

\begin{lem}\label{lem:strat-funct}
Let $F \colon \cat{C} \to \cat{D}$ be a functor between stratified categories.
The following are equivalent.
\begin{enumerate}
\item\label{lem:strat-funct:a}
The functor $F$ is stratified.
\item\label{lem:strat-funct:b}
The functor $F$ preserves terminal objects and \indecarr arrows.
\end{enumerate}
\end{lem}

\begin{proof}
That \ref*{lem:strat-funct:a} implies
\ref*{lem:strat-funct:b} is clear.
The converse is by induction on the length of objects using that,
for every $f \colon X \to Y$,
$L_{\C}(X) \jdeq L_{\C}(Y) + 1$ implies $L_{\cat{D}}(F(X)) \jdeq L_{\cat{D}}(F(Y)) +1$.
\end{proof}

\begin{lem}\label{lem:strat-slice}
Let $\C$ be a stratified category with stratification functor $L$.
Then for every object $X$ and every $f \colon Y \to X$,
\begin{linenomath*}
\[
L_X(f) :\jdeq L(Y) - L(X)
\]
\end{linenomath*}
defines a stratification functor $L_X$ for the slice $\C/X$.
\end{lem}

\begin{proof}
The above clearly defines a functor $L_X \colon \C/X \to (\mathbb{N},\geq)$ and
conditions (\ref*{def:stratification:term}--\ref*{def:stratification:fact}) in \cref{def:stratification} are easily verified.
\end{proof}

\begin{cor}
Let $F \colon \cat{C} \to \cat{D}$ be a stratified functor
between stratified categories.
For every object $X$ in $\cat{C}$, the functor
\begin{linenomath*}
\[\begin{tikzcd}[column sep=3em]
\cat{C}/X	\ar[r,"F/X"]
&	\cat{D}/FX
\end{tikzcd}\]
\end{linenomath*}
is stratified.
\end{cor}

\subsection{Rooted trees}\label{sec:rooted-trees}

In this section, we compare stratified categories to rooted trees.
Rooted trees were used by Cartmell~\cite{DBLP:journals/apal/Cartmell86} to give his original definition of contextual categories.

\begin{defi}\label{def:rttr-to-cat}
\hfill
\begin{enumerate}
\item
We define a \define{rooted tree} $T$ to be a family of sets
$(T_n)_{n \in \mathbb{N}}$ indexed by the natural numbers
such that $T_0$ is a singleton, together with functions
$(t_n \colon T_{n+1} \to T_n)_{n \in \mathbb{N}}$
mapping a node to its parent.
A homomorphism of rooted trees $f \colon T \to S$ is a family of functions
$(f_n \colon T_n \to S_n)_{n\in\mathbb{N}}$
such that $f_n \circ t_n \jdeq s_n \circ f_{n+1}$
for every $n \in \mathbb{N}$.
Let $\rttr$ be the category of rooted trees and homomorphisms.
\item
Let $\grph$ be the category of directed (multi)graphs and homomorphisms.
We define the functor $\Gfctr : \rttr\to\grph$ as follows.
For a rooted tree $T$, the directed graph $\Gfctr(T)$
has the set of vertices given by the disjoint union $\coprod_{n\in\mathbb{N}}T_n$,
and there is an edge $(n+1,X) \to (n,t_n(X))$
for every $n \in \mathbb{N}$ and $X \in T_{n+1}$.
It is straightforward to verify that each homomorphism $T \to T'$ of rooted trees
gives rise to a homomorphism of graphs $\Gfctr(T) \to \Gfctr(T')$.
\item
Let $\Ffctr: \grph \to \Cat$ be the well-known functor~\cite[II.7]{MacLane98}
that takes a graph to the category freely generated by it.
\end{enumerate}
\end{defi}

We now show that
the image of the composite
\begin{linenomath*}
\[\begin{tikzcd}[column sep=4em]
\rttr	\ar[r,"\Gfctr"]	&	\grph	\ar[r,"\Ffctr"]	&	\Cat
\end{tikzcd}\]
\end{linenomath*}
is the subcategory of \emph{stratified \scats} defined in \cref{def:stratcat}.

\begin{prop}\label{lem:rttrStrCat}
The functor $\Ffctr\Gfctr \colon \rttr \to \Cat$ lifts to an equivalence $\rttr \overset{\simeq}{\longrightarrow} \strCat$.
\end{prop}

\begin{proof}
First, observe that, for a rooted tree $T$, the free category $\Ffctr\Gfctr(T)$ is stratified.
We define $L: \Ffctr\Gfctr(T) \to (\N,\geq)$ by sending an object $(n,X)$ to $n$ and a generating morphism $(n+1, X) \to (n, t_n(X))$ to $n+1 \geq n$.
Given a morphism $f: S \to T$ of rooted trees, the functor $\Ffctr\Gfctr(f) : \Ffctr\Gfctr(S) \to \Ffctr\Gfctr(T)$ is stratified by construction.
Thus, the functor $\Ffctr\Gfctr: \rttr \to \Cat$ lifts to an functor $\rttr \to \strCat$.

Next we define a functor $\Ifctr \colon \strCat \to \rttr$.
Consider a stratified category $(\C,L)$ and define a rooted tree  $\Ifctr(\C,L)$ as follows. Let $\Ifctr(\C,L)_n :\jdeq L^{-1}(n)$.
By \cref{lem:strat-alt}\eqref{lem:strat-alt:down}, for every $X \in \Ifctr(\C,L)_n$ there is
exactly one \indecarr arrow with domain $X$, say $X \to X'$.
Then we define $t_n : \Ifctr(\C,L)_{n+1} \to \Ifctr(\C,L)_n$ by $t_n(X) :\jdeq X'$.
A stratified functor $F \colon \C \to \D$ induces 
a homomorphism of rooted trees $\Ifctr(\C,L) \to \Ifctr(\C,M)$
since it commutes with the stratification functors
and it preserves \indecarr arrows.

It is now straightforward to verify that
$\Ffctr\Gfctr \circ \Ifctr \cong 1_{\mathbf{Cat_s}}$ and
$\Ifctr \circ \Ffctr\Gfctr \cong 1_{\rttr}$.
\end{proof}

\section{The category of C-systems}\label{sec:c-sys}

This section is dedicated to the study of C-systems.

In~\cref{ssec:c-sys} we review Voevodsky's definition of C-system, an equivalent formulation of Cartmell's contextual categories.
We then give, in~\cref{ssec:ce-sys} our definition of CE-system, and identify,
in~\cref{ssec:c2ce}, the category of C-systems as a subcategory of ``stratified'' objects in the category of CE-systems.

\subsection{The category of C-systems}
\label{ssec:c-sys}

John Cartmell~\cite[Section~14]{DBLP:journals/apal/Cartmell86} defined \emph{contextual categories} as mathematical structures for the interpretation of type theories.
Vladimir Voevodsky~\cite[Definition~2.1]{MR3475277} gave a slightly modified, but obviously equivalent definition, and coined them \emph{C-systems}.

\begin{defi}[C-system, {\cite[Def.~2.1]{MR3475277}}]
  \label{def:Csys}
  A \define{C-system} consists of
  \begin{enumerate}
  \item a \scat $\C$,
  \item a ``length'' function $\length \colon \ob{\C} \to \N$,
  \item a chosen object $\Cterminal \in \ob{\C}$,
  \item a function $\ft \colon \ob{\C} \to \ob{\C}$,
  \item for any object $\Gamma \in \ob{\C}$ such that $\length(\Gamma) > 0$, a morphism $\p_{\Gamma} \colon \Gamma \to \ft(\Gamma)$,
  \item for any $\Gamma \in \ob{\C}$ with $\length(\Gamma) > 0$ and any $f \colon \Delta \to \ft(\Gamma)$, an object $f^*\Gamma$ and a morphism $\q(f,\Gamma) \colon f^*\Gamma \to \Gamma$.
  \end{enumerate}
  satisfying the following axioms:
  \begin{enumerate}[label=\textit{\roman*}),ref=\textit{\roman*}]
  \item\label{def:Csys:lt}
  $\length^{-1}(0) \jdeq \{\Cterminal\}$,
  \item\label{def:Csys:lft}
  for $\Gamma$ with $\length(\Gamma) > 0$, we have $\length(\ft(\Gamma)) \jdeq \length(\Gamma)-1$,
  \item\label{def:Csys:ftt}
  $\ft(\Cterminal) = \Cterminal$,
  \item\label{def:Csys:t}
  $\Cterminal$ is a final object,
  \item\label{def:Csys:pb}
  for $\Gamma \in \ob{\C}$ with $\length(\Gamma) > 0$ and $f \colon \Delta \to \ft(\Gamma)$,
  one has $\length(f^*\Gamma) > 0$, $\ft(f^*\Gamma) = \Delta$, and the square
  \begin{linenomath*}
    \begin{equation}\label{Csys-pbsq}
      \begin{tikzcd}
        f^*\Gamma \ar[r, "{\q(f, \Gamma)}"]  \ar[d, "\p_{f^*(\Gamma)}"']
        &
        \Gamma \ar[d, "\p_\Gamma"]
        \\
        \Delta \ar[r, "f"]
        &
        \ft(\Gamma)
      \end{tikzcd}
    \end{equation}
  \end{linenomath*}
  commutes and is a pullback square,
  \item\label{def:Csys:pbid}
  for $\Gamma \in \ob{\C}$ with $\length(\Gamma)>0$, we have $\left(\catid{\ft(\Gamma)}\right)^*\Gamma = \Gamma$ and $\q(\catid{\ft(\Gamma)},\Gamma) = \catid{\Gamma}$, and
  \item\label{def:Csys:pbcmp}
  for $\Gamma \in \ob{\C}$ with $\length(\Gamma)>0$, $g \colon \Delta \to \ft(\Gamma)$ and $f \colon E \to \Delta$, we have $(g \comp f)^*\Gamma = f^* g^*\Gamma$ and $\q(g\comp f, \Gamma) = \q(g, \Gamma) \comp \q (f, g^*\Gamma)$.
  \end{enumerate}  
\end{defi}

\noindent
Intuitively, an object $\Gamma$ of the category underlying a C-system can be thought of as a context of length $\length(\Gamma)$.
Types in context $\Gamma$ are encoded by the projections $\p_{\Delta}$ with $\ft(\Delta) = \Gamma$
(hence, in particular, $\length(\Delta) = \length(\Gamma)+1$).
Terms are not explicitly given; a term of type $\p_{\Delta}$ (in context $\ft(\Delta)$) corresponds to a section to $\p_{\Delta}$.
This is exactly how terms are defined in the E-system constructed from a CE-system in \cref{constr:CE2E-fun}.

In case the reader wonders whether the length function $\length$ lifts to a stratification, in \cref{cor:freeproj} we show that it does so on a suitable subcategory of $\C$.

\begin{defi}\label{def:CsysHom}
A \define{morphism of C-systems} from \sys{C} to \sys{D} is a functor
$F \colon \C \to \D$ between the underlying categories that strictly preserves the rest of the structure, that is:
\begin{enumerate}[label=\textit{\roman*}),ref=\textit{\roman*}]
\item\label{def:CsysHom:t}
$F (\Cterminal_{\sys{C}}) = \Cterminal_{\sys{D}}$,
\item\label{def:CsysHom:l}
$\length_{\sys D} \circ \ob{F} = \length_{\sys C} \colon \ob{\C} \to \mathbb{N}$,
\item\label{def:CsysHom:ft}
$\ob{F} \circ \ft[\sys C] = \ft[\sys D] \circ \ob{F} \colon \ob{\C} \to \ob{\D}$,
\item\label{def:CsysHom:p}
$F \p_{\Gamma} = \p_{F\Gamma}$, for every $\Gamma \in \ob{\C}$,
\item\label{def:CsysHom:pb}
$F(f^*\Gamma) = (Ff)^*(F\Gamma)$ and $F(\q(f,\Gamma)) = \q(Ff,F\Gamma)$,
for every $\Gamma \in \ob{\C}$ such that $\length_{\sys C}(\Gamma) > 0$ and $f \colon \Delta \to \ft(\Gamma)$.
\end{enumerate}
\end{defi}

\begin{exa}[C-systems and Lavwere theories]
  Fiore and Voevodsky~\cite{1512.08104} construct an isomorphism of categories between the category of Lawvere theories and the category of \emph{$\length$-bijective} C-systems, that is, of C-systems whose length function is a bijection. Intuitively, such a C-system can be seen as modelling an untyped (or single-sorted) language.
\end{exa}

\begin{exa}[C-systems and contextual categories]
  C-systems are equivalent to Cartmell's contextual categories.
  In his Ph.D.\ dissertation, Cartmell~\cite[Section~2.4]{cartmell_phd} constructs an equivalence between the category of contextual categories and homomorphisms between them and the category of Generalized Algebraic Theories (GATs) and (equivalence classes of) interpretations between them.
  Hence C-systems are equivalent to GATs.
\end{exa}

\begin{exa}[C-system from a universe category]
  Any universe category gives rise to a C-system, via a construction by Voevodsky~\cite[Construction~2.12]{MR3402489}.
  A universe category is a category with a chosen terminal object and a \emph{universe}, that is, a morphism $p : \tU \to \U$ together with a choice of pullback of $p$ along any morphism $X \to \U$.
  Roughly, the C-system constructed from a universe category has, as objects of length $n$, sequences of $n$ morphisms $f_1,\ldots,f_n$ into $\U$ such that the domain of $f_{i+1}$ is the chosen pullback of $p$ along $f_i$. Such a sequence can be thought of as a sequence of (dependent) types $(A_1, A_2,\ldots,A_n)$ such that $A_1,\ldots A_i\vdash A_{i+1}$.
  Furthermore, any small C-system can be obtained via this construction;
  given a C-system $\sys{C}$, a universe can be constructed \cite[Construction~5.2]{MR3402489} on the presheaf category $\hat{\sys{C}}$ such that the C-system obtained from that universe is isomorphic to the C-system $\sys{C}$.
  For a brief overview of these constructions, see~\cite[Section~1.3]{kapulkin2012univalence}.
  
  Voevodsky's simplicial model of univalent foundations~\cite{kapulkin2012univalence} is built on top of a C-system obtained from a universe in the category of simplicial sets.
\end{exa}

\begin{prob}\label{prob:csys2rttr}
To construct a functor $\Ctorttr \colon \Csys \to \rttr$.
\end{prob}

\begin{construction}{prob:csys2rttr}\label{constr:csys2rttr}
Let $\sys{C} = (\C,\Cterminal,\length,\ft,\p,\dots)$ be a C-system.
The objects of $\C$ can be arranged into a rooted tree by defining
\[
T_n :\jdeq \left\{ \Gamma \mid \length(\Gamma)\jdeq n \right\}
\qquad\text{and}\qquad
t_n(\Gamma) :\jdeq \ft(\Gamma) \in T_n, \text{ for } \Gamma \in T_{n+1}.
\]
That is, the front square in the diagram of sets and functions
\begin{linenomath*}
\[\begin{tikzcd}
T_{n+1}	\ar[dd] \ar[rr,hook] \ar[dr,dashed,"t_n"{swap}]	&&	\ob{\C}	\ar[dd,"\length"{near end}] \ar[dr,"\ft"] 
&\\
&	T_n  \ar[rr,hook,crossing over]	&&	\ob{\C}	\ar[dd,"\length"]
\\
1	\ar[rr,"n+1"{near end}] \ar[dr]	&&	\mathbb{N}	\ar[dr,"\mathrm{pred}"{description}]
&\\
&	1 \ar[from=uu, crossing over]	\ar[rr,"n"]	&&	\mathbb{N}
\end{tikzcd}\]
\end{linenomath*}
is a pullback for every $n \in \mathbb{N}$,
and the function $t_n$ is defined by its universal property
as the right-hand square commutes by~\ref{def:Csys}.\ref{def:Csys:lft}.
The set $T_0$ is a singleton by~\eqref{def:Csys:lt} in \cref{def:Csys}.

A homomorphism of C-systems $F \colon \sys{C} \to \sys{D}$ restricts, for every $n \in \mathbb{N}$,
to a function $F_n : T_n \to S_n$
between the fibres $T_n$ and $S_n$ of the length function of $\C$ and $\D$, respectively,
by~\ref{def:CsysHom}.\ref{def:CsysHom:l}
as in the front part of the diagram below.
\begin{linenomath*}
\begin{equation}\label{constr:csys2rttr:morph}
\begin{tikzcd}
T_{n+1}	\ar[dd,"F_{n+1}"{swap}] \ar[rr,hook] \ar[dr,"t_n"{swap}]	&&	\ob{\C}	\ar[dd,"\ob{F}"{near end}] \ar[dr,"\ft"]
&\\
&	T_n \ar[rr,hook,crossing over]	&&	\ob{\C}	\ar[dd,"\ob{F}"{swap}] \ar[dddd,bend left,"\length"]
\\
S_{n+1}	\ar[dd] \ar[rr,hook] \ar[dr,"s_n"{swap}]	&&	\ob{\D}	\ar[dd,"\length"{near end}] \ar[dr,"\ft"] 
&\\
&	S_n	\ar[from=uu,crossing over,"F_n"{near end}] \ar[rr,hook,crossing over]	&&	\ob{\D}	\ar[dd,"\length"{swap}]
\\
1	\ar[rr,"n+1"{near end}] \ar[dr]	&&	\mathbb{N}	\ar[dr,"\mathrm{pred}"{description}]
&\\
&	1	\ar[from=uu,crossing over] \ar[rr,"n"]	&&	\mathbb{N}
\end{tikzcd}
\end{equation}
\end{linenomath*}
The upper-right square commutes by~\cref{def:CsysHom}.\ref{def:CsysHom:ft},
thus the upper-left square commutes as well since the rest of the diagram commutes.
Functoriality holds since each $F_n$ is defined by a universal property.
\end{construction}

\begin{lem}\label{lem:freeproj}
Let $\sys{C}$ be a C-system with underlying \scat $\C$ and
let $\p(\sys{C})$ denote the wide subgraph of $\C$ on the canonical projections
$\p_{\Gamma}$ for $\Gamma$ in $\C$.
Then $\p(\sys{C})$ is isomorphic to the graph $\Gfctr \circ \Ctorttr(\sys{C})$
naturally in $\sys{C}$,
where $\Gfctr\colon \rttr \to \grph$ is from \cref{def:rttr-to-cat}.
\end{lem}

\begin{proof}
The vertices of $\Gfctr \circ \Ctorttr(\sys{C})$ are pairs $(\length(\Gamma),\Gamma)$
and edges are of the form
$(\length(\Gamma),\Gamma) \to (\length(\ft(\Gamma)),\ft(\Gamma))$
for $\length(\Gamma)>0$.
In particular, every vertex $(n+1,\Gamma)$
has exactly one outgoing edge.
The bijection between vertices then extends to an isomorphism
between $\p(\sys{C})$ and $\Gfctr \circ \Ctorttr(\sys{C})$.

Every C-homomorphism $F \colon \sys{C} \to \sys{D}$ induces
a morphism of graphs $\p(F) \colon \p(\C_{\sys{C}}) \to \p(\C_{\sys{D}})$
by \ref{def:CsysHom}.\ref{def:CsysHom:p}.
Naturality then follows from \ref{def:CsysHom}\ref{def:CsysHom:l}.
\end{proof}

\begin{cor}\label{cor:freeproj}
Let $\F$ be the free category on the graph $\p(\sys{C})$ on the canonical projections.
Then the terminal object $\Cterminal$ of $\C$ is terminal in $\F$ and
the function $\length$ extends to a stratification functor on $\F$.
\end{cor}

\begin{proof}
By \cref{lem:freeproj} there is an iso
$\F \jdeq \Ffctr \p(\sys{C}) \cong \Ffctr \circ \Gfctr \circ \Ctorttr(\sys{C})$.
The claim thus follows from \cref{lem:rttrStrCat}.
\end{proof}

\subsection{The category of CE-systems}
\label{ssec:ce-sys}

In this section, we define CE-systems and their morphisms.

\begin{defi}\label{def:CEsys}
A \define{CE-system} consists of
two \scat structures $\CEfam0{}$ and $\CEcat0{}$
on the same set of objects $\mathrm{Ob}(\CEfam0{}) \jdeq \mathrm{Ob}(\CEcat0{})$
and an identity-on-objects functor $I \colon \CEfam0{} \to \CEcat0{}$ between them,
together with
\begin{enumerate}
\item\label{def:CEsys:term}
a chosen object $\CEroot0{}$ which is terminal in $\CEfam0{}$, and
\item\label{def:CEsys:pb}
for any $f\colon\Delta\to\Gamma$ in $\CEcat0{}$ and any $A\in\CEfam0{}/\Gamma$,
a functorial choice of a pullback square 
\begin{linenomath*}
\begin{equation*}
\begin{tikzcd}[column sep=large]
  \CEctxext{\Delta}{\CEpb00{f}{A}} \arrow[r,"{\pi_2(f,A)}"] \arrow[d,swap,"{I(\CEpb00{f}{A})}"]
  &
  \CEctxext{\Gamma}{A} \arrow[d,"I(A)"]
  \\
  \Delta \arrow[r,"f"]
  &
  \Gamma
\end{tikzcd}
\end{equation*}
\end{linenomath*}
such that $\CEpb00{f}{A}\in\cat{F}/\Delta$.
Explicitly, the functoriality requirement is that
\begin{enumerate}
\item \label{def:CEsys:a}
For any $f\colon\Delta\to\Gamma$, one has
\begin{linenomath*}
\begin{equation*}
\CEpb01{f}{\catid{\Gamma}}\jdeq\catid{\Delta}
\qquad \text{and} \qquad
\CEqar{f}{\catid{\Gamma}}\jdeq f.
\end{equation*}
\end{linenomath*}
\item \label{def:CEsys:b}
  For any $A\in\cat{F}/\Gamma$, one has
  \begin{linenomath*}
\begin{equation*}
\CEpb10{\catid{\Gamma}}{A}\jdeq A
\qquad \text{and} \qquad
\CEqar{\catid{\Gamma}}{A}\jdeq\catid{\CEctxext{\Gamma}{A}}
\end{equation*}
\end{linenomath*}
\item \label{def:CEsys:c}
  For any $f\colon\Delta\to\Gamma$, $g\colon\Xi\to\Delta$ and $A\in\cat{F}/\Gamma$, one has
  \begin{linenomath*}
\begin{equation*}
\CEpb10{f\circ g}{A}\jdeq \CEpb01{g}{\CEpb00{f}{A}}
\qquad \text{and} \qquad
\CEqar{f\circ g}{A}\jdeq \CEqar{f}{A}\circ\CEqar{g}{\CEpb00{f}{A}}
\end{equation*}
\end{linenomath*}
\item \label{def:CEsys:d}
  For any $P\in\cat{F}/\ctxext{\Gamma}{A}$ and $f\colon\Delta\to\Gamma$, one has
  \begin{linenomath*}
\begin{equation*}
\CEpb01{f}{\CEfcmp{A}{P}}\jdeq \CEpb00{f}{A} \circ \CEpb10{\CEqar{f}{A}}{P}
\qquad \text{and} \qquad
\CEqar{f}{\CEfcmp{A}{P}} \jdeq \CEqar{\CEqar{f}{A}}{P}
\end{equation*}
\end{linenomath*}
\end{enumerate}
\end{enumerate}

A CE-system is \define{\rtdCE}
if $I(\CEroot0{}) \jdeq \CEroot0{}$ is terminal in $\CEcat0{}$.

For any $f\colon\Delta\to\Gamma$ we write $\CEpbf0f$ for the induced functor
$\cat{F}/\Gamma\to\cat{F}/\Delta$
and refer to the arrows in $\cat{F}$ as the \define{families} of the CE-system.
We shall write arrows in $\cat{F}$ with a double head as in the above diagram.

We may write $\CEfun1{A} \colon \CEfam1{A} \to \CEcat1{A}$ for the categories and functor underlying a CE-system $\sys{A}$,
whenever we need to make the CE-system explicit.
\end{defi}

We show in \cref{ssec:c2ce} that CE-systems generalize C-systems.
To provide some intuition, we can think of the image of $\CEfam0{}$ in $\CEcat0{}$ as the subcategory of $\CEcat0{}$ spanned by the projections $\p_{\Gamma} \colon \Gamma \to \ft(\Gamma)$ of a C-system.

\begin{exa}[CE-system on finite sets]\label{eg:CE-fin-set}
Let $\FF$ be the category whose objects are natural numbers, and whose morphisms $f \colon m \to n$
are functions $f \colon \std{m} \to \std{n}$ from the standard finite set of $m$ elements to the standard finite set of $n$ elements.
Consider the identity-on-objects functor $\fin{-} \colon (\mathbb{N},\geq) \to \FF\opcat$ given, on $n+k \geq n$, by the opposite of the initial-segment inclusion, which we write $i_n^{n+k}\colon \fin{n+k} \to \fin{n}$.

We equip it with the structure of a CE-system as follows.
The chosen pullback of a family $n+k \geq n$ and an arrow $f \colon \fin{m} \to \fin{n}$ in $\FF\opcat$ is
\begin{linenomath*}
\[\begin{tikzcd}[row sep=3em,column sep=6em]
\fin{m+k}	\ar[d,"i_m^{m+k}"] \ar[r,"{\pi_2(f,n+k \geq n)}"]
&	\fin{n+k}	\ar[d,"i_n^{n+k}"]
\\
\fin{m}	\ar[r,"f"]
&	\fin{n}
\end{tikzcd}\]
\end{linenomath*}
where the morphism $\pi_2(f,n+k \geq n)$ is the opposite of the arrow $[f,1_k] \colon \fin{n+k} \to \fin{m+k}$ in $\FF$
obtained from the universal property of the coproduct $\fin{n+k}$.
Functoriality follows immediately from the definitions.

This CE-system is, of course, rooted ---
as $\fin{0}$ is terminal in $\FF\opcat$ ---
and stratified in the sense of \cref{def:stratCE} ---
as initial-segment inclusions factor uniquely into arrows $i_n^{n+1}$ which are \indecarr in the sense of \cref{def:individual-ar}.
Note also that the choice of pullback squares is forced by \cref{rmk:grassmann}.

We can think of this example as the category of renamings, that is, variable-for-variable substitutions, of a untyped (or uni-typed) theory;
see, for instance, \cite{DBLP:conf/lics/FiorePT99,lamiaux2024introduction}.
\end{exa}

\begin{exa}
Categories with attributes~\cite{cartmell_phd}, or type categories~\cite{PittsCatLog}, produce examples of CE-systems which are rooted but not stratified.
A category with attributes consists of a category $\cat{C}$ with a terminal object $1$ together with a set of ``types'' $T(X)$ for each object of $\cat{C}$
such that each $A\in T(X)$ is assigned an arrow $p_A$ in $\cat{C}$ with codomain $X$.
Moreover a strictly functorial choice of pullbacks of these arrows along any arrow in $\cat{C}$ is required.
A CE-system is obtained by taking as $\cat{F}$ the free category on the arrows of the form $p_A$ and as $I$ the obvious functor into $\cat{C}$.
Another CE-system is obtained by taking as $\cat{F}$ the subcategory of $\cat{C}$ spanned by the arrows of the form $p_A$.
In this case the functor $I$ is simply the inclusion.

Display map categories~\cite{Taylor1999} and clans~\cite{Joyal2017clan} also produce examples of rooted non-stratified CE-systems, as soon as the choice of pullbacks is strictly functorial.
Recall that a display map category consists of a category $\cat{C}$ together with a class of arrows $\cat{D}$ such that pullbacks of arrows in $\cat{D}$ along any arrow in $\cat{C}$ exist and are again in $\cat{D}$.
Clans also have a terminal object $1$ and require $\cat{D}$ to be closed under composition and to contain all arrows towards $1$.
When the choice of pullbacks is strictly functorial,
the wide subcategory of $\cat{C}$ on the arrows in $\cat{D}$ together with the inclusion $\cat{D}\to\cat{C}$ provides an example of a rooted, non-stratified CE-system.
\end{exa}

\begin{defi}\label{def:CEhom}
Let $\sys{A}$ and $\sys{B}$ be two CE-systems.
A \define{CE-homomorphism} $F \colon \sys{A} \to \sys{B}$
consists of a commutative square of functors
\begin{linenomath*}
\begin{equation*}
\begin{tikzcd}[column sep=large]
\CEfam1{A}	\ar[d,swap,"\CEfun1{A}"] \ar[r,"\hCEfam{F}"]
&	\CEfam1{B}	\ar[d,"\CEfun1{B}"]
\\
\CEcat1{A}	\ar[r,"\hCEcat{F}"]	&	\CEcat1{B}
\end{tikzcd}
\end{equation*}
\end{linenomath*}
such that,
\begin{enumerate}
\item\label{def:CEhom:t}
$\hCEfam{F}(\CEroot1{A}) \jdeq \CEroot1{B}$, and
\item\label{def:CEhom:pb}
for every $A \in \CEfam1{A}/\Gamma$ and $f \colon \Delta \to \Gamma$,
it is
\begin{linenomath*}
\[\hCEfam{F}(\CEpb00{f}{A}) \jdeq \CEpb11{\hCEcat{F}f}{\hCEfam{F}A}
\qquad \text{and} \qquad
\hCEcat{F}(\CEqar{f}{A}) \jdeq \CEqar{\hCEcat{F}f}{\hCEfam{F}A}.
\]
\end{linenomath*}
\end{enumerate}
\end{defi}

\begin{rem}\label{rmk:CEhom-root}
If $F$ is a CE-homomorphism between \rtdCE CE-systems \sys{A} and \sys{B},
then $\hCEcat{F}(\CEroot1{A}) \jdeq \CEroot1{B}$ and
$\hCEcat{F}$ preserves terminal objects in the usual categorical sense.
\end{rem}

\begin{defi}
We write $\CEsys$ for the category of CE-systems and CE-system homomorphisms
and $\rCEsys$ for its full subcategory on \rtdCE CE-systems.
\end{defi}

For the comparison of CE-systems with C-systems, the notion of stratification of a CE-system is needed:

\begin{defi}\label{def:stratCE}
A CE-system \sys{A} is \define{stratified} if its category of families $\F$
is stratified in the sense of \cref{def:stratification} and,
for every $f \colon \Delta \to \Gamma$ in $\C$, the functor
\[\begin{tikzcd}[column sep=4em]
\F/\Gamma	\ar[r,"f^*"]	&	\F/\Delta
\end{tikzcd}\]
induced by the functorial choice of pullbacks is stratified
with respect to the stratification induced on slices in \cref{lem:strat-slice}.

A CE-homomorphism between stratified CE-systems is \define{stratified}
if its component on families is a stratified functor.
\end{defi}

\begin{rem}
  It follows from~\cref{prop:strat-is-prop} that CE-systems are stratified in at most one way.
\end{rem}

\begin{defi}
  We denote by $\strCEsys \hookrightarrow \CEsys$ and $\strrCEsys  \hookrightarrow \rCEsys$
  the respective subcategories spanned by \emph{stratified} (rooted) CE-systems and \emph{stratified} CE-homomorphisms between them.
\end{defi}

\begin{rem}\label{rmk:grassmann}
In a stratified CE-system,
for every $f \colon \Delta \to \Gamma$ in $\C$ and $A \in \F/\Gamma$ we have
\[
L(\Delta.f^*A) \jdeq L(\Delta) + L(\Gamma.A) - L(\Gamma).
\]
\end{rem}

\begin{lem}\label{lem:CEhom-ind}
Let \sys{A} and \sys{B} be two stratified CE-system.
A commuting square of functors
\begin{linenomath*}
\[\begin{tikzcd}[column sep=large]
\CEfam1{A}	\ar[d,swap,"\CEfun1{A}"] \ar[r,"\hCEfam{F}"]
&	\CEfam1{B}	\ar[d,"\CEfun1{B}"]
\\
\CEcat1{A}	\ar[r,"\hCEcat{F}"]	&	\CEcat1{B}
\end{tikzcd}\]
\end{linenomath*}
is a stratified CE-homomorphism $\sys{A} \to \sys{B}$ if and only if 
\begin{enumerate}
\item
$\hCEfam{F}$ is a stratified functor, and
\item
for every \indecarr arrow $A \in \CEfam1{A}/\Gamma$
and every $f \colon \Delta \to \Gamma$,
we have
\begin{linenomath*}
\[\hCEfam{F}(\CEpb00{f}{A}) \jdeq \CEpb11{\hCEcat{F}f}{\hCEfam{F}A}
\qquad \text{and} \qquad
\hCEcat{F}(\CEqar{f}{A}) \jdeq \CEqar{\hCEcat{F}f}{\hCEfam{F}A}.
\]
\end{linenomath*}
\end{enumerate}
\end{lem}

\begin{proof}
One direction is trivial.
The other one is proved by induction on the length $n$ of an arrow $A \in \F/\Gamma$.
\end{proof}

\subsection{Characterising C-systems as stratified CE-systems}
\label{ssec:c2ce}

Recall from \cref{cor:freeproj} that
every C-system $\sys{C}$ has a stratified wide subcategory $\F$ of its underlying category $\C$.
In this section, we show that the inclusion $\F \to \C$
has the structure of a stratified CE-system (\cref{constr:C2CE-ob}).
Moreover, we prove that this correspondence is functorial (\cref{constr:C2CE-funct})
and, in fact, an equivalence between the category of C-systems and the category of stratified CE-systems (\cref{thm:CasCE}).

\begin{prob}\label{prob:C2CE-ob}
To construct a CE-system $\CtoCE(\sys{C})$
from a C-system $\sys{C} \jdeq (\C, \Cterminal,\length,\ft,\ldots)$.
\end{prob}

\begin{construction}{prob:C2CE-ob}\label{constr:C2CE-ob}
Recall from \cref{lem:freeproj} that
$\p(\sys{C})$ denotes the wide subgraph of $\C$ on the canonical projections
$\p_{\Gamma}$ for $\Gamma$ in $\C$
and let $\F$ be the free category on $\p(\sys{C})$.
In particular, $\F$ has the same objects of $\C$
and the object $\Cterminal$ is terminal in $\F$ by \cref{cor:freeproj}.
It follows that the inclusion $\p(\sys{C}) \hookrightarrow \C$
extends to an identity-on-objects functor $I \colon \cat{F} \to \C$
that maps a path of length $n>0$
in $\p(\sys{C})$, \ie a list of composable canonical projections
\begin{linenomath*}
\begin{equation}\label{constr:C2CE-ob:fact}
\begin{tikzcd}[column sep=4em]
\Gamma	\ar[r,"\p_{\Gamma}"]	&	\ft(\Gamma)	\ar[r,"\p_{\ft(\Gamma)}"]	&	\cdots	\ar[r,"\p_{\ft^2(\Gamma)}"]
&	\ft^{n-1}(\Gamma)	\ar[r,"\p_{\ft^{n-1}(\Gamma)}"]	&	\ft^n(\Gamma).
\end{tikzcd}
\end{equation}
\end{linenomath*}
to their composite in $\C$.

It remains to provide $I$ with a suitable choice of pullback squares
along an arbitrary arrow $f \colon \Delta \to \Gamma$ in $\C$.
As an arrow $p \colon \Xi \to \Gamma$ in $\F$ is a path in $\p(\sys{C})$,
we proceed by induction on the length $n$ of the path $p$,
proving also conditions~\eqref{def:CEsys:b} and~\eqref{def:CEsys:c}
from \cref{def:CEsys}.

If $n\jdeq0$, the path $p$ is the identity on $\Gamma$ and we take
$f^*(\catid{\Gamma}) :\jdeq \catid{\Delta}$ and $\pi_2(f,\catid{\Gamma}) :\jdeq f$.
This choice is clearly functorial in $f$
and it trivially gives rise to a pullback square.
It also ensures condition~\eqref{def:CEsys:a}.

For $n>0$, it is $I(p) \jdeq \p_{\Xi} \circ I(p')$ where the length of $p'$ is $n-1$.
By inductive hypothesis we have $f^*p' \in \F/\Delta$ and a chosen pullback square of $I(p')$ along $f$,
which is the lower square in the diagram below.
The upper square is the canonical pullback square~\eqref{Csys-pbsq} given by the C-system structure.
\begin{linenomath*}
\begin{equation}\label{constr:C2CE-ob:pb}
\begin{tikzcd}[column sep=6em,row sep=3em]
(\pi_2(f,p'))^*\Xi	\ar[d,"\p_{(\pi_2(f,p'))^*\Xi}"{swap}] \ar[r,"{\q(\pi_2(f,p'),\Xi)}"]
&	\Xi	\ar[d,"\p_{\Xi}"]
\\
\ctxext{\Delta}{f^*(p')}	\ar[d,"I(f^*p')"{swap}] \ar[r,"{\pi_2(f,p')}"]
&	\ft(\Xi)	\ar[d,"I(p')"]
\\
\Delta	\ar[r,"f"]	&	\Gamma
\end{tikzcd}
\end{equation}
\end{linenomath*}
Thus we define $\pi_2(f,p) :\jdeq \q(\pi_2(f,p'),\Xi)$ and
$f^*p$ to be the concatenation of $f^*p'$ with $\p_{\p_{(\pi_2(f,p'))^*\Xi}}$
so that $I(f^*p) \jdeq I(f^*p') \circ \p_{\p_{(\pi_2(f,p'))^*\Xi}}$.
Functoriality in $f$ of this choice of pullback squares follows
from the fact that both the lower and upper pullback squares are functorial
by inductive hypothesis and by assumption, respectively.
In more details:
given $g \colon \Theta \to \Delta$, the inductive hypothesis yields
$(f \circ g)^*p' \jdeq g^*(f^*p')$ and
$\pi_2(f \circ g,p') \jdeq \pi_2(f,p') \circ \pi_2(g,f^*p')$.
It follows by~\ref{def:Csys}.\ref{def:Csys:pbcmp} that
\begin{linenomath*}
\[
\pi_2(f \circ g,p')^*\Xi \jdeq \pi_2(g,f^*p')^*\inpar1{\pi_2(f,p')^*\Xi}
\]
\end{linenomath*}
and, in turn, that $(f \circ g)^*p \jdeq g^*(f^*p)$.
The other component also follows from~\ref{def:Csys}.\ref{def:Csys:pbcmp}:
\begin{linenomath*}
\begin{align*}
\pi_2(f \circ g,p) &\jdeq \q(\pi_2(f \circ g,p'),\Xi)
\\&\jdeq
\q(\pi_2(f,p'),\Xi) \circ \q(\pi_2(g,f^*p'),\pi_2(f,p')^*{\Xi})
\\&\jdeq
\pi_2(f,p) \circ \pi_2(g,f^*p).
\end{align*}
\end{linenomath*}
Finally, condition~\eqref{def:CEsys:d} for a composite $q \circ p$ in $\F$
is proven by induction on the length of the path $p$.
\end{construction}

\begin{lem}\label{lem:C2CE-ob-props}
Let \sys{C} be a C-system and $\F$ the category of families of $\CtoCE(\sys{C})$.
\begin{enumerate}
\item\label{lem:C2CE-ob-props:indiv}
The \indecarr arrows in $\F$ are of the form $\p_{\Gamma}$
for some object $\Gamma$.
\item\label{lem:C2CE-ob-props:a}
The CE-system $\CtoCE(\sys{C})$ is stratified and
$L(\Gamma) \jdeq \ell(\Gamma)$, for every object $\Gamma$.
\item\label{lem:C2CE-ob-props:c}
The CE-system $\CtoCE(\sys{C})$ is rooted.
\end{enumerate}
\end{lem}

\begin{proof}~
  \begin{itemize}[align=left]
  \item [\ref*{lem:C2CE-ob-props:indiv}.]
    Immediate from the description of arrows in $\F$ in~\eqref{constr:C2CE-ob:fact}
    and~\ref*{def:Csys}.\ref{def:Csys:lft}.
   
  \item [\ref*{lem:C2CE-ob-props:a}.]
    By \cref{cor:freeproj},
    the category $\F$ is stratified
    and $L(\Gamma) \jdeq \ell(\Gamma)$.
    By \cref{lem:strat-funct}, it is enough to show that the choice of pullback squares
    in \cref{constr:C2CE-ob} preserves \indecarr arrows.
    But this follows immediately from the construction of pullbacks in~\eqref{constr:C2CE-ob:pb}
    and \eqref{lem:C2CE-ob-props:indiv} just shown.
    
  \item [\ref*{lem:C2CE-ob-props:c}.]
    The terminal object in $\F$ is terminal in $\C$ by assumption. \qedhere
    
  \end{itemize}
  
\end{proof}

\begin{prob}\label{prob:C2CE-funct}
To construct a functor $\CtoCE \colon \Csys \to \strrCEsys$
into rooted stratified CE-systems and stratified homomorphisms.
\end{prob}

\begin{construction}{prob:C2CE-funct}\label{constr:C2CE-funct}
The action of $\CtoCE$ on objects is defined in \cref{constr:C2CE-ob}.
Every morphism $F \colon \sys{C} \to \sys{D}$ of C-systems
restricts to the graphs of canonical projections
$\p(F) \colon \p(\sys{C}) \to \p(\sys{D})$
by conditions~(\ref*{def:CsysHom:t},\ref*{def:CsysHom:ft},\ref*{def:CsysHom:p})
in \cref{def:CsysHom}
and induces, in turn, a functor between free categories
$F_{\F} \colon \F_{\sys{C}} \to \F_{\sys{D}}$
whose action is determined by the action of $F$ on \indecarr arrows.
The square
\begin{linenomath*}
\[\begin{tikzcd}[column sep=3em]
\F_{\sys{C}}	\ar[d,"I_{\sys{C}}"{swap}] \ar[r,"F_{\F}"]
&	\F_{\sys{D}}	\ar[d,"I_{\sys{D}}"]
\\
\C_{\sys{C}}	\ar[r,"F"]	&	\C_{\sys{D}}
\end{tikzcd}\]
\end{linenomath*}
commutes since it does so when precomposed by the unit
$\p(\sys{C}) \to \F_{\sys{C}}$.
The functor $F_{\F}$ is stratified by~\ref{def:CsysHom}.\ref*{def:CsysHom:l}.
\Cref{lem:CEhom-ind} then ensures that the pair $\CtoCE(F):\jdeq (F_{\F},F)$
lifts to a stratified CE-homomorphism as soon as
it preserves pullbacks of \indecarr arrows.
But this is precisely condition~\ref{def:CsysHom}.\ref*{def:CsysHom:pb}.
Functoriality of $\CtoCE$ follows since $F_{\F}$ is defined by a universal property.
\end{construction}

\begin{prob}\label{prob:CEtoC-ob}
To construct a C-system $\CEtoC(\sys{A})$ from a stratified and rooted CE-system \sys{A}.
\end{prob}

\begin{construction}{prob:CEtoC-ob}\label{constr:CEtoC-ob}
Let $I \colon \F \to \C$ be the underlying functor of \sys{A}.
The underlying category of $\CEtoC(\sys{A})$ is $\C$
and the length function $\ell$ is given by the action of the stratification functor
$L$ on objects.
Since \sys{A} is rooted, the chosen terminal object $\CEroot0{}$ in $\F$ is terminal in $\C$ too.
Conditions~\eqref{def:Csys:t}~and~\eqref{def:Csys:lt} are clearly met.

Given an object $X$ with $n \jdeq L(X) > 0$,
let $X \xrightarrow{x_n} X_{n-1} \to \cdots \to X_1 \xrightarrow{x_1} \CEroot0{}$
be the factorisation of $X \to \CEroot0{}$ into $n$ \indecarr arrows in $\F$.
We define
\begin{linenomath*}
\begin{equation}\label{constr:CEtoC-ob:p}
\ft(\CEroot0{}) :\jdeq \CEroot0{},\quad
\ft(X) :\jdeq X_{n-1}\quad \text{and}\quad
\p_X :\jdeq I(x_n).
\end{equation}
\end{linenomath*}
Conditions~\eqref{def:Csys:lft} and~\eqref{def:Csys:ftt}
hold by construction.

Given also $f \colon Y \to \ft(X)$,
let $Y \xrightarrow{y_n} Y_{n-1} \to \cdots \to Y_1 \xrightarrow{y_1} \CEroot0{}$
be the factorisation of $Y \to \CEroot0{}$ into \indecarr arrows and
consider the pullback square below.
\begin{linenomath*}
\begin{equation}\label{constr:CEtoC-ob:pb}
\begin{tikzcd}[column sep=3em]
Y.f^*\!x_n	\ar[d,"I(f^*\!x_n)"{swap}] \ar[r,"{\pi_2(f,x_n)}"]	&	X	\ar[d,"I(x_n)"]
\\
Y	\ar[r,"f"]	&	\ft(X)
\end{tikzcd}
\end{equation}
\end{linenomath*}
It is $L(Y.f^*(x_n)) \jdeq L(Y) + 1$
by \cref{rmk:grassmann}, thus
$Y.f^*x_n \xrightarrow{f^*\!x_n} Y \xrightarrow{y_n} Y_{n-1} \to \cdots \to Y_1 \xrightarrow{y_1} \CEroot0{}$
is the factorisation of $Y.f^*\!x_n \to \CEroot0{}$ into \indecarr arrows.
It follows that $\ft(Y.f^*\!x_n) \jdeq Y$ and $\p_{f^*\!x_n} \jdeq I(f^*\!x_n)$.
Condition~\eqref{def:Csys:pb} follows defining $f^*X :\jdeq Y.f^*(x_n)$ and $\q(f,X) :\jdeq \pi_2(f,x_n)$.
Condition~\eqref{def:Csys:pbid} holds by~\ref{def:CEsys}.\ref{def:CEsys:b} since $\ft(X).x_n \jdeq X$,
and~\eqref{def:Csys:pbcmp} by~\ref{def:CEsys}.\ref{def:CEsys:c} as below:
\begin{linenomath*}
\begin{align*}
(f \circ g)^*X &\jdeq Z.(f \circ g)^*x_n \jdeq Z.\inpar1{g^*(f^*\!x_n)}
\\&\jdeq
g^*(f^*X)
\\[1ex]
\q(f \circ g, X) &\jdeq \pi_2(f \circ g, x_n) \jdeq \pi_2(f, x_n) \circ \pi_2(g,f^*\!x_n)
\\&\jdeq
\q(f,X) \circ \q(g,f^*X).
\qedhere
\end{align*}
\end{linenomath*}
\end{construction}

\begin{lem}\label{lem:CEtoC-fun}
Let $F \colon \sys{A} \to \sys{B}$ be a stratified homomorphism
of rooted stratified CE-systems.
Then the underlying functor
$F \colon \CEcat1A \to \CEcat1B$
is a homomorphism of C-systems
$\CEtoC(F) \colon \CEtoC(\sys{A}) \to \CEtoC(\sys{B})$.
\end{lem}

\begin{proof}
We verify the conditions in \cref{def:CsysHom}.
\eqref{def:CsysHom:t}
The functor $F$ maps the chosen terminal object of $\sys{A}$
to the one of $\sys{B}$ by assumption.
\eqref{def:CsysHom:l}
Since $F$ is stratified,
its action on objects commutes with the length functions.
(\ref{def:CsysHom:ft}--\ref{def:CsysHom:p})
The action on objects also
preserves \indecarr arrows by \cref{lem:strat-funct},
thus it commutes with the father functions
and preserves canonical projections.
\eqref{def:CsysHom:pb}
$F$ maps chosen pullback squares in $\sys{A}$
to chosen ones in $\sys{B}$ by \ref{def:CEhom}.\ref{def:CEhom:pb}.
In particular, it preserves the choice of pullbacks along \indecarr arrows.
\end{proof}

\begin{defi}\label{def:CEtoC-fun}
Let $\CEtoC \colon \strrCEsys \to \Csys$ be the functor given
by \cref{constr:CEtoC-ob} and \cref{lem:CEtoC-fun}.
\end{defi}

\begin{lem}\label{lem:csys-retr}
For every C-system $\sys{C}$, 
the identity functor on the underlying strict category of $\sys{C}$
is an isomorphism $\CEtoC(\CtoCE(\sys{C})) \cong \sys{C}$ of C-systems,
naturally in $\sys{C}$.
\end{lem}

\begin{proof}
Let $\C$ be the underlying strict category of $\sys{C}$.
To see that the identity functor $\catid{\C}$ is a \sys{C}-homomorphism note first that
the category $\C$, its terminal object and the length function
are the same in $\CEtoC(\CtoCE(\sys{C}))$ and $\sys{C}$.
Since \indecarr arrows in $\CtoCE(\sys{C})$
coincide with the canonical projections $\p_{\Gamma}$
by \cref{lem:C2CE-ob-props},
factorisations in $\CtoCE(\sys{C})$ into \indecarr arrows
are of the form in \eqref{constr:C2CE-ob:fact}.
It follows that the function $\ft$ and the canonical projections
as defined in~\eqref{constr:CEtoC-ob:p}
are equal to the ones from $\sys{C}$.
Since the choice of pullback squares in $\CtoCE(\sys{C})$
is defined inductively by the choice along \indecarr arrows in~\eqref{constr:C2CE-ob:pb},
the choice of pullbacks along canonical projections in~\eqref{constr:CEtoC-ob:pb}
coincides with the one in $\sys{C}$.

Naturality follows from the fact that $\CEtoC(\CtoCE(F)) \jdeq F$
for every C-homomorphism $F$.
\end{proof}

\begin{thm}\label{thm:CasCE}
The functor $\CtoCE \colon \Csys \to \strrCEsys$
from \cref{constr:C2CE-funct}
is an equivalence.
\end{thm}

\begin{proof}
By \cref{lem:csys-retr}, it is enough to find,
for every stratified rooted CE-system \sys{A},
an isomorphism $\CtoCE(\CEtoC(\sys{A})) \cong \sys{A}$
natural in $\sys{A}$.
Let $I \colon \F \to \C$ be the underlying functor of $\sys{A}$
and let $\F_{\mathrm{i}} :\jdeq \p(\CEtoC(\sys{A}))$
be the subgraph of $\F$ on the \indecarr arrows.
The CE-system $\CtoCE(\CEtoC(\sys{A}))$ consists, in particular,
of a functor $\widehat{I} \colon \widehat{\F} \to \C$,
where $\widehat{\F}$ is the free category on $\F_{\mathrm{i}}$,
and $\widehat{I}$ maps a list of composable \indecarr arrows
to the composite of their images in $\C$ under $I$,
by~\eqref{constr:CEtoC-ob:p},~\eqref{constr:C2CE-ob:fact} and \cref{lem:C2CE-ob-props}.\ref{lem:C2CE-ob-props:a}.

The inclusion $\F_{\mathrm{i}} \hookrightarrow \F$ induces
an identity-on-objects functor $\mathsf{comp} \colon \widehat{\F} \to \F$.
Conversely, the factorisation into \indecarr arrows~\eqref{def:stratification:fact}
in $\sys{A}$ yields an identity-on-objects functor $\mathsf{fact} \colon \F \to \widehat{\F}$,
which is a (strict) inverse of $\mathsf{comp}$.
Since $I$ is a functor, the squares
\begin{linenomath*}
\begin{equation}\label{prop:CasCE:natiso}
\begin{tikzcd}[column sep=3em,row sep=2em]
\widehat{\F}		\ar[d,"\widehat{I}"{swap}] \ar[r,"\mathsf{comp}"]	&	\F	\ar[d,"I"]
\\
\C	\ar[r,"\catid{\C}"]	&	\C
\end{tikzcd}
\hspace{5em}
\begin{tikzcd}[column sep=3em,row sep=2em]
\F		\ar[d,"I"{swap}] \ar[r,"\mathsf{fact}"]	&	\widehat{\F}		\ar[d,"\widehat{I}"]
\\
\C	\ar[r,"\catid{\C}"]	&	\C
\end{tikzcd}
\end{equation}
\end{linenomath*}
commute.
Since both functors $\mathsf{comp}$ and $\mathsf{fact}$
are identities on objects and on \indecarr arrows,
the squares above are stratified CE-homomorphisms
$\CtoCE(\CEtoC(\sys{A})) \to \sys{A}$ and $\sys{A} \to \CtoCE(\CEtoC(\sys{A}))$,
respectively, by \cref{lem:CEhom-ind}.
Therefore $\CtoCE(\CEtoC(\sys{A})) \cong \sys{A}$.

To see that this isomorphism is natural in $\sys{A}$,
note that $(\mathsf{comp},\catid{\C})$ is natural in $\sys{A}$
since $\mathsf{comp}$ is equivalently defined as the composite of
the counit of the free-forgetful adjunction at $\F$
with the image under the left adjoint
of the graph inclusion $\F_{\mathrm{i}} \hookrightarrow \F$.
\end{proof}

\section{The category of B-systems}\label{sec:b-sys}

In this section, we study Voevodsky's B-systems.

In~\cref{ssec:b-sys} we review the definition of B-systems and their homomorphisms.
In~\cref{ssec:e-sys} we introduce the notion of E-system and their homomorphisms.
Intuitively, E-systems model type theory with strict $\Sigma$-types, see~\cref{ssec:e2ce-pair}.
Finally, in~\cref{ssec:b2e} we construct an equivalence between the category of B-systems and the subcategory of ``stratified'' E-systems.

In order to simplify the construction of such equivalence,
we structure the definitions in the next sections in three steps.
In the case of B-systems, for example,
we first introduce some piece of structure on sets
consisting of functions,
which we refer to as \emph{pre}-B-systems,
see \cref{def:b-frame-structure}.
Then we define morphisms between these structures,
also called \emph{pre}-homomorphisms,
and finally we define B-systems as those pre-B-systems
whose structure functions are themselves pre-homomorphisms.
Homomorphisms are then just pre-homomorphisms between B-systems.
We shall follow the same pattern
when introducing each of the structures that give rise to an E-system,
in \cref{sssec:substsys,sssec:weaksys,sssec:projsys},
and when defining E-systems in \cref{sssec:esys}.

\subsection{The category of B-systems}
\label{ssec:b-sys}

In this section, we review the definition of Voevodsky's B-systems~\cite{VV_B-systems}.
We shall rephrase his definition in order to introduce a few auxiliary intermediate structures
which we will use in later constructions.
An explicit comparison is in \cref{rem:Bsys-compar}.

\begin{defi}
  A \define{B-frame} $\sys{B}$ is a diagram of sets of the following form:
  \begin{linenomath*}
  \begin{equation*}
    \begin{tikzcd}[column sep=1ex]
      &[4ex] & & \tilde{B}_1 \arrow[dl,swap,"\bd"] & & \tilde{B}_2 \arrow[dl,swap,"\bd"] \\
      \{\ast\}\cong B_0 & & B_1 \arrow[ll,"\ft"] & & B_2 \arrow[ll,"\ft"] & & \cdots \arrow[ll,"\ft"]
    \end{tikzcd}
  \end{equation*}
  \end{linenomath*}
  In other words, a B-frame consists of:
  \begin{enumerate}
  \item for all $n\in\mathbb{N}$ two sets $B_n$ and $\tilde{B}_{n+1}$. 
  \item for all $n\in\mathbb{N}$ functions of the form
    \begin{linenomath*}
      \begin{align*}
        \ft[n] & : B_{n+1}\to B_n \\
        \bd[n] & : \tilde{B}_{n+1}\to B_{n+1}.
      \end{align*}
    \end{linenomath*}
    called \define{father functions} and \define{boundary functions}, respectively.
  \item $B_0$ is a singleton.
  \end{enumerate}
  For $m,n\in\mathbb{N}$, we denote the composition $\ft[n]\circ\cdots\circ\ft[n+m]:B_{n+m+1}\to B_n$ by $\ft[n]^m$.

  A \define{homomorphism $H:\sys{B}\to\sys{A}$ of B-frames} is a natural transformation of B-frames, i.e., it consists of maps
  $H_n:B_n\to A_n$ and $\tilde{H}_{n+1}:\tilde{B}_{n+1}\to\tilde{A}_{n+1}$ such that
  \begin{linenomath*}
  \begin{align*}
  \ft(H(X)) & \jdeq H(\ft(X)) \\
  \bd(\tilde{H}(x)) & \jdeq H(\bd(x))
  \end{align*}
  \end{linenomath*}
  for any $X\in B_n$ and $x\in\tilde{B}_{n+1}$. The category of B-frames is
  denoted by $\Bfr$. 
\end{defi}

As we already did in the above definition, we shall often omit the subscripts from father and boundary functions and from homomorphisms of B-frames,
since these can be easily inferred from the context (often from their argument).

We should remark that we are abusing notation when denoting the pieces of structure of B-frames and, later on, of B-systems.
Regarding a B-frame as a diagram on the ``comb category''
\[\begin{tikzcd}[column sep=small,row sep=tiny]
      & & & \tilde{1} \arrow[dl,swap] & & \tilde{2} \arrow[dl,swap] \\
      0 & & 1 \arrow[ll] & & 2 \arrow[ll] & & \cdots \arrow[ll]
\end{tikzcd}\]
we denote the value of a B-frame at $n$ by changing the font from blackboard bold to roman,
and the value at $\tilde{n}$ by also adding a tilde.
This (can) only apply to the blackboard bold letter:
for instance, the values of the slice B-frame (defined in \cref{def:Bfr-slice}) $\sys{B}/X$
are denoted $(B/X)_m$ and $(\tilde{B}/X)_{m+1}$.

We shall consider specific B-frames (and B-systems) only as individual examples, and we do not need to give names to them.
Otherwise, we shall only deal with generic B-frames (and B-systems),
denoted by capital blackboard letters.
Therefore we do not expect this abuse of notation to create any inconvenience.

To provide some intuition for B-frames, we look back at the introduction, where we constructed, implicitly, a B-frame from a module over a monad.

\begin{exa}\label{eg:b-frame-R-LM}
  Recall from \cref{sec:b-syst-constr} the two sets $B(R,LM)$ (see \cref{eq:B-sets-of-module}) and $\wt{B}(R,LM)$ (see \cref{eq:Btilde-sets-of-module}).
  From these sets, we obtain a B-frame with the following sets of families (note the shift in the indexing of $\tilde{B}$),
  \begin{linenomath*}
  \begin{align*}
    B_n &\defeq B(R,LM)_n \defeq \prod_{i=0}^{n-1}LM(\ff{i})
    \\
    \tilde{B}_{n+1} &\defeq \wt{B}(R,LM)_n \defeq \prod_{i=0}^{n-1}LM(\ff{i})\times R(\ff{n})\times LM(\ff{n})
  \end{align*}
  \end{linenomath*}
  and the obvious maps for $\ft$ and $\bd$. We call this B-frame the \emph{B-frame generated by a module $LM$ over a monad $R$}.
  We write elements of $B_n$ as $A_0,\ldots,A_{n-1} \vdash A_n$,
  and elements of $\wt{B}_{n+1}$ as $A_0,\ldots,A_{n} \vdash t : A_{n+1}$, where $t \in R(\ff{n})$.
\end{exa}

More generally, the elements of $B_{n+1}$ of a B-frame can be thought of as a pair of a context of length $n$, and a type in that context.
Hence, the elements of $B_1$ are the types in the empty context.
Just like with C-systems, there is no explicit structure to denote types in a given context.
An element $t \in \tilde{B}_{n+1}$ is then a term, and the context and type $t$ lives in is given by $\bd[n+1](t)$.

\begin{exa}\label{eg:b-frame-finsets}
Recall from \cref{eg:CE-fin-set} that $\std{n}$ denotes the set $\{0,\ldots,n-1\}$.
We shall consider the B-frame defined, for each $n \in \N$,
by $B_n \defeq \{n\}$ and $\tilde{B}_{n+1} \defeq \std{n}$.
\end{exa}

\begin{exa}\label{eg:b-frames-equiv-type-term}
  B-frames are the same as Garner's ``type-and-term structures''~\cite[Def.~8]{GARNER20151885}.
  Garner~\cite[Prop.~13]{GARNER20151885} constructs an equivalence between the category of type-and-term structures
  and the category of $\emptyset$-GATs, that is, of Generalized Algebraic Theories~\cite{DBLP:journals/apal/Cartmell86} without weakening, projection, and substitution rules, and interpretations between them.
\end{exa}

We now define more structure on B-frames which represents operations on syntax.

The first operation could be called ``slicing''; given a B-frame $\sys{B}$ and a ``context'' $X \in B_n$ in that B-frame, we construct the slice of $\sys{B}$ over $X$:
\begin{defi}\label{def:Bfr-slice}
  For every B-frame $\sys{B}$ and any $X\in B_n$, there is a B-frame
  $\sys{B}/X$ given by
  \begin{linenomath*}
    \begin{align*}
      (B/X)_{m} & \defeq \{Y\in B_{n+m}\mid\ft^{m}(Y)\jdeq X\}\\
      (\tilde{B}/X)_{m+1} & \defeq \{y\in \tilde{B}_{n+m+1}\mid\ft^{m+1}(\bd(y))\jdeq X\}.
    \end{align*}
  \end{linenomath*}
  Also, for any homomorphism $H:\sys{B}\to\sys{A}$ of B-frames and any
  $X\in B_n$, there is a homomorphism $H/X:\sys{B}/X\to\sys{A}/H(X)$
  defined in the obvious way.
\end{defi}

Note that for $X\in B_n$ and $Y\in B_{n+m}$ such that $\ft^m(Y)\jdeq X$, 
we have an isomorphism $(\sys{B}/X)/Y\cong B/Y$ of B-frames, constructed in the obvious way, which is natural in the sense that for any homomorphism $H:\sys{B}\to\sys{A}$ of B-frames, the square
\begin{linenomath*}
\begin{equation*}
\begin{tikzcd}
(\sys{B}/X)/Y \arrow[r,"\cong"] \arrow[d,swap,"(H/X)/Y"] &  \sys{B}/Y \arrow[d,"H/Y"] \\
(\sys{A}/H(X))/H(Y) \arrow[r,"\cong"] & \sys{A}/H(Y)
\end{tikzcd}
\end{equation*}
\end{linenomath*}
commutes.

\begin{defi}\label{def:B-to-rttr}
Every B-frame $\sys{B}$ has an underlying rooted tree given by the sets
$B_n$ and the functions $\ft[n] \colon B_{n+1} \to B_n$, for $n \in \N$.
Similarly, a homomorphism of B-frames is in particular a homomorphism of rooted trees.
Thus we define
\begin{linenomath*}
\[\begin{tikzcd}[column sep=4em]
\Bfr	\ar[r,"\Btorttr"]	&	\rttr
\end{tikzcd}\]
\end{linenomath*}
to be the forgetful functor from B-frames to rooted trees.
\end{defi}

We will now consider different type-theoretic structures on B-frames, specifically
substitution, weakening, and projection.
Garner considers similar structures in terms of algebras of suitable monads on the category of B-frames a.\,k.\,a.~type-and-term structures.
We have not established a precise relationship (e.g., an equivalence) between our structures and the ones obtained by Garner as the algebras for his monads.

\begin{defi}\label{def:b-frame-structure}
  \begin{enumerate}
  \item A \define{substitution structure} on a B-frame $\sys{B}$ is a collection of homomorphisms
    \begin{linenomath*}
    \begin{equation*}
      S_x : \sys{B}/\bd(x) \to \sys{B}/\ft(\bd(x))
    \end{equation*}
    \end{linenomath*}
    for all $x\in \tilde{B}_{n+1}$ and all $n\in\mathbb{N}$.
  \item A \define{weakening structure} on a B-frame $\sys{B}$ is a collection of homomorphisms
    \begin{linenomath*}
    \begin{equation*}
      W_X : \sys{B}/\ft(X)\to\sys{B}/X
    \end{equation*}
    \end{linenomath*}
    for all $X\in B_{n+1}$ and all $n\in\mathbb{N}$.
  \item The \define{structure of generic elements} on a B-frame $\sys{B}$ equipped with weakening structure $W$ is a collection of functions
    \begin{linenomath*}
    \begin{align*}
      \delta_n & : B_{n+1}\to \tilde{B}_{n+2}
    \end{align*}
    \end{linenomath*}
    such that $\bd(\delta_n(X))\jdeq W_{X}(X)$ for any $X\in B_{n+1}$.
  \end{enumerate}
  A \define{pre-B-system} $\sys{B}$ is a B-frame equipped with weakening structure, substitution structure, and generic elements.
\end{defi}

We shall often omit the subscript from the functions $\delta_n$,
since it can be easily inferred from the context.

\begin{exa}
  Consider the B-frame generated by a module $LM$ over a monad $R$ of  \cref{eg:b-frame-R-LM}.
  Given an element $x \in \tilde{B}_{n+1}$, and hence in particular, a term $t \in R(\ff{n})$,
  we obtain a substitution map $S_x : \sys{B}/\bd(x) \to \sys{B}/\ft(\bd(x))$ that substitutes the term $t$ for the ``last'' free variable in any element of $\sys{B}$ lying ``over'' $\bd(x)$.
  For instance, taking $x$ to be $A_0 \vdash t_1 : A_1$, the substitution $S_x$ maps the element $A_0,A_1 \vdash s : A_2$ to $A_0 \vdash s[t_1] : A_2[t_1]$.

  For weakening, consider $X \in B_{1+1}$ to be a context $A_0 \vdash A_1$.
  The weakening $W_X$ maps any context of the form $A_0,A'_1,\ldots,A'_n\vdash A'_{n+1}$ to the weakened context $A_0,A_1,A'_1,\ldots,A'_n\vdash A'_{n+1}$, and similar for elements in $\wt{B}$.

  For the generic element, consider, for instance, a context $X = A_0 \vdash A_1$ in $B_2$.
  This context induces the generic element $A_0,A_1 \vdash \textsf{var}(1) : A_1$, where $\eta(1)\in R(\ff{2})$ is the ``de Bruijn'' variable $1$ bound by $A_1$ in the context, and considered as a term by being wrapped in an application of the monadic unit $\eta$ of the monad $R$ (the inclusion of variables into terms).
  We have
  \begin{linenomath*}
    \[ \bd(A_0,A_1 \vdash \textsf{var}(1) : A_1) \enspace = \enspace A_0, A_1 \vdash A_1 \enspace = \enspace W_{A_0 \vdash A_1}(A_0 \vdash A_1).\]
  \end{linenomath*}
\end{exa}

\begin{exa}\label{eg:b-frame-struct-finsets}
Recall the B-frame of finite sets defined in \cref{eg:b-frame-finsets}.
Here we construct structures of substitution, weakening and generic elements on it.

Note first that its slice on the (unique) element $n$ in $B_n$
is such that
\begin{linenomath*}
\[
(\tilde{B}/n)_{m+1} \jdeq \tilde{B}_{n+m+1} \jdeq \std{n+m}.
\]
\end{linenomath*}
It follows that a substitution structure
must consist of a family of functions
$S_{x,j} \colon \std{n+1+j} \to \std{n+j}$,
for $n,j\in\N$ and $x \in \tilde{B}_{n+1}$.
We define $S_{x,j} \defeq s_x + \catid{j} $,
where $s_x$ is the function $[\catid{n},x] \colon \std{n+1} \to \std{n}$
given by the universal property of the coproduct $\std{n+1}$.
In other words, $S_{x,j}$ lists all elements in $\std{n+j}$
repeating the element $x\in\std{n}$ in position $n+1$.
In particular, it fixes the first $n$ elements,
and decreases the last $j$ by $1$.

Similarly,
a weakening structure
must consist of a family of functions $W_{n,j} \colon \std{n+j} \to \std{n+1+j}$.
We define $W_{n,j}$ to be the function $i_n + \catid{j}$,
where $i_n \colon \std{n} \to \std{n+1}$ is the initial-segment inclusion.
In other words, it lists all elements in $\std{n+1+j}$ except for $n$.
Equivalently, it fixes the first $n$ elements,
and increases the remaining $j$ by $1$.

Finally, the structure of generic elements is given by
an element $\delta_{n} \in \tilde{B}_{n+2} \jdeq \std{n+1}$
for every $n \in \N$,
which we define to be its maximum,
that is, $\delta_{n} \defeq n$.

Taking advantage of the fact that finite sets are finite coproducts,
and slightly abusing notation,
we find it convenient to write
\begin{linenomath*}
\[\begin{tikzcd}[row sep=0pt,column sep=4em]
n	\ar[dr] \ar[dd,phantom,"+"{description}]
\\&
	n	\ar[dd,phantom,"+"{description}]
\\
1	\ar[ur,"x"{swap}] \ar[dd,phantom,"+"{description}]
\\&
	j
\\
j	\ar[ur]
\end{tikzcd}
\hspace{8em}
\begin{tikzcd}[row sep=0pt,column sep=4em]
&	n	\ar[dd,phantom,"+"{description}]
\\
n	\ar[ur] \ar[dd,phantom,"+"{description}]
\\&
	1	\ar[dd,phantom,"+"{description}]
\\
j	\ar[dr]
\\&
	j
\end{tikzcd}\]
\end{linenomath*}
for the functions
$S_{x,j} \colon \std{n+1+j} \to \std{n+j}$
and $W_{n,j} \colon \std{n+j} \to \std{n+1+j}$, respectively.
\end{exa}

\begin{defi}\label{def:Bsys-pres}~
  \begin{enumerate}
  \item\label{def:Bsys-pres:sub}
    Consider two B-frames $\sys{B}$ and $\sys{A}$, both equipped with substitution structure. A homomorphism $H:\sys{B}\to\sys{A}$ of B-frames is said to \define{preserve the substitution structure} if the diagram
    \begin{linenomath*}
    \begin{equation*}
      \begin{tikzcd}[column sep=huge]
        \sys{B}/\bd(x) \arrow[r,"H/\bd(x)"] \arrow[d,swap,"S_x"] & \sys{A}/\bd(H(X)) \arrow[d,"S_{\tilde{H}(x)}"] \\
        \sys{B}/\ft(\bd(x)) \arrow[r,swap,"H/\ft(\bd(x))"] & \sys{A}/\ft(\bd(H(X)))
      \end{tikzcd}
    \end{equation*}
    \end{linenomath*}
    of B-frame homomorphisms commutes for every $x\in\tilde{B}_{n+1}$ and every $n\in\mathbb{N}$.
  \item\label{def:Bsys-pres:weak}
    Consider two B-frames $\sys{B}$ and $\sys{A}$, both equipped with weakening structure. A homomorphism $H:\sys{B}\to\sys{A}$ of B-frames is said to \define{preserve the weakening structure} if the diagram
    \begin{linenomath*}
    \begin{equation*}
      \begin{tikzcd}[column sep=huge]
        \sys{B}/X \arrow[r,"H/X"] & \sys{A}/H(X) \\
        \sys{B}/\ft(X) \arrow[u,"W_X"] \arrow[r,swap,"H/\ft(X)"] & \sys{A}/\ft(H(X)) \arrow[u,swap,"W_{H(X)}"]
      \end{tikzcd}
    \end{equation*}
    \end{linenomath*}
    of B-frame homomorphisms commutes for all $X\in B_n$ and all $n\in\mathbb{N}$.
  \item\label{def:Bsys-pres:gen}
    Consider two B-frames $\sys{B}$ and $\sys{A}$, both equipped with weakening structure, and both equipped with generic elements. A B-frame homomorphism $H:\sys{B}\to\sys{A}$ is said to \define{preserve the generic elements} if
    \begin{linenomath*}
    \begin{equation*}
      \tilde{H}(\delta(X))\jdeq\delta(H(X))
    \end{equation*}
    \end{linenomath*}
    for any $X\in B_{n+1}$ and any $n\in\mathbb{N}$.
  \end{enumerate}
  A \define{pre-B-homomorphism} $H:\sys{B}\to\sys{A}$ is a homomorphism of pre-B-systems preserving the weakening structure, substitution structure and the generic elements.
\end{defi}

\begin{defi}\label{def:Bsys}
A \define{B-system} is a pre-B-system for which the following conditions
hold:
\begin{enumerate}
  \item\label{def:Bsys:sub}
    Every $S_x$ is a pre-B-homomorphism.
  \item\label{def:Bsys:weak}
    Every $W_X$ is a pre-B-homomorphism.
  \item\label{def:Bsys:SxW}
    For every $x\in \tilde{B}_{n+1}$ one has $S_x\circ W_{\bd(x)}\jdeq
    \catid{\sys{B}/\ft(\bd(x))}$. 
  \item\label{def:Bsys:SxId}
    For every $x\in\tilde{B}_{n+1}$ one has $S_x(\delta(\bd(x)))\jdeq x$.
  \item\label{def:Bsys:SIdW}
    For every $X\in B_{n+1}$ one has $S_{\delta(X)}\circ W_X/X\jdeq
    \catid{\sys{B}/X}$. 
  \end{enumerate}
  \define{B-homomorphisms} are pre-B-homomorphisms between B-systems.
  We denote the category of B-systems by $\Bsys$.
\end{defi}

The idea is that first substitution and weakening preserve all the structure of a\linebreak[5] (pre-)B-system.
The third axiom asserts that substitution in weakened type families is constant.
Furthermore, the generic elements should behave like internal identity morphisms.
Axioms \ref{def:Bsys:SxId} and \ref{def:Bsys:SIdW} are akin to two of the well-known monadic laws of substitution.

\begin{rem}\label{rem:Bsys-compar}
Here we provide an explicit comparison of \cref{def:Bsys} with Voevodsky's definition of B-system in the arXiv version of~\cite{VV_B-systems}.

Conditions~1--3 in~\cite[Def.~2.1]{VV_B-systems} and condition~1 in~\cite[Def.~2.5]{VV_B-systems} define a B-frame $\sys{B}$.
The functions $T$ and $\widetilde{T}$ from condition~4  in~\cite[Def.~2.1]{VV_B-systems}
together with conditions~2 and~3 in~\cite[Def.~2.5]{VV_B-systems} define a weakening structure $W$ on $\sys{B}$.
The functions $S$ and $\widetilde{S}$ from condition~4  in~\cite[Def.~2.1]{VV_B-systems}
together with conditions~4 and~5 in~\cite[Def.~2.5]{VV_B-systems} define a substitution structure $S$ on $\sys{B}$.
The function $\delta$ in~\cite[Def.~2.1]{VV_B-systems} together with the condition in~\cite[Def.~2.6]{VV_B-systems} defines a structure of generic elements on $\sys{B}$ with weakening structure $W$.
Therefore what we call a pre-B-system is a unital B0-system in~\cite{VV_B-systems}.

Consider now the conditions in~\cite[Def.~3.1 and~3.2]{VV_B-systems}.
The TT-condition amounts to saying that every $W_X$ preserves the weakening structure,
the ST-condition amounts to saying that every $W_X$ preserves the substitution structure,
and the $\delta$T-condition amounts to saying that every $W_X$ preserves the generic elements.
Therefore condition~\eqref{def:Bsys:weak} in \cref{def:Bsys} unfolds to~\cite[3.1.1, 3.1.4, 3.2.1]{VV_B-systems}.
Similarly, the SS-condition amounts to saying that every $S_x$ preserves the substitution structure,
the TS-condition amounts to saying that every $S_x$ preserves the weakening structure,
and the $\delta$S-condition amounts to saying that every $S_x$ preserves the generic elements.
Therefore condition~\eqref{def:Bsys:sub} in \cref{def:Bsys} unfolds to~\cite[3.1.2, 3.1.3, 3.2.2]{VV_B-systems}.
Finally, conditions~\eqref{def:Bsys:SxW},~\eqref{def:Bsys:SxId}, and~\eqref{def:Bsys:SIdW} in \cref{def:Bsys}
unfold to conditions~3.1.5,~3.2.3, and~3.2.4, respectively, in~\cite{VV_B-systems}.
\end{rem}

\begin{exa}[B-frames with structure and $D$-GATs]\label{eg:b-frames-structure}

Garner \cite{GARNER20151885} constructs an equivalence between the category of GATs and a category of algebras for a monad on B-frames.
We expect B-systems to be equivalent to Garner's algebras.
\end{exa}

\begin{lem}\label{lem:BsysForget}
The forgetful functor $\Bsys \to \Bfr$ is faithful.
\end{lem}

\begin{proof}
  This functor faithful because its action on morphisms only forgets a property.
\end{proof}

\begin{exa}\label{eg:b-sys-finsets}
The structures given in \cref{eg:b-frame-struct-finsets}
make the B-frame defined in \cref{eg:b-frame-finsets}
into a B-system as follows.
This B-system and the category of renamings from~\cref{eg:CE-fin-set} correspond to each other under the equivalence between B-systems and C-systems in \cref{thm:BeqvC}, in the sense that each of them is isomorphic to the image of the other (under the correct functor).
This B-system can thus be regarded as the B-system of renamings of an untyped theory.

Consider first homomorphism of B-frames
$S_y \colon \sys{B}/(k+1) \to \sys{B}/k$,
for $k \in \N$ and $y \in \std{k}$.
The homomorphism $S_y$ preserves the substitution structure if,
for every $n\in\N$, $x \in \std{k+1+n}$ and $j\in\N$,
the square
\begin{linenomath*}
\[\begin{tikzcd}[column sep=5em]
\std{k+1+n+1+j}	\ar[d,"S_{x,j}"{swap}] \ar[r,"S_{y,n+1+j}"]
&	\std{k+n+1+j}	\ar[d,"S_{S_{y,n}(x),j}"]
\\
\std{k+1+n+j}	\ar[r,"S_{y,n+j}"]	&	\std{k+n+j}
\end{tikzcd}\]
\end{linenomath*}
commutes.
This can be readily verified in the three cases
$x<k$, $x\jdeq k$ or $k<x<n+1+k$.
For example, in the last case $S_{y,n}(x) \jdeq x-1$ and
\begin{linenomath*}
\[\begin{tikzcd}[row sep=0pt,column sep=4em]
k	\ar[dr] \ar[dd,phantom,"+"{description}]
\\&
	k	\ar[dr] \ar[dd,phantom,"+"{description}]
\\
1	\ar[dr] \ar[dd,phantom,"+"{description}]
&&		k	\ar[dd,phantom,"+"{description}]
\\&
	1	\ar[ur,"y"{swap}] \ar[dd,phantom,"+"{description}]
\\
n	\ar[dr] \ar[dd,phantom,"+"{description}]
&&		n	\ar[dd,phantom,"+"{description}]
\\&
	n	\ar[ur] \ar[dd,phantom,"+"{description}]
\\
1	\ar[ur,"x"{swap}] \ar[dd,phantom,"+"{description}]	&&	j
\\&
	j	\ar[ur]
\\
j	\ar[ur]
\end{tikzcd}
\quad\jdeq\quad
\begin{tikzcd}[row sep=1em,column sep=5em]
k	\ar[dr] \ar[d,phantom,"+"{description}]
\\
1	\ar[r,"y"{swap}] \ar[d,phantom,"+"{description}]
&	k	\ar[d,phantom,"+"{description}]
\\
n	\ar[r] \ar[d,phantom,"+"{description}]
&	n	\ar[d,phantom,"+"{description}]
\\
1	\ar[ur,"x-1"{swap,near start}] \ar[d,phantom,"+"{description}]
&	j
\\
j	\ar[ur]
\end{tikzcd}
\quad\jdeq\quad
\begin{tikzcd}[row sep=0pt,column sep=4em]
k	\ar[dr] \ar[dd,phantom,"+"{description}]
\\&
	k	\ar[dr] \ar[dd,phantom,"+"{description}]
\\
1	\ar[ur,"y"{swap}] \ar[dd,phantom,"+"{description}]
&&		k	\ar[dd,phantom,"+"{description}]
\\&
	n	\ar[dr] \ar[dd,phantom,"+"{description}]
\\
n	\ar[ur] \ar[dd,phantom,"+"{description}]
&&		n	\ar[dd,phantom,"+"{description}]
\\&
	1	\ar[ur,"S_{y,n}(x)"{swap,near start}] \ar[dd,phantom,"+"{description}]
\\
1	\ar[ur] \ar[dd,phantom,"+"{description}]
&&		j
\\&
	j	\ar[ur]
\\
j	\ar[ur]
\end{tikzcd}\]
\end{linenomath*}
The homomorphism $S_y$
preserves the weakening structure if
for every $n,j\in\N$, the square
\begin{linenomath*}
\[\begin{tikzcd}[column sep=5em]
\std{k+1+n+1+j}	\ar[r,"S_{y,n+1+j}"]	&	\std{k+n+1+j}
\\
\std{k+1+n+j}	\ar[u,"W_{k+1+n,j}"] \ar[r,"S_{y,n+j}"]
&	\std{k+n+j}	\ar[u,"W_{k+n,j}"{swap}]
\end{tikzcd}\]
\end{linenomath*}
commutes. This is indeed the case:
\begin{linenomath*}
\[\begin{tikzcd}[row sep=0pt,column sep=4em]
&	k	\ar[dr]
\\
k	\ar[ur]
&&		k
\\&
	1	\ar[ur,"y"{swap}]
\\
1	\ar[ur]
&&		n
\\&
	n	\ar[ur]
\\
n	\ar[ur]
&&		1
\\&
	1	\ar[ur]
\\
j	\ar[dr]
&&		j
\\&
	j	\ar[ur]
\end{tikzcd}
\quad\jdeq\quad
\begin{tikzcd}[row sep=1em,column sep=5em]
k	\ar[r]
&		k
\\
1	\ar[ur,"y"{swap}]
&		n
\\
n	\ar[ur]
&		1
\\
j	\ar[r]
&		j
\end{tikzcd}
\quad\jdeq\quad
\begin{tikzcd}[row sep=0pt,column sep=4em]
k	\ar[dr]
&&		k
\\&
	k	\ar[ur]
\\
1	\ar[ur,"y"{swap}]
&&		n
\\&
	n	\ar[ur]
\\
n	\ar[ur]
&&		1
\\&
	j	\ar[dr]
\\
j	\ar[ur]
&&		j
\end{tikzcd}\]
\end{linenomath*}
Finally, for every $n \in \N$, the function
$S_{y,n+1} \colon \std{k+1+n+1} \to \std{k+n+1}$
preserves the maximum.
It follows that $S_y$ preserves the generic elements.

We have shown that $S_y$ is a pre-B-homomorphism.
We leave the verification that $W_n \colon \sys{B}/n \to \sys{B}/(n+1)$
is a pre-B-homomorphism to the reader
and consider instead the remaining three conditions of \cref{def:Bsys}.

Condition~\ref{def:Bsys:SxW} amounts to the commutativity of the left-hand diagram below, for every $n,j\in\N$ and $x \in \std{n}$.
Its commutativity is shown in the right-hand diagram.
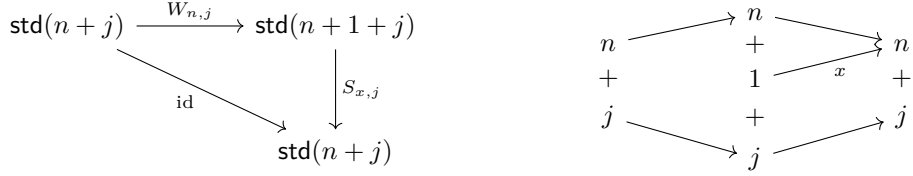
\begin{linenomath*}
\[\begin{tikzcd}[column sep=4em,row sep=3em]
\std{n+j}	\ar[dr,"\catid{}"{swap}] \ar[r,"W_{n,j}"]
&	\std{n+1+j}	\ar[d,"S_{x,j}"]
\\
&	\std{n+j}
\end{tikzcd}
\hspace{6em}
\begin{tikzcd}[row sep=0pt,column sep=4em]
&	n	\ar[dr] \ar[dd,phantom,"+"{description}]
\\
n	\ar[ur] \ar[dd,phantom,"+"{description}]
&&	n	\ar[dd,phantom,"+"{description}]
\\&
	1	\ar[ur,"x"{swap}] \ar[dd,phantom,"+"{description}]
\\
j	\ar[dr]	&&	j
\\&
	j	\ar[ur]
\end{tikzcd}\]
\end{linenomath*}

Condition~\ref{def:Bsys:SxId} holds since
$S_{x,0}(\delta_n) \jdeq S_{x,0}(n) \jdeq x$,
for every $n\in\N$ and $x \in \std{n}$.

Condition~\ref{def:Bsys:SIdW}  amounts to the commutativity of the left-hand diagram below, for every $n,j\in\N$.
Its commutativity is shown in the right-hand diagram.
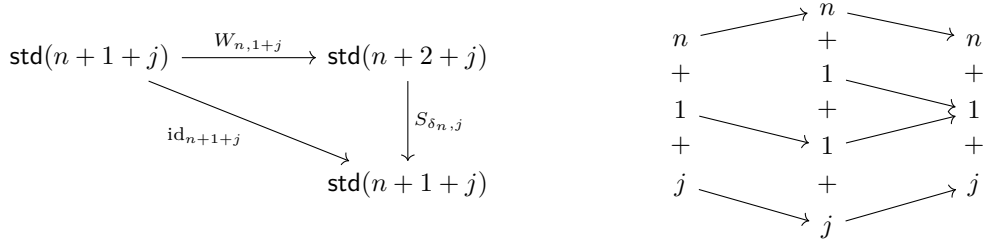
\begin{linenomath*}
\[\begin{tikzcd}[column sep=5em,row sep=3em]
\std{n+1+j}	\ar[dr,"\catid{n+1+j}"{swap}] \ar[r,"W_{n,1+j}"]
&	\std{n+2+j}	\ar[d,"S_{\delta_n,j}"]
\\
&	\std{n+1+j}
\end{tikzcd}
\hspace{6em}
\begin{tikzcd}[row sep=0pt,column sep=4em]
&	n	\ar[dr] \ar[dd,phantom,"+"{description}]
\\
n	\ar[ur] \ar[dd,phantom,"+"{description}]
		&&	n	\ar[dd,phantom,"+"{description}]
\\&
	1	\ar[dr] \ar[dd,phantom,"+"{description}]
\\
1	\ar[dr] \ar[dd,phantom,"+"{description}]
		&&	1	\ar[dd,phantom,"+"{description}]
\\&
	1	\ar[ur] \ar[dd,phantom,"+"{description}]
\\
j	\ar[dr]
		&&	j
\\&
	j	\ar[ur]
\end{tikzcd}\]
\end{linenomath*}
\end{exa}

\begin{exa}[B-systems and Generalized Algebraic Theories]

  Continuing~\cref{eg:b-frames-structure}, any Generalized Algebraic Theory (in Garner's taxonomy also known as $\{w,p,s\}$-GATs \cite[Definition~4]{GARNER20151885}) gives rise to a B-system.
  The axioms of~\cref{def:Bsys} follow mostly from the definition of substitution and the congruence rules that substitution satisfies.

  Composing Cartmell's equivalence of categories between contextual categories and GATs with our equivalence between B-systems and C-systems constructed in~\cref{ssec:eqv-b-c}, we later can establish a more precise relationship between B-systems and GATs, in the form of an equivalence of categories.
  
\end{exa}

\subsection{The category of E-systems}
\label{ssec:e-sys}

In \cref{ssec:b2e} we will show how for any B-frame
we get a category $\cat{F}$ with objects $(n,X)$ where $X\in B_n$.
As we saw in \cref{def:B-to-rttr}, the family of sets $B_n$ induces a tree,
with objects $(n,X)$, and $\cat{F}$ is the free category generated by this tree.
The sets $\tilde{B}_{n+1}$ then induce a family of sets of terms indexed by the morphisms of $\cat{F}$.
In particular for a morphism $(n+1,X)\to (n,\ft(X))$ we get a set of terms $\bd^{-1}(X)$.

In this section we will define the structure of a type theory directly on $\cat{F}$ of the kind that one gets when turning a B-system into a category. Such systems are called E-systems, and in \cref{ssec:b2e} we will show that the category of B-systems is equivalent to a subcategory of E-systems. Thus, E-systems can be seen as a generalisation of B-systems.

Just like B-systems (and different from C-systems), E-systems have an explicit structure for ``terms''.
Indeed, the first step towards the definition of E-system is that of a ``term structure'':

\begin{defi}
  A \define{category with term structure} is a category $\cat{F}$ equipped with a family of sets $(T(A))_{A \in \Mor{\cat{F}}}$ indexed by the morphisms $\Mor{\cat{F}}$ of $\cat{F}$. Given two categories $\cat{F}$ and $\cat{D}$ with term structure, a \define{functor with term structure} from $\cat{F}$ to $\cat{D}$ is a functor $F:\cat{F}\to\cat{D}$ equipped with a family of functions $T(A)\to T(F(A))$ for every morphism $A$ in $\cat{F}$.
\end{defi}

Any B-frame, and hence any B-system, generates a category with term structure;
details will be given in \cref{constr:BtoTrmStr}.

The identity functor with term structure $\catid{\cat{F}}:\cat{F}\to\cat{F}$ is the identity functor on $\cat{F}$ equipped with the identity functions $T(A)\to T(A)$ indexed by the morphisms $A$ in $\cat{F}$. Similarly, the composition $G\circ F$ of two functors $F$ and $G$ with term structure is defined to be the composition of the underlying functors, equipped with the composites
\begin{linenomath*}
\begin{equation*}
  \begin{tikzcd}
    T(A) \arrow[r] & T(F(A)) \arrow[r] & T(G(F(A))).
  \end{tikzcd}
\end{equation*}
\end{linenomath*}

\begin{defi}\label{def:slT}
Let $\cat{F}$ be a \scat with term structure and $\Gamma$ an object of $\cat{F}$.
The \define{slice term structure} on the \sslcat $\cat{F}/\Gamma$ 
is given by $T_{\cat{F}/\Gamma}(A)\jdeq T_\cat{F}(A)$.
\end{defi}

\begin{rem}\label{rmk:slT}
Every functor $F:\cat{F}\to\cat{F}'$ with term structure
gives rise to a functor with term structure $F/X:\cat{F}/X\to\cat{F}'/F(X)$.
\end{rem}

In order to illustrate the additional structure
that we shall consider on a category with term structure,
we introduce the following example.

\begin{exa}\label{eg:fin-set-trmstr}
Consider the the poset $(\N,\geq)$.
We write $(n,k) \colon n+k \geq n$ for arrows in $(\N,\geq)$.
Let $\mathcal{N}$ be the category with term structure 
which consists of the poset $(\N,\geq)$ and
the term structure given by $T(n,k) \defeq \Set([k],[n])$,
\ie the set of functions from the standard set with $k$ elements
to the standard set with $n$ elements.

This category with term structure, equipped with the additional structure described in this section, corresponds (up to isomorphism) to the B-system of renaming from \cref{eg:b-sys-finsets} under the equivalence between B-systems and stratified E-sytems in \cref{thm:BasE}.
\end{exa}

\begin{exa}\label{eg:esys-display-trmstr}
Consider a category $\cat{C}$ with a terminal object
together with a class of arrows $\cat{F}$
such that pullbacks along arrows in $\cat{F}$ exist in $\cat{C}$,
$\cat{F}$ is closed under composition and pullback,
and it contains all isomorphisms and arrows towards a terminal object.
This is a type-theoretic fibration category in the sense of~\cite{Shulman2015},
or a clan in the sense~\cite{Joyal2017clan},
or a display map category~\cite{Taylor1999} which models $\Sigma$-types in the sense of~\cite{DBLP:journals/mscs/North19}.

If we also denote by $\cat{F}$ the wide subcategory of $\cat{C}$
on the arrows that occur in $\cat{F}$,
then we can equip $\cat{F}$ with a term structure $T$ by requiring $T(A)$ to be the set of sections of $A$, that is, those arrows $x\colon\Gamma\to\Gamma.A$ in $\cat{C}$ such that $A\circ x=\catid{\Gamma}$.
\end{exa}

\subsubsection{Substitution systems}
\label{sssec:substsys}

Given a (strict) category $\cat{F}$, an object $\Gamma\in\cat{F}$ and an object $A\in\cat{F}/\Gamma$, we will write $\ctxext{\Gamma}{A}$ for the domain of $A$. In other words, $A$ is a morphism $\ctxext{\Gamma}{A}\to \Gamma$.

\begin{defi}
A \define{pre-substitution structure} on a \scat with term structure $\cat{F}$ consists
of a functor with term structure $S_x:\cat{F}/\ctxext{\Gamma}{A}\to\cat{F}/\Gamma$ for every $x\in T(A)$ and
$A\in\cat{F}/\Gamma$, such that $S_x(\catid{\ctxext{\Gamma}{A}})\jdeq\catid{\Gamma}$.

A \define{pre-substitution system} is a \scat with term structure together with a
pre-substitution structure. 
\end{defi}

\begin{defi}
A \define{pre-substitution homomorphism} $F:\cat{F}\to\cat{D}$ is a functor with term structure for
which the diagram
\begin{linenomath*}
\begin{equation*}
\begin{tikzcd}[column sep=large]
\cat{F}/\ctxext{\Gamma}{A}
\arrow[r,"F/\ctxext{\Gamma}{A}"]
\arrow[d,swap,"S_x"]
&
\cat{D}/F(\ctxext{\Gamma}{A})
\arrow[d,"S_{F(x)}"]
\\
\cat{F}/\Gamma
\arrow[r,swap,"F/\Gamma"]
&
\cat{D}/F(\Gamma)
\end{tikzcd}
\end{equation*}
\end{linenomath*}
commutes for every $x\in T(A)$ and $A\in\cat{F}/\Gamma$.
\end{defi}

\begin{defi}\label{def:sliceSsys}
Let $\cat{F}$ be a pre-substitution system and $\Gamma$ an object of  $\cat{F}$.
The \define{slice pre-substitution structure}
on the \sslcat with term structure $\cat{F}/\Gamma$ from \cref{def:slT}
is given by $S(\cat{F}/\Gamma)_x \jdeq S(\cat{F})_x$,
for every $A\in\cat{F}/\Gamma$, $P\in\cat{F}/\ctxext{\Gamma}{A}$
and $x \in T_{\cat{F}}(P)$.
\end{defi}

\begin{defi}\label{def:subst}
A \define{substitution system} is a pre-substitution system for which each
$S_x$
is a pre-substitution homomorphism. A \define{substitution homomorphism} is a
pre-substitution homomorphism between substitution systems.
\end{defi}

\begin{cor}\label{cor:sliceSsys}
For any object $\Gamma$ of a substitution system $\cat{F}$,
the slice pre-substitution system $\cat{F}/\Gamma$ from \cref{def:sliceSsys} is a substitution system, called the \define{slice substitution system} on $\Gamma$.
\end{cor}

\begin{rem}\label{rmk:substitution}
The condition that every $S_x$ is a substitution homomorphism, asserts that
the diagram
\begin{linenomath*}
\begin{equation*}
\begin{tikzcd}[column sep=large]
\cat{F}/\ctxext{\ctxext{\ctxext{\Gamma}{A}}{P}}{Q}
\arrow[r,"S_x/\ctxext{P}{Q}"]
\arrow[d,swap,"S_y"]
&
\cat{F}/\ctxext{\ctxext{\Gamma}{S_x(P)}}{S_x(Q)}
\arrow[d,"S_{S_x(y)}"]
\\
\cat{F}/\ctxext{\ctxext{\Gamma}{A}}{P}
\arrow[r,swap,"S_x/P"]
&
\cat{F}/\ctxext{\Gamma}{S_x(P)}
\end{tikzcd}
\end{equation*}
\end{linenomath*}
commutes for every $y\in T(Q)$.
\end{rem}

\begin{exa}\label{eg:fin-set-subst}
We can equip the category with term structure $\mathcal{N}$ from \cref{eg:fin-set-trmstr}
with a substitution structure as follows.
Consider the functor $-k \colon \N/(n+k) \to \N/n$
that maps $(n+k+j,l)$ to $(n+j,l)$.
It preserves terminal objects
since an arrow $(m,i)$ is an identity if and only if $i=0$.
Given $(n,k) \colon n+k \geq n$ and a function $f \colon [k] \to [n]$,
define $S_f \colon \mathcal{N}/(n+k) \to \mathcal{N}/n$ as the functor $-k$
together with functions $T(n+k+j,l) \to T(n+j,l)$ defined by postcomposition
\begin{linenomath*}
\[
\begin{tikzcd}[row sep=2ex,column sep=3ex]
{[l]}	\ar[dd,"h"{swap},""{name=D}]
&&	{[l]}	\ar[dd,"\,{S_f(h)}"{swap,name=C}] \ar[dr,"h"]
&\\
&&&	{[n+k+j]}	\ar[dl,"{[\catid{n},f] + \catid{j}}"]
\\
{[n+k+j]}	&&	{[n+j]}	\ar[mapsto,from=D,to=C]	&
\end{tikzcd}
\]
\end{linenomath*}
where $[\catid{n},f]$ is the function given by the universal property of the coproduct $[n] \leftarrow [n+k] \to [k]$ in $\Set$,
and similarly for $[\catid{n},f] + \catid{j}$.

The fact that $S_f$ is a pre-substitution homomorphism follows from the fact
that postcomposition distributes on $[\blank,\blank]$ as shown below:
given $g \colon [l] \to [n+k+j]$, then
\begin{linenomath*}
\[
S_f/(n+k,j) \circ S_g \jdeq S_{S_f(g)} \circ S_f/(n+k,j+l)
\]
\end{linenomath*}
since
\begin{linenomath*}
\begin{align*}
\left( [\catid{n},f] + \catid{j} \right) [\catid{n+k+j},g]
&\jdeq
[ [\catid{n},f] + \catid{j}, S_f(g) ]
\\&\jdeq
[\catid{n+j}, S_f(g)] \left( [\catid{n},f] + \catid{j+l} \right).
\end{align*}
\end{linenomath*}
\end{exa}

\begin{exa}\label{eg:esys-display-subst}
Given a clan $(\cat{C},\cat{F})$,
consider the induced category with term structure from \cref{eg:esys-display-trmstr},
where $T(A)$ is the set of all sections of $A\in\cat{F}/\Gamma$.

Say that a choice of pullbacks of arrows in $\cat{F}$ is locally functorial if,
for every $f\colon\Delta\to\Gamma$ in $\cat{C}$, we have
$f^*(\catid{\Gamma}) \jdeq \catid{\Delta}$, $f^{\catid{\Gamma}}=f$,
and, for every composable $A,P$ in $\cat{F}$, we also have
$f^*(A\circ P) \jdeq f^*\!A\circ (f^A)^*P$ and $f^{A\circ P} \jdeq (f^A)^P$,
where $f^*\!A$ and $f^A$ are the first and second leg, respectively, of the chosen pullback of $A$ along $f$.

Every choice of pullbacks of arrows in $\cat{F}$ uniquely determines a choice of pullbacks of sections of arrows in $\cat{F}$.
It follows that the category with term structure from \cref{eg:esys-display-trmstr} can be equipped with a pre-substitution structure, by setting $S_x\coloneqq x^*$ for every $x\in T(A)$.
This pre-substitution structure gives rise to a substitution system if the choice of pullbacks is functorial,
\ie~such that $(f\circ g)^*\!A \jdeq g^*(f^*\!A)$
and $(f\circ g)^A \jdeq f^A\circ g^{f^*A}$.
Note that every choice of pullbacks can be made normal,
\ie~such that $\catid{\Gamma}^{\,*}A \jdeq A$ and $\catid{\Gamma}^A \jdeq \catid{\Gamma.A}$
(this holds true more generally for every cleavage on a Grothendieck fibration).
\end{exa}

The following is not an intended example, but rather a surprising one.

\begin{exa}\label{eg:esys-as-group}
  Consider a group $G$. A term structure on $G$ consists of a set $T(g)$ for every element $g$ of $G$.

A pre-substitution structure on $G$ consists of a functor with term structure $S_x: G / \bullet \to G / \bullet$ (where $\bullet$ denotes the only object in $G$ viewed as a category) for every $g \in G$ and every $x \in T(g)$ such that $S_x(\catid{\bullet}) = \catid{\bullet}$. One can show that such functors $S_x: G / \bullet \to G / \bullet$ correspond to functions $G \to G$ which preserve the identity, so a pre-substitution structure amounts to functions $S_x : G \to G$ for every $g \in G$, $x \in T(g)$ preserving the identity together with functions $S_x: T(h) \to T(S_x(h))$ for every $g,h \in G$, $x \in T(g)$.

A substitution structure $T$ on $G$ is a pre-substitution structure $S$ as described above such that the following diagrams commute for all $g,h,k \in G$, $x \in T(g)$, and $y \in T(h)$.
\begin{linenomath*}
\[
   \begin{tikzcd}
      G \ar[r,"S_x"] \ar[d,"S_y"] & G \ar[d,"S_{S_x(y)}"]
      \\ 
      G \ar[r,"S_x"] & G
   \end{tikzcd}
   \hspace{5em}
   \begin{tikzcd}
    T(k) \ar[r,"S_x"] \ar[d,"S_y"] & T(S_x k) \ar[d,"S_{S_x(y)}"]
    \\ 
    T(S_y k) \ar[r,"S_x"] & T(S_x S_y k)
 \end{tikzcd}
\]
\end{linenomath*}
Now for a particular example, suppose that each $T(g)$ is $\mathrm{Aut}(G)$, that each $S_x: G \to G$ is just the automorphism $x$, and that each $S_x : T(h) \to T(S_x h)$ takes $y \in T(h)$ to $x y x^{-1}$. Then we find indeed that the first diagram commutes since $(xy x^{-1}) x = xy$ for all $x \in S_x = \mathrm{Aut}(G)$ and all $y \in S_y = \mathrm{Aut}(G)$. The second diagram commutes since $(xy x^{-1}) x z x^{-1}(xy x^{-1})^{-1} = x y z y^{-1} x^{-1}$ for all $x \in T(g) = \mathrm{Aut}(G)$, $y \in T(h) = \mathrm{Aut}(G)$, and $z \in T(k) = \mathrm{Aut}(G)$.
\end{exa}

\subsubsection{Weakening systems}
\label{sssec:weaksys}

\begin{defi}\label{def:preweak}
Consider a category $\cat{F}$ with term structure $T$. 
A \define{pre-weakening structure} on $\cat{F}$ is a family of functors with term structure $W_A:\cat{F}/\Gamma\to \cat{F}/\ctxext{\Gamma}{A}$ indexed by the morphisms $A:\ctxext{\Gamma}{A}\to \Gamma$ in $\cat{F}$ such 
that
\begin{enumerate}
\item\label{def:preweak:id}
$W_{\catid{\Gamma}}\jdeq \catid{\cat{F}/\Gamma}$ for every object $\Gamma\in\cat{F}$.
\item\label{def:preweak:funct}
$W_{A\circ P}\jdeq W_P\circ W_A$ for every $P\in\cat{F}/\ctxext{\Gamma}{A}$ and $A\in\cat{F}/\Gamma$.
\item\label{def:preweak:presterm}
$W_A$ strictly preserves the final object, i.e., $W_A(\catid{\Gamma})\jdeq \catid{\ctxext{\Gamma}{A}}$.
\end{enumerate}
A \define{pre-weakening system} is a \scat with term structure equipped with a pre-weakening structure.
\end{defi}

\begin{defi}
A \define{pre-weakening homomorphism} $F:\cat{F}\to\cat{D}$ between pre-weakening systems
is a functor $F:\cat{F}\to\cat{D}$ with term structure such that the square
\begin{linenomath*}
\begin{equation*}
\begin{tikzcd}[column sep=large]
\cat{F}/\ctxext{\Gamma}{A}
\arrow[r,"F/\ctxext{\Gamma}{A}"]
&
\cat{D}/F(\ctxext{\Gamma}{A})
\\
\cat{F}/\Gamma
\arrow[u,"W_A"]
\arrow[r,swap,"F/\Gamma"]
&
\cat{D}/F(\Gamma)
\arrow[u,swap,"W_{F(A)}"]
\end{tikzcd}
\end{equation*}
\end{linenomath*}
of functors with term structure commutes for any $A\in\cat{F}/\Gamma$.
\end{defi}

\begin{defi}\label{def:sliceWsys}
Let $\cat{F}$ be a pre-weakening system and $\Gamma$ an object of $\cat{F}$.
The \define{slice pre-weakening system} on the \sslcat with term structure from \cref{def:slT}
$\cat{F}/\Gamma$ is given by $W(\cat{F}/\Gamma)_P\jdeq W(\cat{F})_P$ for every $P\in\cat{F}/\ctxext{\Gamma}{A}$ and $A\in\cat{F}/\Gamma$.
\end{defi}

\begin{defi}
A \define{weakening system} is a pre-weakening system $\cat{F}$ such that $W_A$ is a pre-weakening homomorphism for every morphism $A$ in $\cat{F}$. A \define{weakening homomorphism} is a pre-weakening 
homomorphism between weakening systems.
\end{defi}

\begin{rem}
  The condition that every $W_A$ is a pre-weakening homomorphism implies that the square
  \begin{linenomath*}
  \begin{equation*}
    \begin{tikzcd}[column sep=huge]
      \cat{F}/\ctxext{\ctxext{\Gamma}{B}}{Q}
      \arrow[r,"W_A/\ctxext{B}{Q}"]
      &
      \cat{F}/\ctxext{\ctxext{\Gamma}{A}}{W_A(\ctxext{B}{Q})}
      \\
      \cat{F}/\ctxext{\Gamma}{B}
      \arrow[u,"W_Q"]
      \arrow[r,swap,"W_A/B"]
      &
      \cat{F}/\ctxext{\ctxext{\Gamma}{A}}{W_A(B)}
      \arrow[u,swap,"W_{W_A(Q)}"]
    \end{tikzcd}
  \end{equation*}
  \end{linenomath*}
  commutes for each $A,B\in\cat{F}/\Gamma$ and $Q\in\cat{F}/\ctxext{\Gamma}{B}$. On objects, this 
  property asserts that for any $k\in \cat{F}/\mathrm{E}$, the dotted arrows in the diagram
  \begin{linenomath*}
\begin{equation*}
\begin{tikzcd}[column sep=tiny]
\bullet \arrow[dr,densely dotted,swap,near start,"(W_A/\ctxext{B}{Q})(W_Q(R))"] & & \bullet \arrow[dl,densely dotted,near start,"W_{(W_A/B)(Q)}((W_A/B)(R))"] & & & & & & & & & & \bullet \arrow[dl,"W_Q(R)"] \\
& \bullet \arrow[dr,swap,"(W_A/B)(Q)"] & & \bullet \arrow[dl,"(W_A/B)(R)"] & & & & & & & & \bullet \arrow[dr,swap,"Q"] & & \bullet \arrow[dl,"R"] \\
& & \bullet \arrow[d,swap,"W_A(B)"] & & & & & & & & & & \ctxext{\Gamma}{B} \arrow[d,"B"] \\
& & \ctxext{\Gamma}{A} \arrow[rrrrrrrrrr,swap,"A"] & & & & & & & & & & \Gamma
\end{tikzcd}
\end{equation*}
\end{linenomath*}
are equal.

A useful special case of this property is where $B\jdeq \catid{\Gamma}$. Thus, if $W$ is
a weakening system, then the diagram
\begin{linenomath*}
\begin{equation*}
\begin{tikzcd}[column sep=huge]
\cat{F}/\ctxext{\Gamma}{C}
\arrow[r,"W_A/C"]
&
\cat{F}/\ctxext{\ctxext{\Gamma}{A}}{W_A(C)}
\\
\cat{F}/\Gamma
\arrow[u,"W_C"]
\arrow[r,swap,"W_A"]
&
\cat{F}/\ctxext{\Gamma}{A}
\arrow[u,swap,"W_{W_A(C)}"]
\end{tikzcd}
\end{equation*}
\end{linenomath*}
commutes for every $A,C\in\cat{F}/\Gamma$. In particular, we see
that $W_A(W_C(D))\jdeq W_{W_A(C)}(W_A(D))$ for any $D\in\cat{F}/\Gamma$,
i.e.~that weakening is a self-distributive operation. 
\end{rem}

\begin{cor}\label{cor:sliceWsys}
For any object $\Gamma$ of a weakening system $\cat{F}$,
the slice pre-weakening system on $\cat{F}/\Gamma$ from \cref{def:sliceWsys} is a weakening system, called the \define{slice weakening system} on $\Gamma$.
\end{cor}

\begin{exa}\label{eg:fin-set:weak}
Consider the category with term structure $\mathcal{N}$ from \cref{eg:fin-set-trmstr}.
We can equip $\mathcal{N}$ with a weakening structure as follows.
Consider the functor $+k \colon \N/n \to \N/(n+k)$
that maps $(n+j,l)$ to $(n+k+j,l)$.
It preserves terminal objects as in \cref{eg:fin-set-subst}.
Given $(n,k) \colon n+k \geq n$,
define $W_{n,k} \colon \mathcal{N}/n \to \mathcal{N}/(n+k)$ as the functor $+k$
together with functions $T(n+j,l) \to T(n+k+j,l)$ defined by postcomposition
\begin{linenomath*}
\[
\begin{tikzcd}[row sep=2ex,column sep=3ex]
{[l]}	\ar[dd,"h"{swap},""{name=D}]
&&	{[l]}	\ar[dd,"\,{W_{n,k}(h)}"{swap,name=C}] \ar[dr,"h"]
&\\
&&&	{[n+j]}	\ar[dl,"{i_n^{n+k} + \catid{j}}"]
\\
{[n+j]}	&&	{[n+k+j]}	\ar[mapsto,from=D,to=C]	&
\end{tikzcd}
\]
\end{linenomath*}
where $i_n^{n+k} \colon [n] \to [n+k]$ is the initial-segment inclusion and
$i_n^{n+k} + \catid{j}$ is the function given by the universal property of the coproduct $[n] \leftarrow [n+j] \to [j]$ in $\Set$.

The fact that $W_{n,k}$ is a pre-weakening homomorphism follows from 
the fact that initial-segment inclusions factor uniquely into inclusions whose images have codimension 1:
given $(n+j,l)$ in $\mathcal{N}/n$, then
\begin{linenomath*}
\[
W_{n,k}/(n,j+l) \circ W_{n+j,l}	\jdeq W_{W_{n,k}(n+j,l)} \circ W_{n,k}/(n,j)
\]
\end{linenomath*}
since
\begin{linenomath*}
\[
(i_n^{n+k} + \catid{j+l}) i_{n+j}^{n+j+l} \jdeq
i_{n+k+j}^{n+k+j+l} (i_n^{n+k} + \catid{j}).
\]
\end{linenomath*}
\end{exa}

\begin{exa}\label{eg:esys-display-weak}
Given a clan $(\cat{C},\cat{F})$,
the induced category with term structure from \cref{eg:esys-display-trmstr}
can be made into a weakening system if there is a choice of pullbacks in $\cat{F}$ which is functorial and locally functorial in the sense of \cref{eg:esys-display-subst}.

Note that, since $\cat{F}$ is closed under pullbacks,
every clan with a choice of pullbacks that gives rise to a substitution structure also has an induced weakening structure.
\end{exa}

\begin{exa}\label{eg:esys-as-group-weak}
Consider the situation of \cref{eg:esys-as-group} above where the underlying category is a group $G$ with term structure $S$.

A pre-weakening structure on $G$ is a family of functions $W_g : G \to G$ for each $g \in G$ which preserves the identity in each coordinate (i.e. $W_e(g) = W_g(e) = g$ for any $g \in G$) and where $W_{hg} = W_g \circ W_h$ together with term structure $W_g : T(h) \to T(W_g(h))$ for any $g, h \in G$.

If each $W_g: G \to G$ is a group homomorphism, this structure is a weakening system when the following diagrams commute 
for every $g,h,k \in G$.
\begin{linenomath*}
\begin{equation*}
\begin{tikzcd}
G \arrow[r,"W_h"]
&
G
\\
G \arrow[u,"W_g"]
\arrow[r,swap,"W_h"]
&
G \arrow[u,swap,"W_{W_h(g)}"]
\end{tikzcd} \hspace{5em}
\begin{tikzcd}
  T(W_g k) \arrow[r,"W_h"]
  &
  T(W_{gh} k)
  \\
  T(k) \arrow[u,"W_g"]
  \arrow[r,swap,"W_h"]
  &
  T(W_h k) \arrow[u,swap,"W_{W_h(g)}"]
  \end{tikzcd}
\end{equation*}
\end{linenomath*}
Now consider the more particular example discussed in \cref{eg:esys-as-group}, where $G$ is still an arbitrary group, but $T(g) = \mathrm{Aut}(G)$ for all $g \in G$. We can let each $W_g : G \to G$ be $\phi_g$, the conjugation automorphism sending $h$ to $ghg^{-1}$, and we can let each $W_g : T(k) \to T(k)$ be `conjugation by conjugation' taking each automorphism $x \in T(k)$ to $\phi_g x \phi_g^{-1}$. Since $\phi_h \phi_g = \phi_{\phi_h(g)} \phi_h$
, the left-hand diagram above commutes, and using that equation we find that 
$\phi_h \phi_g (-) \phi_g^{-1} \phi_h^{-1} = \phi_{\phi_h(g)} \phi_h (-) \phi_{\phi_h(g)}^{-1} \phi_h^{-1}$ so the right-hand diagram commutes.

\end{exa}

\subsubsection{Projection systems}
\label{sssec:projsys}

\begin{defi}
A \define{pre-projection system} is a pre-weakening system $\cat{F}$ equipped with an element
$\tfid{A}\in T(W_A(A))$ for every $A\in\cat{F}/\Gamma$ and $\Gamma\in\cat{F}$. 
\end{defi}

\begin{defi}
A \define{pre-projection homomorphism} $F:\cat{F}\to\cat{D}$ is a pre-weakening homomorphism for which
\begin{linenomath*}
\begin{equation*}
F(\tfid{A})\jdeq \tfid{F(A)}
\end{equation*}
\end{linenomath*}
for every $A\in\cat{F}/\Gamma$ and $\Gamma\in\cat{F}$.
\end{defi}

\begin{defi}\label{def:slicePsys}
Let $\cat{F}$ be a pre-projection system and $\Gamma$ an object of $\cat{F}$.
The \define{slice pre-projection structure} on $\cat{F}/\Gamma$
is given by the slice pre-weakening structure in \cref{def:sliceWsys}
together with $\tfid{P}^{\Gamma} \coloneqq \tfid{P}$,
for every $P \in \cat{F}/\ctxext{\Gamma}{A}$ and $A \in \cat{F}/\Gamma$.
\end{defi}

\begin{defi}
A \define{projection system} is a pre-projection system for which every $W_A$ is
a pre-projection homomorphism. A \define{projection homomorphism} is a pre-projection homomorphism
between projection systems.
\end{defi}

\begin{cor}\label{cor:slicePsys}
For any object $\Gamma$ of a projection system $\cat{F}$,
the slice pre-projection system on $\cat{F}/\Gamma$ from \cref{def:slicePsys} is a projection system, called the \define{slice projection system} on $\Gamma$.
\end{cor}

\begin{exa}\label{eg:fin-set-proj}
Consider the weakening system on $\mathcal{N}$ from \cref{eg:fin-set:weak}.
We can equip it with a projection structure defining,
for every $(n,k)$ in $(\N,\geq)$,
the element $\tfid{n,k} \in T(W_{n,k}(n,k)) \jdeq \Set([k],[n+k])$
to be the final-segment inclusion
\begin{linenomath*}
\[\begin{tikzcd}[column sep=4em]
{[k]}	\ar[r,"i_k^{n+k}"]	&	{[n+k]}
\end{tikzcd}\]
\end{linenomath*}
The fact that each $W_{n,j} \colon \mathcal{N}/n \to \mathcal{N}/(n+j)$ is a projection homomorphism is readily verified:
\begin{linenomath*}
\begin{align*}
W_{n,j}(\tfid{n+m,k})	&\jdeq
(i_n^{n+j} + \catid{m+k}) i_k^{n+m+k}
\\&\jdeq
i_k^{n+m+j+k}
\\&\jdeq
\tfid{W_{n,j}(n+m,k)}.
\end{align*}
\end{linenomath*}
\end{exa}

\begin{exa}\label{eg:esys-display-proj}
Given a clan $(\cat{C},\cat{F})$,
the weakening system from \cref{eg:esys-display-weak} can be upgraded to a projection system by defining the element $\tfid{A}\in T(W_A(A))$
to be the unique section of the pullback of $A$ along itself induced by the pair of identity arrows on $\Gamma.A$.

Note that no additional condition on the choice of pullbacks has to be imposed, besides those mentioned in \cref{eg:esys-display-weak}.
In particular, every clan with a choice of pullbacks that gives rise to a weakening structure has also an induced projection structure.
\end{exa}

\begin{exa}\label{eg:esys-as-group-proj}
  Consider the particular example discussed in \cref{eg:esys-as-group-weak} where the underlying category is an arbitrary group $G$, each $T(g)$ is the set of automorphisms of $G$, and $W_g$ is conjugation by $G$ both on elements of $G$ and terms (automorphisms of $G$).

  A pre-projection system consists of an element $\tfid{g} \in \mathrm{Aut}(G)$ for every $g \in G$. We will let $\tfid{g}$ be the identity automorphism on $G$.

  For a projection system on a group $G$, we need $W_g (\tfid{h}) = \tfid{W{g}(h)}$ for every $g,h \in G$. In our particular example, this means $g 1_G g^{-1} = 1_G$, which holds.
\end{exa}

\subsubsection{The definition of E-systems}
\label{sssec:esys}

We can now give the definition of E-systems.

\begin{defi}\label{def:preEsys}
A \define{pre-E-system} \sys{E} is a \scat $\cat{F}$ with term structure
equipped with a chosen terminal object $\Eroot{}$ in $\cat{F}$, the structure of a pre-substitution system $S$,
the structure of a pre-weakening system $W$, and the structure of a pre-projection system $\tfid{}$.
\end{defi}

\begin{defi}\label{def:preEhom}
A \define{pre-E-homomorphism} from $\sys{E}$ to $\sys{D}$ is a functor
$H:\Efam{E}\to \Efam{D}$ between the underlying categories
with term structure such that $F(\Erootsys{E}) = \Erootsys{D}$, which is a pre-substitution homomorphism, a pre-weakening homomorphism, and a pre-projection homomorphism.
\end{defi}

\begin{defi}\label{def:sliceEsys}
For every object $\Gamma$ in a pre-E-system $\sys{E}$,
the \define{slice pre-E-system} $\sys{E}/\Gamma$ on the slice category with term structure $\Efam{E}/\Gamma$
is given by the slice structures from \cref{def:sliceSsys,def:sliceWsys,def:slicePsys},
and the identity on $\Gamma$ as terminal object.
\end{defi}

\begin{defi}\label{def:Esys}
An \define{E-system} is a pre-E-system \sys{E} such that
\begin{enumerate}
\item\label{def:Esys:sub}
  each $S_x$ is a pre-E-homomorphism, %
\item\label{def:Esys:weak}
  each $W_A$ is a pre-E-homomorphism, %
\item\label{def:Esys:SxW} %
  $S_x\circ W_A\jdeq \catid{\sys{E}/\Gamma}$ for any $x\in T(A)$ and $A\in\cat{F}/\Gamma$,
\item\label{def:Esys:SxId}
  $S_x(\tfid{A})\jdeq x$ for any $x\in T(A)$ and $A\in\cat{F}/\Gamma$, and
\item\label{def:Esys:SIdW} %
  $S_{\tfid{A}}\circ W_A/A\jdeq \catid{\sys{E}/\ctxext{\Gamma}{A}}$ for any $A\in\cat{F}/\Gamma$.
\end{enumerate}

An \define{E-homomorphism} $H:\mathbb{E}\to\mathbb{D}$
is a pre-E-homomorphism from an E-system $\mathbb{E}$ to an E-system $\mathbb{D}$.
We write $\Esys$ for the category of E-systems and E-homomorphisms.
\end{defi}

\begin{rem}
The condition that each $W_A$ is a substitution homomorphism asserts that
the diagram
\begin{linenomath*}
\begin{equation*}
\begin{tikzcd}[column sep=huge]
\cat{F}/\ctxext{\ctxext{\Gamma}{B}}{Q}
\arrow[r,"W_A/\ctxext{B}{Q}"]
\arrow[d,swap,"S_y"]
&
\cat{F}/\ctxext{\ctxext{\ctxext{\Gamma}{A}}{W_A(B)}}{W_A(Q)}
\arrow[d,"S_{W_A(y)}"]
\\
\cat{F}/\ctxext{\Gamma}{B}
\arrow[r,swap,"W_A/B"]
&
\cat{F}/\ctxext{\ctxext{\Gamma}{A}}{W_A(B)}
\end{tikzcd}
\end{equation*}
\end{linenomath*}
of functors with term structure commutes for every 
$Q\in\cat{F}/\ctxext{\Gamma}{B}$, $B\in\cat{F}/\Gamma$ and each $y\in T(Q)$.

Likewise, the condition that each $S_x$ is a weakening homomorphism
asserts that the diagram
\begin{linenomath*}
\begin{equation*}
\begin{tikzcd}[column sep=huge]
\cat{F}/\ctxext{\ctxext{\Gamma}{A}}{P}
\arrow[r,"S_x/P"]
\arrow[d,swap,"W_Q"]
&
\cat{F}/\ctxext{\Gamma}{S_x(P)}
\arrow[d,"W_{S_x(Q)}"]
\\
\cat{F}/\ctxext{\ctxext{\ctxext{\Gamma}{A}}{P}}{Q}
\arrow[r,swap,"S_x/\ctxext{P}{Q}"]
&
\cat{F}/\ctxext{\ctxext{\Gamma}{S_x(P)}}{S_x(Q)}
\end{tikzcd}
\end{equation*}
\end{linenomath*}
of functors with term structure commutes for every 
$Q\in\cat{F}/\ctxext{\ctxext{\Gamma}{A}}{P}$.
\end{rem}

\begin{cor}\label{cor:sliceEsys}
For any object $\Gamma$ of a E-system $\cat{F}$,
the slice pre-E-system on $\cat{F}/\Gamma$ from \cref{def:sliceEsys} is an E-system, called the \define{slice E-system} on $\Gamma$.
\end{cor}

\begin{exa}\label{eg:fin-set-Esys}
We can finally show that the category with term structure $\mathcal{N}$ from \cref{eg:fin-set-trmstr} can be equipped with the structure of an E-system.
It can be turned into a pre-E-system because of
\cref{eg:fin-set-subst,eg:fin-set:weak,eg:fin-set-proj}.
The terminal object is $[0]$.
Conditions \ref{def:Esys:sub}~and~\ref{def:Esys:weak} of \cref{def:Esys} are left to the reader.
The other ones are verified as follows:
\begin{enumerate}
\item[\ref{def:Esys:SxW}.]
Given $f \colon [k] \to [n]$,
it is $S_f \circ W_{n,k} \jdeq \catid{\mathcal{N}/n}$
since $[\catid{n},f] i_n^{n+k} \jdeq \catid{n}$.
\item[\ref{def:Esys:SxId}.]
Given $f \colon [k] \to [n]$,
it is $S_f(\tfid{n,k}) \jdeq [\catid{n},f] i_k^{n,k} \jdeq f$.
\item[\ref{def:Esys:SIdW}.]
Given $(n,k)$ in $\mathcal{N}$,
it is $S_{\tfid{n,k}} \circ W_{n,k}/(n,k) \jdeq \catid{\mathcal{N}/(n+k)}$
since $[\catid{n+k},i_k^{n+k}] (i_n^{n+k} + \catid{k}) \jdeq \catid{n+k}$.
\end{enumerate}
\end{exa}

\begin{exa}\label{eg:esys-display-Esys}
Given a clan $(\cat{C},\cat{F})$ together with a functorial and locally functorial choice of pullbacks of arrows in $\cat{F}$ along arrows in $\cat{C}$ in the sense of \cref{eg:esys-display-subst},
the induced category with term structure from \cref{eg:esys-display-trmstr} can be made into an E-system as follows.
It is a pre-E-system because of \cref{eg:esys-display-subst,eg:esys-display-weak,eg:esys-display-proj}.
Conditions~\eqref{def:Esys:sub} and~\eqref{def:Esys:weak} are satisfied
since weakening is a particular case of substitution.
Condition~\eqref{def:Esys:SxW} follows from $A\circ x=\catid{}$
and the functoriality conditions on the choice of pullbacks.
In particular, the left-hand vertical square in the commutative diagram below is a pullback,
since so is the right-hand one.
It follows
that the upper square is a pullback too.
Therefore condition~\eqref{def:Esys:SxId} is also satisfied.
Finally, condition~\eqref{def:Esys:SIdW} follows from the commutativity of the upper triangle involving $\tfid{A}$ in the commutative diagram below,
together, again, with the functoriality conditions on the choice of pullbacks.
\[\begin{tikzcd}
	\Gamma && {\Gamma.A}
	&&[3ex]\\
	& {\Gamma.A} && {\Gamma.A.W_AA} & {\Gamma.A}
	\\[3ex]
	& \Gamma && {\Gamma.A} & \Gamma
	\arrow["x", from=1-1, to=1-3]
	\arrow["x"{pos=0.6}, from=1-1, to=2-2]
	\arrow["{\catid{}}"'{near end}, from=1-1, to=3-2, bend right=4ex]
	\arrow["{\tfid{A}}"{description}, from=1-3, to=2-4]
	\arrow["{\catid{}}"{pos=.85}, from=1-3, to=2-5, bend left=3ex]
	\arrow["{\catid{}}"'{near end}, from=1-3, to=3-4, bend right=4ex]
	\arrow[from=2-2, to=2-4, crossing over]
	\arrow["A"{near end}, from=2-2, to=3-2]
	\arrow[from=2-4, to=2-5]
	\arrow["{W_AA}"{near end}, from=2-4, to=3-4]
	\arrow["A"{near end}, from=2-5, to=3-5]
	\arrow["x", from=3-2, to=3-4]
	\arrow["A", from=3-4, to=3-5]
	\arrow["{\catid{}}"'{pos=.8}, from=2-2, to=2-5, bend right=4ex, crossing over]
	\arrow["{\catid{}}"'{pos=.8}, from=3-2, to=3-5, bend right=4ex]
\end{tikzcd}\]
\end{exa}

\begin{exa}
  Consider the situation in \cref{eg:esys-as-group-proj} where our underlying category is an arbitrary group $G$, the terms of each $g \in G$ are $\mathrm{Aut}$, substitution $S_x : G \to G$ is given by the automorphism $x$ itself and substitution $S_x : T(h) \to T(S_x h)$, weakening $W_g : G \to G$, and weakening $W_g : T(h) \to T(W_g h)$ are given by conjugation.

  Understood as a category, $G$ does not have a terminal object (unless it is trivial), but we can still understand the conditions of \cref{def:Esys}. The condition~\ref{def:Esys:weak} in that definition means that the following diagrams must commute for $g,h,k \in G$ and $x \in T(h)$
  \begin{linenomath*}
  \[
     \begin{tikzcd}
        G \ar[r,"W_g"] \ar[d,"S_x"] & G \ar[d,"S_{W_g(x)}"]
        \\ 
        G \ar[r,"W_g"] & G
     \end{tikzcd}
     \hspace{5em}
     \begin{tikzcd}
      T(k) \ar[r,"W_g"] \ar[d,"S_x"] & T(W_g k) \ar[d,"S_{W_g(x)}"]
      \\ 
      T(S_x k) \ar[r,"W_g"] & T(S_{W_g(x)} W_g k)
   \end{tikzcd}
  \]
  \end{linenomath*}
Since 
$\phi_g x = \phi_g x \phi_g^{-1} \phi_g$, the left-hand diagram above commutes, and since $\phi_g x (-) x^{-1} \phi_g^{-1} = (\phi_g x \phi_g^{-1}) \phi_g (-) \phi_g^{-1} (\phi_g x \phi_g^{-1} )^{-1}$, the right-hand square above commutes.

The condition~\ref{def:Esys:sub} in \cref{def:Esys} means that the following diagrams must commute for $g,h,k \in G$ and $x \in T(h)$.
\begin{linenomath*}
\[
   \begin{tikzcd}
      G \ar[r,"S_x"] \ar[d,"W_g"] & G \ar[d,"W_{S_x(g)}"]
      \\ 
      G \ar[r,"S_x"] & G
   \end{tikzcd}
   \hspace{5em}
   \begin{tikzcd}
    T(k) \ar[r,"S_x"] \ar[d,"W_g"] & T(S_x k) \ar[d,"W_{S_x(g)}"]
    \\ 
    T(W_g k) \ar[r,"S_x"] & T(S_x W_g k)
 \end{tikzcd}
\]
\end{linenomath*}
Since $x \phi_g = \phi_{x(g)} x$, the left-hand diagram commutes, and since 
then $ x \phi_g (-) \phi_g^{-1} x^{-1} = \phi_{x(g)} x (-) x^{-1} \phi_{x(g)}^{-1} $, the right-hand diagram commutes.

Condition~\ref{def:Esys:SxW} does not hold since (on $G$) $S_x \circ W_g = x \phi_g$, and in general this is not the identity.

Condition~\ref{def:Esys:SxId} does not hold since $S_x \tfid{g} = x 1_G x^{-1}$ which is not in general $x$.

Condition~\ref{def:Esys:SIdW} does not hold since (on $G$) $S_{\tfid{g}}  W_g = 1_G \phi_g = \phi_g$ which is not the identity in general.
\end{exa}

We introduce more convenient notation for weakening and substitution.

\begin{defi}\label{def:not-weak}
Let $A\in\cat{F}/\Gamma$. Recall that $W_A:\cat{F}/\Gamma\to\cat{F}/\ctxext{\Gamma}{A}$ acts on objects,
morphisms and terms. We introduce the infix form of weakening by $A\in\cat{F}/\Gamma$ to be
$\ctxwk{A}{\blank}$. Thus, we will write
\begin{linenomath*}
\begin{align*}
\ctxwk{A}{B} & \defeq W_A(B) & & \text{for $B\in\cat{F}/\Gamma$} \\
\ctxwk{A}{Q} & \defeq W_A(Q) & & \text{for $B\in\cat{F}/\Gamma$ and $Q\in\cat{F}/\ctxext{\Gamma}{B}$}\\
\ctxwk{A}{g} & \defeq W_A(g) & & \text{for $B\in\cat{F}/\Gamma$, $Q\in\cat{F}/\ctxext{\Gamma}{B}$ and $g\in T(Q)$}
\end{align*}
\end{linenomath*}
\end{defi}

\begin{defi}\label{def:not-subst}
Let $x\in T(A)$ for a family $A\in\cat{F}/\Gamma$. The infix form of substitution
by $x$ is taken to be $\subst{x}{\blank}$. Thus, we will write
\begin{linenomath*}
\begin{align*}
\subst{x}{P} & \defeq S_x(P) & & \text{for $P\in\cat{F}/\ctxext{\Gamma}{A}$} \\
\subst{x}{Q} & \defeq S_x(Q) & & \text{for $P\in\cat{F}/\ctxext{\Gamma}{A}$ and $Q\in\cat{F}/\ctxext{\ctxext{\Gamma}{A}}{P}$} \\
\subst{x}{g} & \defeq S_x(g) & & \text{for $P\in\cat{F}/\ctxext{\Gamma}{A}$, $Q\in\cat{F}/\ctxext{\ctxext{\Gamma}{A}}{P}$ and $g\in T(Q)$}
\end{align*}
\end{linenomath*}
\end{defi}

\begin{defi}\label{def:stratEsys}
A (pre-)E-system is \define{stratified} if its underlying category is
stratified in the sense of \cref{def:stratification}
and the underlying functor of each $W_A$ and $S_x$ is stratified with respect to the stratification induced on slices.

A morphism of stratified (pre-)E-systems is \define{stratified} if its underlying functor is stratified.

The category of stratified E-systems and stratified E-homomorphisms between them is denoted by $\strEsys$.
\end{defi}

\begin{exa}\label{eg:fin-set-stratEsys}
The E-system on $\mathcal{N}$ from \cref{eg:fin-set-Esys} is stratified
by the identity functor.
\end{exa}

\subsubsection{Pairing and the projections}
\label{ssec:e2ce-pair}

The composition $\ctxext{A}{P}$ of $A\in\cat{F}/\Gamma$ and $P\in\cat{F}/\ctxext{\Gamma}{A}$
behaves like a strict $\Sigma$-type.
In this section we define the pairing term
$\text{pair}^{A,P}\defeq\tfid{\ctxext{A}{P}}\in T(W_P(W_A(\ctxext{A}{P})))$
and the projections and prove several useful properties about them. The strictness
is found, among other things, in the fact that we can prove judgmental $\eta$-equality,
and that pairing is strictly associative.

In this section we make use of the infix form of the
weakening and substitution operations introduced in \cref{def:not-weak,def:not-subst}.

\begin{defi}\label{def:E-pair}
Let $x\in T(A)$ and $u\in T(S_x(P))$ for $A\in\cat{F}/\Gamma$ and $P\in\cat{F}/\ctxext{\Gamma}{A}$. 
We define the \define{term extension of $x$ and $u$} to be
\begin{linenomath*}
\begin{equation*}
\tmext{x}{u}\defeq \subst{u}{\subst{x}{\tfid{\ctxext{A}{P}}}}\in T(\ctxext{A}{P}).
\end{equation*}
\end{linenomath*}
It is well defined since
\begin{linenomath*}
\begin{equation}\label{pairing}
\begin{tikzcd}
T(W_{\ctxext{A}{P}}(\ctxext{A}{P}))	\arrow[r,"S_x"]
&	T(W_{\subst{x}{P}}(\ctxext{A}{P}))	\arrow[r,"S_u"]
&	T(\ctxext{A}{P})
\end{tikzcd}
\end{equation}
\end{linenomath*}
where $S_x \circ W_{\ctxext{A}{P}} \jdeq W_{\subst{x}{P}}$
because $S_x$ is a weakening homomorphism
and $S_x \circ W_A = \mathrm{Id}$.
\end{defi}

To prove anything about the term $\tmext{x}{u}$, we need the following property.

\begin{thm}\label{subst_by_tmext}
Let $x\in T(A)$ and $u\in T(S_x(P))$ for $A\in\cat{F}/\Gamma$ and $P\in\cat{F}/\ctxext{\Gamma}{A}$.
Then we have
\begin{linenomath*}
\begin{equation*}
S_{\tmext{x}{u}}\jdeq S_u\circ (S_x/P)
\colon \sys{E}/\ctxext{\Gamma}{\ctxext{A}{P}}\to \sys{E}/\Gamma
\end{equation*}
\end{linenomath*}
\end{thm}

\begin{proof}
\begin{linenomath*}
\begin{align*}
S_{\tmext{x}{u}} & \jdeq S_{S_u(S_x(\tfid{\ctxext{A}{P}}))}
\tag{By~\ref{def:E-pair}}
\\
& \jdeq S_{S_u(S_x(\tfid{\ctxext{A}{P}}))}\circ (S_u\circ W_{S_x(P)})\circ (S_x\circ W_{A})
\tag{By~\ref{def:Esys}.\ref{def:Esys:SxW}}
\\
& \jdeq S_u \circ S_{S_x(\tfid{\ctxext{A}{P}})} \circ W_{S_x(P)}\circ S_x\circ W_A
\tag{By~\ref{def:Esys}.\ref{def:Esys:sub}}
\\
& \jdeq S_u \circ (S_x/P) \circ S_{\tfid{\ctxext{A}{P}}} \circ W_P\circ W_A
\tag{By~\ref{def:Esys}.\ref{def:Esys:sub}}
\\
& \jdeq S_u \circ (S_x/P) \circ S_{\tfid{\ctxext{A}{P}}} \circ W_{\ctxext{A}{P}}
\tag{By~\ref{def:preweak}.\ref{def:preweak:funct}}
\\
& \jdeq S_u \circ (S_x/P).
\tag{By~\ref{def:Esys}.\ref{def:Esys:SIdW}}
\end{align*}
\end{linenomath*}
\end{proof}

\begin{cor}\label{cor:tmext_assoc}
For every $x\in T(A)$, $u\in T(S_x(P))$ and $v\in T(S_{\tmext{x}{u}}(Q))$ we have
\begin{linenomath*}
\begin{equation*}
\tmext{(\tmext{x}{u})}{v}\jdeq \tmext{x}{(\tmext{u}{v})}\in T(\ctxext{\ctxext{A}{P}}{Q}).
\end{equation*}
\end{linenomath*}
\end{cor}

\begin{proof}
By \cref{subst_by_tmext}, we have
$S_v\circ (S_{\tmext{x}{u}}/Q)\jdeq S_v\circ (S_u/\subst{x}{Q})\circ (S_x/\ctxext{P}{Q}) \jdeq S_{\tmext{u}{v}}\circ (S_x/\ctxext{P}{Q})$,
so associativity of term extension follows.
\end{proof}

\begin{defi}\label{def:E-proj}
  Let $A\in\cat{F}/\Gamma$ and $P\in\cat{F}/\ctxext{\Gamma}{A}$. We define
  \begin{linenomath*}
\begin{align*}
\cprojfstf{A}{P} & \defeq \ctxwk{P}{\tfid{A}}\in T(\ctxwk{\ctxext{A}{P}}{A})
\\
\cprojsndf{A}{P} & \defeq \tfid{P}\in T(\ctxwk{P}{P})
\end{align*}
\end{linenomath*}
\end{defi}

\begin{lem}\label{lem:Ehpres-pairproj}
Let $F \colon \sys{E} \to \sys{D}$ be an E-homomorphism.
For every $A\in\cat{F}/\Gamma$, $P\in\cat{F}/\ctxext{\Gamma}{A}$,
$x\in T(A)$ and $u\in T(S_x(P))$, it is
\begin{linenomath*}
\[
F(\tmext{x}{u})\jdeq \tmext{F(x)}{F(u)},
\qquad
F(\cprojfstf{A}{P}) \jdeq \cprojfstf{F(A)}{F(P)},
\qquad\text{and}\qquad
F(\cprojsndf{A}{P}) \jdeq \cprojsndf{F(A)}{F(P)}
\]
\end{linenomath*}
\end{lem}

\begin{proof}
  We compute:
\begin{linenomath*}
\begin{gather*}
F(\tmext{x}{u}) \jdeq F(\subst{u}{\subst{x}{\tfid{\Efcmp{A}{P}}}}) \jdeq \subst{Fu}{\subst{Fx}{\tfid{\Efcmp{FA}{FP}}}} \jdeq \tmext{Fx}{Fu},
\\[1ex]
F(\cprojfstf{A}{P}) \jdeq F(\ctxwk{P}{\tfid{A}}) \jdeq \ctxwk{FP}{\tfid{FA}} \jdeq \cprojfstf{FA}{FP},
\\[1ex]
F(\cprojsndf{A}{P}) \jdeq F(\tfid{P}) \jdeq \tfid{FP} \jdeq \cprojsndf{FA}{FP}
\end{gather*}
\end{linenomath*}
where the outer equalities hold by definition,
and the inner ones since $F$ is an E-homomorphism.
\end{proof}

\begin{lem}\label{lem:pairproj}
For every $A\in\cat{F}/\Gamma$, $P\in\cat{F}/\ctxext{\Gamma}{A}$,
$x \in T(A)$ and $u \in  T(S_x(P))$, it is
\begin{linenomath*}
\begin{align*}
\subst{\tmext{x}{u}}{\cprojfstf{A}{P}} &\jdeq x,
\\
\subst{\tmext{x}{u}}{\cprojsndf{A}{P}} &\jdeq u,
\\                                           
\tmext{\cprojfstf{A}{P}}{\cprojsndf{A}{P}} &\jdeq
\tfid{\ctxext{A}{P}}.
\end{align*}
\end{linenomath*}
\end{lem}

\begin{proof}
To show that $\subst{\tmext{x}{u}}{\cprojfstf{A}{P}} \jdeq x$, we use that
$S_{\tmext{x}{u}}\jdeq S_u\circ S_x/P$ to show that
\begin{linenomath*}
\begin{align*}
\subst{\tmext{x}{u}}{\cprojfstf{A}{P}}
  & \jdeq 
\subst{\tmext{x}{u}}{\left(\ctxwk{P}{\tfid{A}}\right)}
\tag{By~\ref{def:E-proj}}
\\
  & \jdeq
\subst{u}{\subst{x}{\left(\ctxwk{P}{\tfid{A}}\right)}}
\tag{By~\cref{subst_by_tmext}}
\\
  & \jdeq
\subst{u}{\left(\ctxwk{\subst{x}{P}}{\subst{x}{\tfid{A}}}\right)}
\tag{By~\ref{def:Esys}.\ref{def:Esys:sub}}
\\
  & \jdeq
\subst{x}{\tfid{A}}
\tag{By~\ref{def:Esys}.\ref{def:Esys:SxW}}
\\
  & \jdeq
x
\tag{By~\ref{def:Esys}.\ref{def:Esys:SxId}}
\end{align*}
\end{linenomath*}

To show that $\subst{\tmext{x}{u}}{\cprojsndf{A}{P}}\jdeq u$, note that
\begin{linenomath*}
\[
\subst{\tmext{x}{u}}{\cprojsndf{A}{P}}
  \jdeq
\subst{u}{\subst{x}{\tfid{P}}}
  \jdeq
\subst{u}{\tfid{\subst{x}{P}}}
  \jdeq
u
\]
\end{linenomath*}
Finally note that
\begin{linenomath*}
\[
\ctxext{\ctxwk{\ctxext{A}{P}}{A}}{(W_{\ctxext{A}{P}}/A)(P)}
\jdeq \ctxwk{\ctxext{A}{P}}{\ctxext{A}{P}}
\colon\Gamma.A.P.\ctxwk{\ctxext{A}{P}}{(\ctxext{A}{P})} \to \Gamma.A.P
\]
\end{linenomath*}
and $\tfid{\ctxwk{\ctxext{A}{P}}{\ctxext{A}{P}}}
\jdeq (W_{\ctxext{A}{P}}/\ctxext{A}{P})(\tfid{\ctxext{A}{P}})$.
Thus $\tmext{\cprojfstf{A}{P}}{\cprojsndf{A}{P}}\jdeq \tfid{\ctxext{A}{P}}$
follows from the commutativity of the outer square in the diagram below.
\begin{linenomath*}
\[\begin{tikzcd}[column sep=large,row sep=large]
T( W_{\ctxwk{\ctxext{A}{P}}{\ctxext{A}{P}}}
		(\ctxwk{\ctxext{A}{P}}{\ctxext{A}{P}}) )
	\arrow[r,"S_{\cprojfstf{A}{P}}"]
&	T(W_P/P (\ctxwk{\ctxext{A}{P}}{\ctxext{A}{P}}))
	\arrow[d,"S_{\cprojsndf{A}{P}}"]
\\
T(\ctxwk{\ctxext{A}{P}}{\ctxext{A}{P}})
	\arrow[u,"W_{\ctxext{A}{P}}/\ctxext{A}{P}"]
	\arrow[ur,"W_P/P"]
	\arrow[r,"\catid{}"]
&	T(\ctxwk{\ctxext{A}{P}}{\ctxext{A}{P}})
\end{tikzcd}\]
\end{linenomath*}
The bottom-right triangle commutes by~\ref{def:Esys}.\ref*{def:Esys:SIdW}.
For the top-left one:
\begin{linenomath*}
\begin{align*}
S_{\ctxwk{P}{\tfid{A}}} \circ W_{\ctxext{A}{P}}/\ctxext{A}{P}
&\jdeq
S_{\ctxwk{P}{\tfid{A}}} \circ W_P/(W_A(\ctxext{A}{P}))
\circ W_A/\ctxext{A}{P}
\tag{By~\ref{def:preweak}.\ref{def:preweak:funct}}
\\&\jdeq
\left( W_P \circ S_{\tfid{A}} \circ W_A/A \right)/P
\tag{By~\ref{def:Esys}.\ref{def:Esys:weak}}
\\&\jdeq
W_P/P.
\tag{By~\ref{def:Esys}.\ref{def:Esys:SIdW}}
\end{align*}
\end{linenomath*}
\end{proof}

\begin{thm}\label{thm:pairing}
For every $A\in\cat{F}/\Gamma$ and $P\in\cat{F}/\ctxext{\Gamma}{A}$,
the map
\begin{linenomath*}
\[\begin{tikzcd}[row sep=tiny]
\coprod_{x\in T(A)} T(\subst{x}{P}) \arrow[r]
&	T(\ctxext{A}{P})
\\
(x,u) \arrow[r,mapsto]	&	\tmext{x}{u}
\end{tikzcd}\]
\end{linenomath*}
is a bijection.
\end{thm}

\begin{proof}
The inverse
to the given map is defined
by $w\mapsto (\subst{w}{\cprojfstf{A}{P}},\subst{w}{\cprojsndf{A}{P}})$.
Thanks to \cref{lem:pairproj} it is enough to show
that, for every $w\in T(\ctxext{A}{P})$, one has
\begin{linenomath*}
\begin{equation*}
\tmext{\subst{w}{\cprojfstf{A}{P}}}{\subst{w}{\cprojsndf{A}{P}}} \jdeq w.
\end{equation*}
\end{linenomath*}
\cref{lem:Ehpres-pairproj} gives us that
y\begin{linenomath*}
\begin{equation*}
\tmext{\cprojfst{A}{P}{w}}{\cprojsnd{A}{P}{w}}
  \jdeq
\subst{w}{(\tmext{\cprojfstf{A}{P}}{\cprojsndf{A}{P}})}.
\end{equation*}
\end{linenomath*}
Thus the claim follows from
$\tmext{\cprojfstf{A}{P}}{\cprojsndf{A}{P}}\jdeq \tfid{\ctxext{A}{P}}$,
which holds again by \cref{lem:pairproj}.
\end{proof}

One consequence of this theorem is that the set $T(\catid{\Gamma})$ has exactly one element, see \cref{cor:Etrm-homiso}.

\subsection{Characterising B-systems as stratified E-systems}
\label{ssec:b2e}

In this section we construct an equivalence of categories between B-systems and the subcategory of $\Esys$ on the stratified E-systems and stratified homomorphisms.
The functor from B-systems to stratified E-systems is constructed in \cref{ssec:b2se}, the one in the other direction in \cref{ssec:se2b}.
That these form  an equivalence is shown in \cref{ssec:b2e:eqv}.

\subsubsection{From B-systems to stratified E-systems}
\label{ssec:b2se}

Note first that
we obtain a functor $\bfr\to\Cat$ as the composition
\begin{linenomath*}
\[\begin{tikzcd}[column sep=3em]
\bfr \ar[r,"\Btorttr"]
&	\rttr \ar[r,"\Gfctr"]
&	\grph \ar[r,"\Ffctr"]
&	\Cat
\end{tikzcd}\]
\end{linenomath*}
where $\Gfctr$ and $\Ffctr$ are the functors from \cref{def:rttr-to-cat}
and $\Rfctr$ is the forgetful functor from \cref{def:B-to-rttr}.
Arrows in $\Ffctr\Gfctr\Btorttr(\sys{B})$ are of the form
$(X,k) \colon (n+k,X) \to (n,\ft^k(X))$, for $X \in B_{n+k}$.

We begin by equipping $\Ffctr\Gfctr\Btorttr(\sys{B})$ with a term structure.
The B-frame $\sys{B}$ already provides us with sets of terms for the edges of $\Gfctr\Btorttr(\sys{B})$, namely $T(X,1) \defeq \partial^{-1}(X)$.
In order to construct sets of terms for $(X,k)$ for each $k$,
which we do in \cref{constr:BtoTrmStr},
we assume that $\sys{B}$ comes with a substitution structure in the sense of \cref{def:subst}.
We then show in \cref{constr:BtoTrmStrFun}
that $\Ffctr\Gfctr\Btorttr$ gives rise to a functor $\BtoTC$
from B-frames with substitution to \scats with term structure.
Next, in \cref{constr:BtopreE} we provide $\BtoTC(\sys{B})$ with a pre-E-system structure when $\sys{B}$ is a B-system,
and prove in \cref{lem:BtoEprehom}
that $\BtoTC$ preserves and reflects weakening and projection homomorphisms.
Finally,
we show in \cref{lem:BtoE-fun} that
the functor $\BtoTC$ lifts to a full and faithful functor from B-systems to stratified E-systems.

\begin{prob}\label{prob:BtoTrmStr}
For every B-frame $\sys{B}$ with substitution structure $S$,
to construct a term structure $T$ on the strict category
$\F_{\sys{B}} :\jdeq \Ffctr\Gfctr\Btorttr(\sys{B})$ and to construct, for any $t \in T(X,k)$, a homomorphism of B-frames
$\Slist{t}{k} \colon \sys{B}/X \to \sys{B}/\ft^k(X)$.
\end{prob}

\begin{construction}{prob:BtoTrmStr}\label{constr:BtoTrmStr}
We define the term structure by induction on $n \in \N$.
More precisely, for any $X \in B_n$ and $k \leq n$ we will define a set
$T(X,k)$ and, for any $t \in T(X,k)$, a homomorphism of B-frames
$\Slist{t}{k} \colon \sys{B}/X \to \sys{B}/\ft^k(X)$.

For every $n$ and $X \in B_n$, let
$T(X,0) :\jdeq \{*\}$ and
$S_*^0 :\jdeq \catid{} \colon \sys{B}/X \to \sys{B}/X$.

For every $n$ and $X \in B_{n+1}$, let
\begin{linenomath*}
\begin{equation}\label{constr:BtoTrmStr:T1}
T(X,1) :\jdeq \bd^{-1}(X) \subseteq \tilde{B}_{n+1}
\end{equation}
\end{linenomath*}
and $S_x^1 \coloneqq S_x \colon \sys{B}/X \to \sys{B}/\ft(X)$
which is a homomorphism of B-frames by assumption.

Suppose now that, for every $m \leq n$ and $Y \in B_m$, we have defined sets
$T(Y,k)$ for $k \leq m$ and, for every $t \in T(Y,k)$,
a homomorphism of B-frames $\Slist{t}{k} \colon \sys{B}/Y \to \sys{B}/\ft^k(Y)$.
Let $X \in B_{n+1}$ and define, for $1 \leq k \leq n$,
\begin{linenomath*}
\begin{equation}\label{constr:BtoTrmStr:Tsuc}
T(X,k+1) :\jdeq \coprod_{t \in T(\ft(X),k)} T(S_t(X),1).
\end{equation}
\end{linenomath*}
and, for $(t,x) \in T(X,k+1)$, a homomorphism of B-frames
$\Slist{(t,x)}{k+1}$ as the composite below
\begin{linenomath*}
\begin{equation}\label{constr:BtoTrmStr:Str}
\begin{tikzcd}
\sys{B}/X	\ar[rr,"\Slist{(t,x)}{k+1}"] \ar[dr,"{\Slist{t}{k}/X}"{swap}]
&&	\sys{B}/\ft^{k+1}(X)
\\
&	\sys{B}/\Slist{t}{k}(X)	\ar[ur,"S_x"{swap}]	&
\end{tikzcd}
\end{equation}
\end{linenomath*}
where $S_x$ comes from the substitution structure
and $\Slist{t}{k}$ from the inductive hypothesis.
\end{construction}

\begin{rems}\label{rmk:BtoTrmStr}
\hfill
\begin{enumerate}
\item\label{rmk:BtoTrmStr:pressl}
For every B-frame $\sys{B}$ and $X \in B_n$,
we have an isomorphism of strict categories
$\Ffctr\Gfctr\Btorttr(\sys{B}/X) \cong \Ffctr\Gfctr\Btorttr(\sys{B})/(n,X)$
natural in $\sys{B}$ which maps $(i,Y)$ to $(n+i,Y)$
and it is the identity on arrows.
It follows that, when $\sys{B}$ is a B-system,
we can choose the identity as the action on the term structure.
Therefore this isomorphism of categories lifts to an
isomorphism of categories with term structure
$(\Ffctr\Gfctr\Btorttr(\sys{B}/X),T_{\sys{B}}) \cong (\Ffctr\Gfctr\Btorttr(\sys{B})/(n,X),T_{\sys{B}})$.

Once we establish an E-system structure on $\Ffctr\Gfctr\Btorttr(\sys{B})$,
we will see that this isomorphism is in fact an isomorphism of E-systems.

\item\label{rmk:BtoTrmStr:pressub}
Let $\sys{A}$ and $\sys{B}$ be B-frames with substitution structure
and $H \colon \sys{A} \to \sys{B}$ be a homomorphism of B-frames.
If $H$ preserves the substitution structure,
then for every $X\in B_{n+k}$ and $t \in T(X,k)$ the square
\begin{linenomath*}
\begin{equation}\label{rmk:BtoTrmStr:pressub:sq}
\begin{tikzcd}[column sep=6em,row sep=3em]
\sys{A}/X	\ar[d,"\Slist{t}{k}"{swap}] \ar[r,"H/X"]
&	\sys{B}/H(X)	\ar[d,"\Slist{\tilde{H}(t)}{k}"]
\\
\sys{A}/\ft^k(X)	\ar[r,"{H/\ft^k(X)}"]
& \sys{B}/\ft^kH(X)
\end{tikzcd}
\end{equation}
\end{linenomath*}
commutes in $\Bfr$, where
$\tilde{H}(t) :\jdeq (\tilde{H}(t_1),\ldots,\tilde{H}(t_k))$.
Indeed, by definition of $\Slist{t}{k}$ in~\eqref{constr:BtoTrmStr:Str},
the square~\eqref{rmk:BtoTrmStr:pressub:sq} factors vertically into $k$ squares of the form in
\cref{def:Bsys-pres}.\ref{def:Bsys-pres:sub},
each of which commutes if $H$ preserves the substitution structure.

\item\label{rmk:BtoTrmStr:presweak}
Let $\sys{A}$ and $\sys{B}$ be B-frames with substitution structure
and $H \colon \sys{A} \to \sys{B}$ be a homomorphism of B-frames.
Suppose that $\sys{A}$ and $\sys{B}$ have weakening structure and define,
for every $X \in B_{n+k}$, the homomorphism of B-frames
$W_X^k \colon \sys{B}/\ft^k(X) \to \sys{B}/X$ as the composite
\begin{linenomath*}
\begin{equation}\label{rmk:BtoTrmStr:W}
\begin{tikzcd}[column sep=4em]
\sys{B}/\ft^k(X) \ar[rrr,bend left=3ex,"W_X^k"] \ar[r,"{W_{\ft^{k-1}(X)}}"{swap}]
&	\cdots	\ar[r,"{W_{\ft(X)}}"{swap}]	&	\sys{B}/\ft(X)	\ar[r,"{W_X}"{swap}]
&	\sys{B}/X
\end{tikzcd}
\end{equation}
\end{linenomath*}
which we take to be $\catid{\sys{B}}$ if $n\jdeq k \jdeq 0$.

If $H$ preserves weakening structure,
then for every $X\in B_{n+k}$ the square
\begin{linenomath*}
\begin{equation}\label{rmk:BtopreE:presw:sq}
\begin{tikzcd}[column sep=6em,row sep=3em]
\sys{A}/X	\ar[r,"{H/X}"]
& \sys{B}/H(X)
\\
\sys{A}/\ft^k(X)	\ar[u,"W_X^k"] \ar[r,"H/\ft^k(X)"]
&	\sys{B}/\ft^k H(X)	\ar[u,"W_{H(X)}^k"{swap}]
\end{tikzcd}
\end{equation}
\end{linenomath*}
commutes in $\Bfr$.
Indeed, by definition of $W_X^k$ in~\eqref{rmk:BtoTrmStr:W}
the square~\eqref{rmk:BtopreE:presw:sq} factors vertically into $k$ squares of the form in
\cref{def:Bsys-pres}.\ref{def:Bsys-pres:weak},
each of which commutes if $H$ preserves the weakening structure.
\end{enumerate}
\end{rems}

\begin{prob}\label{prob:BtoTrmStrFun}
To lift the functor $\Ffctr\Gfctr\Btorttr \colon \Bfr \to \Cat$
to a functor $\BtoTC \colon \catof{SubBfr} \to \catof{TCat}$
from the category $\catof{SubBfr}$ of B-frames with substitution structure
and homomorphisms of B-frames that preserve the substitution structure,
to the category of \scats with term structure.
\end{prob}

\begin{construction}{prob:BtoTrmStrFun}\label{constr:BtoTrmStrFun}
Let $\sys{A}$ and $\sys{B}$ be B-frames with substitution structure.
For every homomorphism of B-frames $H \colon \sys{A} \to \sys{B}$,
the functor $\Ffctr\Gfctr\Btorttr(H) \colon \F_\sys{A} \to \F_\sys{B}$,
maps an object $(n,X)$ to $(n,H(X))$ and an arrow $(X,k)$ to $(H(X),k)$.
Since $H(*) \jdeq *$, the functor $\Ffctr\Gfctr\Btorttr(H)$ strictly preserves the (unique) terminal object.

To make $\BtoTC(H) :\jdeq \Ffctr\Gfctr\Btorttr(H)$ into a functor with term structure
note that, for every $t \jdeq (t_1,\ldots,t_k) \in T(X,k)$ and $1 \leq j \leq k$,
the function $\tilde{H}$ restricts as follows
\begin{linenomath*}
\begin{equation}\label{constr:BtoTrmStrFun:T1}
\begin{tikzcd}[column sep=2em,row sep=3em]
T(\ft^{k-j} S_{t_{j-1}}\cdots S_{t_1}(X),1) \ar[r,hook] \ar[d,"\tilde{H}"]
&	\tilde{A}_{n-k+1}	\ar[d,"\tilde{H}"]
\\
T(\ft^{k-j} S_{\tilde{H}(t_{j-1})}\cdots S_{\tilde{H}(t_1)} H(Y),1)	\ar[r,hook]
&	\tilde{B}_{n-k+1}
\end{tikzcd}
\end{equation}
\end{linenomath*}
since $H$ commutes with the functions $\ft$ and
preserves the substitution structure in the sense of \cref{def:Bsys-pres}.
It follows that
\begin{linenomath*}
\begin{equation}\label{constr:BtoTrmStrFun:Tsuc}
\BtoTC(H)(t) :\jdeq (\tilde{H}(t_1),\ldots, \tilde{H}(t_k)) \in T(\tilde{H}(X),k).
\end{equation}
\end{linenomath*}
This makes $\BtoTC(H) \colon \F/(n,X) \to \F/(n-k,\ft^k(X))$
into a functor with term structure.

The action of $H$ on the sets $T(X,k)$ is clearly functorial in $H$.
\end{construction}

\begin{rems}\label{rmk:BtoE}
\hfill
\begin{enumerate}
\item\label{rmk:BtoE:faith}
The functor $\BtoTC \colon \catof{SubBfr} \to \catof{TCat}$
from \cref{constr:BtoTrmStrFun} is faithful,
since the functors $\Btorttr \colon \Bfr \to \rttr$,
$\Gfctr \colon \rttr \to \grph$
and $\Ffctr \colon \grph \to \Cat$ are faithful and
the sets $T(X,1)$ for $X \in B_n$ form a partition of $\tilde{B}_n$.
\item\label{rmk:BtoE:strat}
For every B-frame with substitution structure $\sys{B}$,
it follows by \cref{prop:strat-is-prop,lem:rttrStrCat} that
the underlying category of $\BtoTC(B)$ is stratified by the functor
that maps $(X,k) \colon (n+k,X) \to (n,\ft^k(X))$ to $n+k \geq n$.
\end{enumerate}
\end{rems}

\begin{prob}\label{prob:BtopreE}
For every B-system $\sys{B}$,
to construct a pre-E-system structure on the category with term structure
$\BtoTC(\sys{B})$ from \cref{constr:BtoTrmStr}.
\end{prob}

\begin{construction}{prob:BtopreE}\label{constr:BtopreE}
\Cref{constr:BtoTrmStr} provides a homomorphism of B-frames
$\Slist{t}{k} \colon \sys{B}/X \to \sys{B}/\ft^k(X)$
for every $X \in B_{n+k}$ and $t \in T(X,k)$.
The homomorphism $\Slist{t}{k}$ preserves the substitution structure
since it factors, as in \cref{rmk:BtoTrmStr}.\ref{rmk:BtoTrmStr:pressub},
into $k$ B-homomorphisms of the form $S_{x_j}$, where $x_j \in \tilde{B}_{n+k-j}$ for $j < k$.
\Cref{constr:BtoTrmStrFun} and \cref{rmk:BtoTrmStr}.\ref{rmk:BtoTrmStr:pressl}
then yield a functor with term structure
\begin{linenomath*}
\begin{equation}\label{constr:BtopreE:sub}
\begin{tikzcd}[column sep=4em]
(\BtoTC(\sys{B})/(n+k,X),T_{\sys{B}}) \ar[r,"S_t :\jdeq \BtoTC(\Slist{t}{k})"]
&	(\BtoTC(\sys{B})/(n,\ft^k(X)),T_{\sys{B}})
\end{tikzcd}
\end{equation}
\end{linenomath*}
as required.

To construct the pre-weakening structure,
consider the homomorphism of B-frames
$W_X^k \colon \sys{B}/\ft^k(X) \to \sys{B}/X$
defined in \cref{rmk:BtoTrmStr}.\ref{rmk:BtoTrmStr:presweak}.
Since $\sys{B}$ is a B-system, each factor of $W_X^k$ in~\eqref{rmk:BtoTrmStr:W}
is a homomorphism of B-systems and so is $W_X^k$.
\Cref{constr:BtoTrmStrFun} and \cref{rmk:BtoTrmStr}.\ref{rmk:BtoTrmStr:pressl}
provide a functor with term structure
\begin{linenomath*}
\begin{equation}\label{constr:BtopreE:weak}
\begin{tikzcd}[column sep=8em]
\BtoTC(\sys{B})/(n,\ft^k(X))	\ar[r,"{W_{(X,k)} :\jdeq \BtoTC(W_X^k)}"]
&	\BtoTC(\sys{B})/(n,X).
\end{tikzcd}
\end{equation}
\end{linenomath*}

It remains to construct the pre-projection structure.
In fact, we will prove a little bit more.
We construct by induction on $n\in\N$,
for every $X \in B_n$ and $k\leq n$, an element
$\tfid{(X,k)}\in T(W_{(X,k)}(X,k)) \jdeq T(W_X^k(X),k)$
with the property that the triangle of B-homomorphisms
\begin{linenomath*}
\[\begin{tikzcd}\label{constr:BtopreE:Sid}
\sys{B}/X	\ar[dr,"W_X^k/X"{swap}] \ar[rr,"\catid{}"]
&&	\sys{B}/X
\\
&	\sys{B}/W_X^k(X)	\ar[ur,"\Slist{\tfid{(X,k)}}{k}"{swap}]	&
\end{tikzcd}\]
\end{linenomath*}
commutes.
This additional condition is needed in the inductive construction.
For every $n$ and $X \in B_n$, let
\begin{linenomath*}
\[
\tfid{(X,0)} :\jdeq * \in T(W_{(X,0)}(X,0),0) \jdeq T(X,0).
\]
\end{linenomath*}
For every $n$ and $X \in B_{n+1}$, it is
$\bd \circ \delta(X) \jdeq W_X(X) \in B_{n+2}$.
Thus we can define
\begin{linenomath*}
\begin{equation}\label{constr:BtopreE:idtm-base}
\tfid{(X,1)} :\jdeq \delta(X) \in T(W_{(X,1)}(X,1)) \jdeq T(W_X(X),1)
\end{equation}
\end{linenomath*}
and $\Slist{\tfid{(X,1)}}{1} \circ W_X/X \jdeq \catid{\sys{B}/X}$
by condition~\ref{def:Bsys:SIdW} in \cref{def:Bsys}.

Suppose now that we have defined,
for every $m\leq n$, $Y \in B_m$ and $i \leq m$,
an element $\tfid{(Y,i)} \in T(W_{(Y,i)}(Y,i))$ such that
$\Slist{\tfid{(Y,i)}}{i}\circ W_Y^i/Y \jdeq \catid{\mathcal{T}/(m,Y)}$.
Let $X \in B_{n+1}$. It follows from~\eqref{constr:BtoTrmStr:Tsuc} that,
for every $1\leq k \leq n$
\begin{linenomath*}
\[
T(W_{(X,k+1)}(X,k+1)) \jdeq
\coprod_{t\in T(W_X^{k+1}(\ft(X)),k)} T(\Slist{t}{k} \circ W_X^{k+1}(X),1).
\]
\end{linenomath*}
But $W_X^{k+1} \jdeq W_X \circ W_{\ft(X)}^k$,
thus
\begin{linenomath*}
\[
\bar{t} :\jdeq \tilde{W}_X (\tfid{(\ft(X),k)}) \in T(W_X^{k+1}(\ft(X)),k)
\]
\end{linenomath*}
and
\begin{linenomath*}
\begin{equation}\label{constr:BtopreE:Sstep}
\begin{split}
\Slist{\bar{t}}{k} \circ W_X^{k+1}/\ft(X) &\jdeq
\Slist{\bar{t}}{k} \circ W_X/W_{\ft(X)}^k(\ft(X)) \circ W_{\ft(X)}^k/\ft(X)
\\&\jdeq
W_X \circ \Slist{\tfid{(\ft(X),k)}}{k} \circ W_{\ft(X)}^k/\ft(X)
\\&\jdeq
W_X
\end{split}
\end{equation}
\end{linenomath*}
by \cref{rmk:BtoTrmStr}.\ref{rmk:BtoTrmStr:pressub}
and the fact that $W_X$ preserves the substitution structure,
and assumption~\eqref{constr:BtopreE:Sid}.
In particular, $T(S_{\bar{t}} \circ W_X^{k+1}(X),1) \jdeq T(W_X(X),1)$
and we can define
\begin{linenomath*}
\begin{equation}\label{constr:BtopreE:idtm-suc}
\tfid{(X,k+1)}\defeq (\bar{t},\delta(X)).
\end{equation}
\end{linenomath*}
It remains to check that
$\Slist{\tfid{(X,k+1)}}{k+1}\circ  W_X^{k+1}/X \jdeq \catid{\sys{B}/X}$.
This is indeed the case by~\eqref{constr:BtoTrmStr:Str},
~\eqref{constr:BtopreE:Sstep} and condition~\ref{def:Bsys:SIdW} in \cref{def:Bsys}:
\begin{linenomath*}
\begin{align*}
\Slist{\tfid{(X,k+1)}}{k+1}\circ W_X^{k+1}/X
&\jdeq
S_{\delta(X)}\circ \left(\Slist{\bar{t}}{k} \circ W_X^{k+1}/\ft(X)\right)/X
\\&\jdeq
S_{\delta(X)}\circ W_X/X
\\&\jdeq
\catid{\sys{B}/X}.
\end{align*}
\end{linenomath*}
This completes the construction of the pre-E-system structure.
\end{construction}

\begin{lem}\label{lem:BtoEprehom}
Let $\sys{A}$ and $\sys{B}$ be B-systems and $H \colon \sys{A} \to \sys{B}$
a homomorphism of B-frames that preserves the substitution structure.
\begin{enumerate}
\item\label{lem:BtoEprehom:sub}
The functor with term structure $\BtoTC(H) \colon \F_{\sys{A}} \to \F_{\sys{B}}$
is a pre-substitution homomorphism.
\item\label{lem:BtoEprehom:weak}
$H$ preserves the weakening structure if and only if
$\BtoTC(H)$ is a pre-weakening homomorphism.
\item\label{lem:BtoEprehom:gen}
$H$ preserves the structure of generic elements if and only if
$\BtoTC(H)$ is a pre-projection homomorphism.
\end{enumerate}
\end{lem}

\begin{proof}
\hfill
\begin{enumerate}
\item
By definition of the pre-substitution structure in \cref{constr:BtopreE} and
\cref{rmk:BtoTrmStr}.\ref{rmk:BtoTrmStr:pressl},
$\BtoTC(H)$ is a pre-substitution homomorphism if
every image under $\BtoTC$ of any square in $\Bfr$ of the 
form~\eqref{rmk:BtoTrmStr:pressub:sq} commutes.
By \cref{rmk:BtoTrmStr}.\ref{rmk:BtoTrmStr:pressub}, such squares commute since
$H$ preserves the substitution structure.
\item
By definition of the pre-weakening structure~\eqref{rmk:BtoTrmStr:W} and
\cref{rmk:BtoTrmStr}.\ref{rmk:BtoTrmStr:pressl},
$\BtoTC(H)$ is a pre-weakening homomorphism if and only if
the image under $\BtoTC \colon \Bsys \to \catof{TCat}$
of any square in $\Bfr$ of the form~\eqref{rmk:BtopreE:presw:sq} commutes.
By \cref{rmk:BtoTrmStr}.\ref{rmk:BtoTrmStr:presweak}, such squares commute if
$H$ preserves the weakening structure.
The converse holds since $\BtoTC$ is faithful by \cref{rmk:BtoE}.\ref{rmk:BtoE:faith}.
\item
By~\eqref{constr:BtoTrmStrFun:T1} and~\eqref{constr:BtoTrmStrFun:Tsuc},
$\BtoTC(H)$ acts componentwise as $\tilde{H}$ on a term $t \in T(X,t)$.
It follows that $\BtoTC(H)$ preserves the terms
$\tfid{(X,k+1)} \jdeq (\delta(\ft^k(X)),\ldots,\delta(X))$
for $X \in B_{n}, k < n$ if and only if
$H$ preserves generic elements.
\qedhere
\end{enumerate}
\end{proof}

\begin{lem}\label{lem:BtoE-obj}
For every B-system $\sys{B}$, the pre-E-system constructed 
in~\ref{constr:BtopreE} is a stratified E-system.
\end{lem}

\begin{proof}
First, we need to verify conditions~\ref*{def:Esys:sub}--\ref*{def:Esys:SxId}
in \cref{def:Esys}, as condition~\ref*{def:Esys:SIdW} holds by construction.

\begin{itemize}
\item [\ref*{def:Esys:sub}.]
It follows from \cref{lem:BtoEprehom} and~\eqref{constr:BtopreE:sub} since $\Slist{t}{k}$,
as defined in~\eqref{constr:BtoTrmStr:Str}, is a homomorphism of B-systems
when $\sys{B}$ is a B-system.

\item [\ref*{def:Esys:weak}.]
As above, it follows by \cref{lem:BtoEprehom} and~\eqref{constr:BtopreE:weak}
since $W_X^k$, as defined in~\eqref{rmk:BtoTrmStr:W},
is a homomorphism of B-systems.

\item [\ref*{def:Esys:SxW}.]
The case $X \in B_n, *\in T(X,0)$ holds trivially.
The case $X \in B_{n+1}, x\in T(X,1)$ follows from
condition~\ref{def:Bsys:SxW} in \cref{def:Bsys} and functoriality of $\BtoTC$.
The case $X \in B_{n+k+1}, (t,x)\in T(X, \linebreak[5] k+1)$,
where $t \in T(\ft(X),k)$ and $x \in T(\Slist{t}{k}(X),1)$,
holds by induction and functoriality of $\BtoTC$ as
\begin{linenomath*}
\[\begin{split}
\Slist{(t,x)}{k} \circ W_X^{k+1}
&\jdeq
S_x \circ \Slist{t}{k}/X \circ W_X \circ W_{\ft(X)}^k
\\&\jdeq
S_x \circ W_{S_t(X)} \circ \Slist{t}{k} \circ W_{\ft(X)}^k
\\&\jdeq
\catid{\sys{B}/\ft^{k+1}(X)}
\end{split}\]
\end{linenomath*}
by~\eqref{constr:BtoTrmStr:Str} and~\eqref{rmk:BtoTrmStr:W},
the fact that $S_t$ is a pre-E-homomorphism,
and \cref{def:Bsys}.2 and the inductive hypothesis.

\item [\ref*{def:Esys:SxId}.]
As above, the case $k=0$ holds trivially and
the case $k=1$ holds by condition~\ref{def:Bsys:SxId} in \cref{def:Bsys}.
For $X \in B_{n+1},\ k \leq n$ and $(t,x) \in T(X,k+1)$,
\begin{linenomath*}
\[\begin{split}
\SlistT{(t,x)}{k+1}(\tfid{(X,k+1)})
&\jdeq
\tilde{S}_x \circ \SlistT{t}{k}( \tilde{W}_X(\tfid{(\ft(X),k)}) , \delta(X) )
\\&\jdeq
\tilde{S}_x(\tilde{W}_{\Slist{t}{k}(X)}\circ \SlistT{t}{k}(\tfid{(\ft(X),k)}) , \delta(\Slist{t}{k}(X)) )
\\&\jdeq
( \tilde{S}_x \circ \tilde{W}_{\Slist{t}{k}(X)}(t) , \tilde{S}_x(\delta(\Slist{t}{k}(X))) )
\\&\jdeq
(t,x)
\end{split}\]
\end{linenomath*}
by~\eqref{constr:BtoTrmStr:Str} and~\eqref{constr:BtopreE:idtm-suc},
the fact that $S_t$ is a pre-E-homomorphism, the inductive  hypothesis,
and conditions \ref{def:Bsys:SxW} and \ref{def:Bsys:SxId} in \cref{def:Bsys}.

\end{itemize}
Finally, the underlying category $\F \jdeq \Ffctr\Gfctr\Btorttr(\sys{B})$ is stratified by \cref{rmk:BtoE}.\ref{rmk:BtoE:strat}.
By definition, weakening and substitution functors preserve
the $\N$-component of objects and arrows.
It follows that $\BtoTC(\sys{B})$ is a stratified E-system.
\end{proof}

\begin{lem}\label{lem:BtoE-fun}
\hfill
\begin{enumerate}
\item\label{lem:BtoE-fun:fun}
The functor $\BtoTC \colon \catof{SubBfr} \to \catof{TCat}$
described in \cref{constr:BtoTrmStrFun}
lifts to a functor $\BtoE \colon \Bsys \to \strEsys$.
\item\label{lem:BtoE-fun:ff}
The functor $\BtoE$ is full and faithful.
\end{enumerate}
\end{lem}

\begin{proof}~
  \begin{enumerate}
  \item [\ref*{lem:BtoE-fun:fun}.]
By \cref{lem:BtoE-obj}, it is enough to show that,
for every homomorphism of B-systems
$H \colon \sys{A} \to \sys{B}$,
the functor with term structure
$\BtoTC(H) \colon \BtoTC(\sys{A}) \to \BtoTC(\sys{B})$
is a stratified homomorphism of E-systems.
By \cref{lem:BtoEprehom}, $\BtoTC(H)$ is a homomorphism of E-systems.
It is stratified since it preserves the $\N$-component of objects and arrows
by definition.

\item [\ref*{lem:BtoE-fun:ff}.]
The functor $\BtoE$ is faithful by \cref{rmk:BtoE}.\ref{rmk:BtoE:faith}.
Let then $K \colon \BtoE(\sys{A}) \to \BtoE(\sys{B})$
be a stratified homomorphism of E-systems.
Since $K$ is stratified,
the function on objects $K \colon \coprod_m A_m \to \coprod_n B_n$
is the identity on indices and gives rise to a family of functions
$H \colon \prod_n (A_n \to B_n)$ such that,
for every object $(n,X)$ and arrow $(X,k)\colon(n+k,X)\to(n,\ft^k(X))$
\begin{linenomath*}
\begin{equation}\label{constr:BtoE-full:fullty}
K(n,X) \jdeq (n,H_n(X))
\qquad\text{and}\qquad
K(X,k)\jdeq (H_{n+k}(X),k).
\end{equation}
\end{linenomath*}
We shall show that $H$ is a morphism of B-system such that $\BtoE(H)=K$.

The functions $H_n$ commute with the father functions $\ft$
since, for every $X \in A_{n+1}$, the arrow
$K(X,1) \colon (n+1,H_{n+1}(X)) \to (n,H_n(\ft(X)))$ in $\F_\sys{B}$
is necessarily of the form $(n+1,Y) \to (n,\ft(Y))$.

The family of sets $T(X,1)$ indexed on $X \in A_{n+1}$ forms a partition of $\tilde{A}_{n+1}$.
Therefore the functions $K_X \colon T(X,1) \to T(H_{n+1}(X),1)$
glue together to form a function
$\tilde{H}_{n+1} \colon \tilde{A}_{n+1} \to \tilde{B}_{n+1}$
such that
\begin{linenomath*}
\begin{equation}\label{constr:BtoE-full:fulltm}
\tilde{H}_{n+1}(x) \jdeq K_{\bd(x)}(x).
\end{equation}
\end{linenomath*}
It follows that $\bd \circ \tilde{H} \jdeq H \circ \bd$
since $\tilde{H}_{n+1}(x) \in T(H_{n+1}(\bd(x)),1)$.
Therefore $H$ is a homomorphism of B-frames from $\sys{A}$ to $\sys{B}$.

Let $x \in \tilde{B}_{n+1}$.
Since $K$ is a substitution homomorphism,
for every $Y \in B_{n+k+1}$ such that $\ft^k(Y) \jdeq \bd(x)$, it is
\begin{linenomath*}
\[\begin{split}
(n+k,\tilde{H}_{n+k} \circ S_x(Y)) %
&\jdeq
K \circ S_x(n+k+1,Y)
\\&\jdeq
S_{K_{\bd(x)}(x)} \circ K(n+k+1,Y)
\\&\jdeq
(n+k,S_{\tilde{H}_{n+1}(x)} \circ H_{n+k+1}(Y))
\end{split}\]
\end{linenomath*}
and, for every $y \in \tilde{B}_{n+k+1}$ such that
$\ft^k \circ \bd(y) \jdeq \bd(x)$, it is
\begin{linenomath*}
\[\begin{split}
\tilde{H}_{n+k} \circ \tilde{S}_x(y) &\jdeq
K_{\bd \circ S_x(y)} \circ S_x(y)
\\&\jdeq
S_{K_{\bd(x)}(x)} \circ K_{\bd(y)}(y)
\\&\jdeq
\tilde{S}_{\tilde{H}_{n+1}(x)} \circ \tilde{H}_{n+k+1}(y).
\end{split}\]
\end{linenomath*}
It follows that the homomorphism of B-frames $H$ preserves the substitution structure.

We can thus apply \cref{constr:BtoTrmStrFun} to $H$
and observe that $\BtoTC(H) \jdeq K$.
Indeed $\BtoTC(H)$ and $K$ have the same action on objects and arrows
because of~\eqref{constr:BtoE-full:fullty} and \cref{constr:BtoTrmStrFun}.
To see that they also agree on the term structure, recall from \cref{constr:BtoTrmStr} that the term structure of an E-system of the form $\BtoE(\sys{B})$ is given by lists of elements in the sets $\tilde{B}_n$,
and then use~\eqref{constr:BtoE-full:fulltm}
and~\eqref{constr:BtoTrmStrFun:Tsuc}.
Therefore $\BtoE(H) \jdeq \BtoTC(H) \jdeq K$
once we show that $H$ is a homomorphism of B-systems.

It remains to verify that $H$ also preserve the weakening structure
and the structure of generic elements.
Since $K$ is a projection homomorphism,
for every $X \in B_{n+1}$ it is
\begin{linenomath*}
\[
\BtoTC(H/X \circ W_X) \jdeq K/(n+1,X) \circ W_{(X,1)}
\jdeq W_{K(X,1)} \circ K/(n,\ft(X))
\jdeq \BtoTC(W_{H_{n+1}(X)} \circ H/\ft(X)).
\]
\end{linenomath*}
The first claim then follows from faithfulness of $\BtoE$.
Finally, $H$ preserves generic elements
\begin{linenomath*}
\[
\tilde{H}_{n+2} \circ \delta(X) \jdeq K_{W_X(X)}(\tfid{(X,1)})
\jdeq \tfid{K(X,1)}
\jdeq \delta \circ H_{n+1}(X)
\]
\end{linenomath*}
since $K$ is a projection homomorphism.
\qedhere
\end{enumerate}
\end{proof}

\subsubsection{From stratified E-sytems to B-systems}
\label{ssec:se2b}

We have constructed a full and faithful functor $\Bsys \to \strEsys$.
Here we construct a functor in the opposite direction.
We begin in \cref{constr:TCtoBfr-fun}
with a functor $\EtoB$ from stratified categories with term structures to B-frames.
In \cref{constr:preEtopreB}
we consider substitution, weakening and projection structures and
prove in \cref{lem:preEtopreBhom} that
$\EtoB$ maps homomorphisms into homomorphisms.
This allows us to lift $\EtoB$ to a functor $\strEsys \to \Bsys$ in \cref{constr:EtoB-fun}.

\begin{prob}\label{prob:TCtoBfr-ob}
Given a stratified category with term structure $(\F,T)$,
to construct a B-frame $\BtoE(\F,T)$.
\end{prob}

\begin{construction}{prob:TCtoBfr-ob}\label{constr:TCtoBfr-ob}
For every object $X$ in $\F$,
let $\overline{X}$ denote the unique \indecarr arrow with domain $X$ given by \cref{lem:strat-alt}.
For every $n\in \N$, define sets
\begin{linenomath*}
\begin{align}
\label{constr:TCtoBfr-obj:B}
B(\F,T)_n &\defeq \left\{X\in \ob{\F}\mid L(X)\jdeq n\right\}
\\[1ex]
\label{constr:TCtoBfr-obj:Bt}
\tilde{B}(\F,T)_{n+1} &\defeq \coprod_{X\in B(\F,T)_{n+1}} T(\overline{X})
\end{align}
\end{linenomath*}
and functions $\ft[n] \colon B(\F,T)_{n+1} \to B(\F,T)_n$
and $\bd[n] \colon \tilde{B}(\F,T)_{n+1} \to B(\F,T)_{n+1}$ by
\begin{linenomath*}
\begin{align}
\label{constr:TCtoBfr-obj:ft}
\ft(X) &\defeq \mathrm{cod}(\overline{X})
\\[1ex]
\label{constr:TCtoBfr-obj:bd}
\bd(X,x) &\defeq X.
\end{align}
\end{linenomath*}           
These definitions give rise to a B-frame $\EtoB(\F,T)$.
\end{construction}

\begin{prob}\label{prob:TCtoBfr-fun}
To construct a functor $\EtoB \colon \catof{TCat_s} \to \Bfr$
from the category of stratified categories with term structure
and stratified functors with term structure
to the category of B-frames and homomorphisms.
\end{prob}

\begin{construction}{prob:TCtoBfr-fun}\label{constr:TCtoBfr-fun}
The action on objects is given by \cref{constr:TCtoBfr-ob}.
Let then $F \colon (\F,T) \to (\F',T')$ be a stratified functor with term structure.
We need to construct a homomorphism of B-frames
$\EtoB(F) \colon \EtoB(\F,T) \to \EtoB(\F',T')$.
Since $F$ is stratified, it maps $B(\F,T)_n$ into $B(\F',T')_n$.
For every $X \in B(\F,T)_{n+1}$, the functor $F$ maps
the \indecarr arrow $\overline{X}$ to the \indecarr arrow $\overline{F(X)}$
by \cref{lem:strat-funct}.
It follows first that
\begin{linenomath*}
\[\begin{split}
F \circ \ft(X) &\jdeq F \circ \mathrm{cod}(\overline{X})
\\&\jdeq
\mathrm{cod}(\overline{F(X)})
\\&\jdeq
\ft \circ F(X),
\end{split}\]
\end{linenomath*}
and secondly that
we can define, for every $n \in \N$, a function
$\tilde{F} \colon \tilde{B}(\F,T)_{n+1} \to \tilde{B}(\F',T')_{n+1}$
such that $\bd \circ \tilde{F}(X,t) \jdeq F \circ \bd(X,t)$ by
\begin{linenomath*}
\begin{equation}\label{constr:TCtoBfr-fun:T}
\tilde{F}(X,t) :\jdeq (F(X),F(t)).
\end{equation}
\end{linenomath*}
This defines a homomorphism of B-frames $\EtoB(F) :\jdeq (F,\tilde{F})$.
\end{construction}

\begin{prob}\label{prob:preEtopreB}
Let $(\F,T)$ be a stratified category with term structure
and consider the B-frame $\EtoB(\F,T)$ from \cref{constr:TCtoBfr-ob}
\begin{enumerate}
\item\label{prob:preEtopreB:sub}
From a stratified pre-substitution structure on $(\F,T)$,
construct a substitution structure on $\EtoB(\F,T)$.
\item\label{prob:preEtopreB:weak}
From a stratified pre-weakening structure on $(\F,T)$,
construct a weakening structure on $\EtoB(\F,T)$.
\item\label{prob:preEtopreB:gen}
From a pre-projection structure on $(\F,T)$,
construct a structure of generic elements on $\EtoB(\F,T)$.
\end{enumerate}
\end{prob}

\begin{construction}{prob:preEtopreB}\label{constr:preEtopreB}~
  \begin{enumerate}
  \item [\ref*{prob:preEtopreB:sub}.]
For every $(X,t) \in \tilde{B}(\F,T)_{n+1}$,
the functor with term structure
$S_t \colon (\F,T)/X \to (\F,T)/\ft(X)$ is stratified.
\Cref{constr:TCtoBfr-fun} then yields a homomorphism of B-frames
\begin{linenomath*}
\begin{equation}\label{constr:preEtopreB:Sdef}
\begin{tikzcd}[column sep=6em]
\EtoB(\F,T)/X	\ar[r,"{S_{(X,t)} :\jdeq \EtoB(S_t)}"]
&	\EtoB(\F,T)/\ft(X).
\end{tikzcd}
\end{equation}
\end{linenomath*}
\item [\ref*{prob:preEtopreB:weak}.]
For every $X \in B(\F,T)_n$,
the functor with term structure
$W_{\overline{X}} \colon (\F,T)/\ft(X) \to (\F,T)/X$ is stratified,
where $\overline{X}$ denotes the unique \indecarr arrow with domain $X$.
\Cref{constr:TCtoBfr-fun} then yields a homomorphism of B-frames
\begin{linenomath*}
\begin{equation}\label{constr:preEtopreB:Wdef}
\begin{tikzcd}[column sep=6em]
\EtoB(\F,T)/\ft(X)	\ar[r,"{W_X :\jdeq \EtoB(W_{\overline{X}})}"]
&	\EtoB(\F,T)/X.
\end{tikzcd}
\end{equation}
\end{linenomath*}
\item [\ref*{prob:preEtopreB:gen}.]
For every $X \in B(\F,T)_{n+1}$,
we can define
\begin{linenomath*}
\begin{equation}\label{constr:preEtopreB:Gdef}
\delta(X) :\jdeq (W_X(X),\tfid{\overline{X}}) \in \tilde{B}(\F,T)_{n+2}
\end{equation}
\end{linenomath*}
since $\overline{W_X(X)} \jdeq W_{\overline{X}}(\overline{X})$.
\qedhere
\end{enumerate}
\end{construction}

\begin{lem}\label{lem:preEtopreBhom}
Let $F \colon (\F,T) \to (\F',T')$ be a stratified functor with term structure.
\begin{enumerate}
\item\label{lem:preEtopreBhom:sub}
If $(\F,T)$ and $(\F',T')$ have stratified pre-substitution structure and
$F$ is a pre-substitution homomorphism,
then $\EtoB(F) \colon \EtoB(\F,T) \to \EtoB(\F',T')$
preserves the substitution structure.
\item\label{lem:preEtopreBhom:weak}
If $(\F,T)$ and $(\F',T')$ have stratified pre-weakening structure and
$F$ is a pre-weakening homomorphism,
then $\EtoB(F) \colon \EtoB(\F,T) \to \EtoB(\F',T')$
preserves the weakening structure.
\item\label{lem:preEtopreBhom:gen}
If $(\F,T)$ and $(\F',T')$ have stratified pre-projection structure and
$F$ is a pre-projection homomorphism,
then $\EtoB(F) \colon \EtoB(\F,T) \to \EtoB(\F',T')$
preserves the structure of generic elements.
\end{enumerate}
\end{lem}

\begin{proof}~
  \begin{enumerate}
  \item [\ref*{lem:preEtopreBhom:sub}.]
We need to show that, for every $(X,t) \in \tilde{B}(\F,T)_{n+1}$, it is
$\EtoB(F)/\ft(X) \circ S_{(X,t)} \jdeq S_{(F(X),F(t))} \circ \EtoB(F)/X$.
This follows from~\eqref{constr:preEtopreB:Sdef}, functoriality of $\EtoB$ and
$F/\ft(X) \circ S_t \jdeq S_{F(t)} \circ F/X$,
which holds because $F$ is a pre-substitution homomorphism.

\item [\ref*{lem:preEtopreBhom:weak}.]
We need to show that, for every $X \in B(\F,T)_n$, it is
$\EtoB(F)/X \circ W_X \jdeq W_{F(X)} \circ \EtoB(F)/\ft(X)$.
This follows from~\eqref{constr:preEtopreB:Wdef}, functoriality of $\EtoB$ and
$F/X \circ W_{\overline{X}} \jdeq W_{\overline{F(X)}} \circ F/\ft(X)$,
which holds because $F$ is a pre-substitution homomorphism
and $\overline{F(X)} \jdeq F(\overline{X})$.

\item [\ref*{lem:preEtopreBhom:gen}.]
  For every $X \in B(\F,T)_{n+1}$, it is
  \begin{linenomath*}
\[
\EtoB(F) \circ \delta(X) \jdeq (F(W_X(X)),F(\tfid{\overline{X}}))
\jdeq (W_{F(X)}(F(X)), \tfid{\overline{F(X)}}) \jdeq \delta \circ \EtoB(F)(X)
\]
\end{linenomath*}
where the first and last equality hold
by~\eqref{constr:TCtoBfr-fun:T} and~\eqref{constr:preEtopreB:Gdef},
and the middle one because $F$ is a pre-projection homomorphism.
\qedhere
\end{enumerate}
\end{proof}

\begin{prob}\label{prob:EtoB-fun}
To lift the functor $\EtoB \colon \catof{TCat_s} \to \Bfr$
to a functor $\EtoB \colon \strEsys \to \Bsys$.
\end{prob}

\begin{construction}{prob:EtoB-fun}\label{constr:EtoB-fun}
Let $\sys{E}$ be a stratified E-system.
Then $\EtoB(\F,T)$ can be given the structure of a pre-B-system
$\EtoB(\sys{E})$ by \cref{constr:preEtopreB}.
To show that $\EtoB(\sys{E})$ is a B-system,
we need to verify conditions \ref{def:Bsys:sub}--\ref{def:Bsys:SIdW} of \cref{def:Bsys}.

\begin{enumerate}
\item [\ref*{def:Bsys:sub},\ref*{def:Bsys:weak}.]
Since, for every $A\in\F/\Gamma$ and $t\in T(A)$, the morphisms $W_A$ and $S_t$ are stratified E-homomorphism, it follows by \cref{lem:preEtopreBhom}
that the homomorphisms of B-frames constructed
in~\eqref{constr:preEtopreB:Sdef} and~\eqref{constr:preEtopreB:Wdef}
are homomorphisms of B-systems.

\item [\ref*{def:Bsys:SxW}.]
For $(X,t) \in \tilde{B}(\sys{E})_{n+1}$,
it is
\begin{linenomath*}
\[
S_{(X,t)} \circ W_X \jdeq \EtoB(S_t \circ W_{\overline{X}})
\jdeq \catid{\EtoB(\sys{E})/\ft(X)}
\]
\end{linenomath*}
by~\eqref{constr:preEtopreB:Sdef},~\eqref{constr:preEtopreB:Wdef},
functoriality of $\EtoB$ and~\ref{def:Esys}.\ref*{def:Esys:SxW}.

\item [\ref*{def:Bsys:SxId}.]
For $(X,t) \in \tilde{B}(\sys{E})_{n+1}$,
it is
\begin{linenomath*}
\[
S_{(X,t)} \circ \delta(X)
\jdeq ((S_{(X,t)} \circ W_X)(X), S_t(\tfid{\overline{X}}))
\jdeq (X,t)
\]
\end{linenomath*}
by~\eqref{constr:preEtopreB:Sdef},~\eqref{constr:preEtopreB:Gdef},
condition~\ref*{def:Bsys:SxW} just proved and~\ref{def:Esys}.\ref*{def:Esys:SxId}.

\item [\ref*{def:Bsys:SIdW}.]
  For every $X \in B(\sys{E})_{n+1}$, it is
  \begin{linenomath*}
\[
S_{\delta(X)} \circ W_X/X
\jdeq \EtoB( S_{\tfid{\overline{X}}} \circ W_{\overline{X}}/X )
\jdeq \catid{\EtoB(\sys{E})/X}
\]
\end{linenomath*}
by~(\ref{constr:preEtopreB:Sdef}--\ref{constr:preEtopreB:Gdef}),
functoriality of $\EtoB$ and~\ref{def:Esys}.\ref*{def:Esys:SIdW}.
\end{enumerate}
Finally, for every stratified E-homomorphism $F \colon \sys{E} \to \sys{D}$,
the homomorphism of B-frames
$\EtoB(F) \colon \EtoB(\sys{E}) \to \EtoB(\sys{D})$
is a homomorphism of B-systems by \cref{lem:preEtopreBhom}.
\end{construction}

\subsubsection{Equivalence of B-systems and stratified E-systems}
\label{ssec:b2e:eqv}

Here we show in \cref{thm:BasE} that
the functors $\BtoE$ from \cref{lem:BtoE-fun}
and $\EtoB$ from \cref{constr:EtoB-fun} form an equivalence of categories.
We do so by showing in \cref{constr:BtoE-surj} that
$\EtoB \colon \strEsys \to \Bsys$ is an
essential section of the full and faithful functor $\BtoE$.

\begin{prob}\label{prob:BtoE-surj}
For every stratified E-system $\sys{E}$,
to construct an isomorphism of stratified E-systems $\BtoE(\EtoB(\sys{E})) \cong \sys{E}$, natural in $\sys{E}$.
\end{prob}

\begin{construction}{prob:BtoE-surj}\label{constr:BtoE-surj}
In this construction we decorate the structures from $\BtoE(\EtoB(\sys{E}))$
with a hat, as in $\hat{\F}$.
Since $\F$ is stratified,
the function mapping $(n,X) \in \coprod_{n}B(\sys{E})_n$ to $X$
extends to an isomorphism $\varphi$ between
the underlying \scat $\hat{\F}$ of $\BtoE(\EtoB(\sys{E}))$,
constructed in~\ref{constr:BtoTrmStr}, and $\F$.
In particular, it maps an arrow $(X,k)$
to the arrow $\overline{X}^k :\jdeq \overline{\ft^{k-1}(X)} \circ \cdots \circ \overline{X} \colon X \to \ft^k(X)$ in $\F$ as in~\eqref{lem:strat-alt:fact}.

In order to lift $\varphi$ to an isomorphism of categories with term structure,
we need to show that $\hat{T}(X,k) \cong T(\overline{X}^k)$
for every $X \in B(\sys{E})_n$ and $k \leq n$,
where $\hat{T}(X,k)$ is the set defined in \cref{constr:BtoTrmStr}.
For every $X \in B_{n+1}$, by~\eqref{constr:BtoTrmStr:T1} it is
\begin{linenomath*}
\[
\hat{T}(X,1) \jdeq \bd^{-1}(X) \jdeq
\left\{ (Y,y) \in \tilde{B}(\sys{E})_{n+1}
\mid Y \jdeq X,\, y \in T(\overline{Y}) \right\}
\cong T(\overline{X}).
\]
\end{linenomath*}
Suppose that $\hat{T}(Y,j) \cong T(\overline{Y}^j)$
for every $m < n, Y \in B_m$ and $j \leq m$.
It follows by~\eqref{constr:BtoTrmStr:Tsuc} that
\begin{linenomath*}
\[
\hat{T}(X,k+1) \jdeq \coprod_{t \in T(\ft(X),k)}T(S_t(X),1) \cong
\coprod_{t \in T(\overline{\ft(X)}^k)} T(\overline{S_t(X)})
\cong T(\overline{X}^{k+1})
\]
\end{linenomath*}
where the last bijection follows from \cref{thm:pairing}
since $\overline{X}^{k+1} \jdeq \overline{\ft(X)}^k \circ \overline{X}$
and $\overline{S_t(X)} \jdeq S_t(\overline{X})$.
In other words, elements of $\hat{T}(X,k)$ are lists of length $k$
of pairs $(Y,y) \in \tilde{B}(\sys{E})_{n+j}$ for $j\jdeq 1,\dots,k$,
where $y \in T(\overline{Y})$,
and the action on terms of $\varphi$ first acts componentwise
dropping the first component of each pair and then applies
the bijection from \cref{thm:pairing}.

Next, we show that this choice of isos is natural in $\sys{E}$.
Given a stratified E-homomorphism $F \colon \sys{E} \to \sys{D}$, we need to show that
$\varphi_{\sys{D}} \circ \BtoE(\EtoB(F)) \jdeq F \circ \varphi_{\sys{E}}$.
The functor $\BtoE(\EtoB(F))$ maps an arrow $(X,k)$ to $(F(X),k)$, thus
\begin{linenomath*}
\[
\varphi_{\sys{D}} \circ \BtoE(\EtoB(F))(X,k)
\jdeq \overline{F(X)}^k
\jdeq F(\overline{X}^k)
\jdeq F \circ \varphi_{\sys{E}}(X,k).
\]
\end{linenomath*}
since $F$ preserves \indecarr arrows by \cref{lem:strat-funct}.
The functor with term structure $\BtoE(\EtoB(F))$ maps
$(X,x) \in \hat{T}(X,1)$ to $(F(X),F(x))$
by~\eqref{constr:BtoTrmStrFun:T1} and~\eqref{constr:TCtoBfr-fun:T}, thus
\begin{linenomath*}
\[
\varphi_{\sys{D}} \circ \BtoE(\EtoB(F))(X,x)
\jdeq F(x)
\jdeq F \circ \varphi_{\sys{E}}(X,x).
\]
\end{linenomath*}
Suppose now that, for every $m\leq n$, $Y \in B(\sys{E})_m$, $i\leq m$
and $(Y,t) \in \hat{T}(Y,i)$, it is
$\varphi_{\sys{D}} \circ \BtoE(\EtoB(F))(Y,t) \jdeq F \circ \varphi_{\sys{E}}(Y,t)$.
Let $X \in B(\sys{E})_{n+1}$ and $(t,(X,x)) \in \hat{T}(X,k+1)$, then
\begin{linenomath*}
\[\begin{split}
\varphi_{\sys{D}} \circ \BtoE(\EtoB(F))(t,(X,x))
&\jdeq \tmext{\inpar1{\varphi_{\sys{D}} \circ \BtoE(\EtoB(F))(t)}}{F(x)}
\\&\jdeq
\tmext{\inpar1{F \circ \varphi_{\sys{E}}(t)}}{F(x)}
\\&\jdeq
F(\tmext{\varphi_{\sys{E}}(t)}{x})
\\&\jdeq
F \circ \varphi_{\sys{E}}(t,(X,x))
\end{split}
\]
\end{linenomath*}
by definition of $\varphi$, inductive hypothesis, \cref{lem:Ehpres-pairproj},
and definition of $\varphi$ again.
Therefore we conclude that, for every E-homomorphism $F \colon \sys{E} \to \sys{D}$,
\begin{linenomath*}
\begin{equation}\label{constr:BtoE-surj:nat}
\varphi_{\sys{D}} \circ \BtoE(\EtoB(F)) \jdeq F \circ \varphi_{\sys{E}}.
\end{equation}
\end{linenomath*}
It remains to show that each component $\varphi_{\sys{E}}$ is an E-homomorphism.

To show that $\varphi$ is a weakening homomorphism,
note that for every $X \in B(\sys{E})_{n+k}$, it is
$\hat{W}_{(X,k)} \jdeq \BtoE(W_X^k)$ by~\eqref{constr:BtopreE:weak} and
\begin{linenomath*}
\[\begin{split}
W_X^k &\jdeq W_{\ft^{k-1}(X)} \circ \cdots W_X
\\&\jdeq
\EtoB( W_{\overline{\ft^{k-1}(X)}} \circ \cdots \circ W_{\overline{X}} )
\\&\jdeq
\EtoB( W_{\varphi(X,k)} )
\end{split}\]
\end{linenomath*}
by, in order,~\eqref{rmk:BtoTrmStr:W};~\eqref{constr:preEtopreB:Wdef} and
functoriality of $\EtoB$;
condition~\ref{def:preweak}.\ref*{def:preweak:id} in the case $k=0$ and
condition~\ref{def:preweak}.\ref*{def:preweak:funct} for $k>0$;
and definition of $\varphi$.
Moreover, $W_{\varphi(X,k)}$ is an E-homomorphism, thus
$\varphi$ is a weakening homomorphism by~\eqref{constr:BtoE-surj:nat}.

To show that $\varphi$ is a substitution homomorphism we reason by induction.
The case $X \in B(\sys{E})_n$ and $* \in \hat{T}(X,0)$ is trivial.
For every $X \in B(\sys{E})_{n+1}$ and $(X,x) \in \hat{T}(X,1)$, it is
$\hat{S}_{(X,x)} \jdeq \BtoE(S_{(X,x)})$ by~\eqref{constr:BtopreE:sub} and
\begin{linenomath*}
\[
S_{(X,x)} \jdeq \EtoB( S_x )
\jdeq \EtoB( S_{\varphi(X,x)} )
\]
\end{linenomath*}
by~\eqref{constr:preEtopreB:Sdef} and definition of $\varphi$.
Suppose now that, for every $m\leq n$, $Y \in B(\sys{E})_m$, $i \leq m$
and $t \in \hat{T}(Y,i)$, it is $S_t \jdeq \EtoB(S_{\varphi(t)})$
as homomorphisms of B-systems.
Then for every $X \in B(\sys{E})_{n+1}$, $k \leq n$ and $(t,(X,x)) \in \hat{T}(X,k+1)$,
it is $\hat{S}_{(t,(X,x))} \jdeq \BtoE(S_{(t,(X,x))})$ by~\eqref{constr:BtopreE:sub} and
\begin{linenomath*}
\[\begin{split}
S_{(t,(X,x))} &\jdeq S_{(X,x)} \circ S_t/X
\\&\jdeq
\EtoB( S_x \circ S_{\varphi(t)}/X )
\\&\jdeq
\EtoB( S_{\tmext{\varphi(t)}{x}} )
\\&\jdeq
\EtoB( S_{\varphi(t,(X,x))} )
\end{split}\]
\end{linenomath*}
by, in order,~\eqref{constr:BtoTrmStr:Str};
inductive hypothesis,~\eqref{constr:preEtopreB:Sdef} and functoriality of $\EtoB$;
\cref{subst_by_tmext}; and definition of $\varphi$.
Therefore $S_t \jdeq \BtoE(\EtoB(S_{\varphi(t)}))$
for every $X \in B(\sys{E})_{n+k}$ and $t \in \hat{T}(X,k))$.
We conclude that $\varphi$ is a substitution homomorphism by naturality~\eqref{constr:BtoE-surj:nat}.

To show that $\varphi$ is a projection homomorphism we reason by induction.
The case $(X,0)$ for $X \in B(\sys{E})_n$ is again trivial.
Let $X \in B(\sys{E})_{n+1}$, then
\begin{linenomath*}
\[
\hat{\tfid{(X,1)}} \jdeq \delta(X) \jdeq (W_X(X),\tfid{\overline{X}})
\]
\end{linenomath*}
by~\eqref{constr:BtopreE:idtm-base} and~\eqref{constr:preEtopreB:Gdef}.
Therefore $\varphi(\hat{\tfid{(X,1)}}) \jdeq \tfid{\varphi(X,1)}$
by definition of $\varphi$.
Suppose that, for every $m\leq n$, $Y \in B(\sys{E})_m$, $i \leq m$,
it is $\varphi(\hat{\tfid{(Y,i)}}) \jdeq \tfid{\varphi(Y,i)}$.
Let $X \in B(\sys{E})_{n+1}$ and $k \leq n$. Then
\begin{linenomath*}
\[\begin{split}
\varphi(\hat{\tfid{(X,k+1)}}) &\jdeq
\varphi(W_{(X,1)}(\hat{\tfid{(\ft(X),k)}}), \delta(X))
\\&\jdeq
\tmext{\inpar1{W_{\overline{X}}(\varphi(\hat{\tfid{(\ft(X),k)}}))}}{\tfid{\overline{X}}}
\\&\jdeq
\tmext{\inpar1{W_{\overline{X}}(\tfid{\varphi(\ft(X),k))}}}{\tfid{\overline{X}}}
\\&\jdeq
\tfid{\overline{\ft(X)}^k \circ \overline{X}}
\\&\jdeq
\tfid{\varphi(X,k+1)}
\end{split}\]
\end{linenomath*}
by~\eqref{constr:BtopreE:idtm-suc},
definition of $\varphi$ and the fact that $\varphi$ is a weakening homomorphism,
the inductive hypothesis, \cref{lem:E-CE-pair-proj},
and definition of $\varphi$ again.
Therefore $\varphi(\hat{\tfid{(X,k)}}) \jdeq \tfid{\varphi(X,k)}$
for every $X \in B(\sys{E})_{n+k}$.
This concludes the proof that $\varphi$ is an E-homomorphism.
\end{construction}

Finally we reach the main result of this section.

\begin{thm}\label{thm:BasE}
The functors $\BtoE \colon \Bsys \to \strEsys$ from \cref{lem:BtoE-fun}
and $\EtoB \colon \strEsys \to \Bsys$ from \cref{constr:EtoB-fun}
form an equivalence of categories.
\end{thm}
\begin{proof}
As the functor $\BtoE$ is fully faithful by \cref{lem:BtoE-fun}.\ref{lem:BtoE-fun:ff},
it is enough to show that $\EtoB$ is an essential section of $\BtoE$.
This holds by \cref{constr:BtoE-surj}.
\end{proof}

\section{Equivalence between B- and C-systems}
\label{sec:equiv-b-c}

In this \lcnamecref{sec:equiv-b-c}, we construct an equivalence between B-systems and C-systems,
in several steps.
We first construct an adjunction between the categories of CE-systems and of E-systems.
To this end, we construct, in~\cref{ssec:ce2e}, a functor from CE-systems to E-systems,
and, in~\cref{ssec:e2ce}, a functor in the other direction, from E-systems to CE-systems.
In~\cref{ssec:eqv-e-ce} we show that these functors form an adjunction that restricts to an equivalence when considering \emph{rooted} CE-systems.
Finally, in~\cref{ssec:eqv-b-c}, we give our equivalence between B-systems and C-systems, obtained by restricting the aforementioned equivalence to \emph{stratified} rooted CE-systems and E-systems, respectively.

\subsection{From CE-sytems to E-systems}
\label{ssec:ce2e}

\begin{defi}\label{defn:CEslice}
Let $\sys{A}$ be a CE-system.
For any $\Gamma\in\cat{C}$,
we define the \define{slice CE-system} $\CEsl{A}{\Gamma}$ as follows.
Let $\CEslcat1{A}{\Gamma}$ be the
\scat with the same objects as $\CEfam1{A}/\Gamma$
and with all arrows from $I(A)$ to $I(B)$ in $\CEcat1{A}/\Gamma$ as arrows from $A$ to $B$.
The functor $I/\Gamma \colon \CEfam1{A}/\Gamma \to \CEcat1{A}/\Gamma$
factors as an identity-on-objects $\CEslfun{\Gamma}$ %
followed by a full and faithful one as shown in the diagram below.
\begin{linenomath*}
\[\begin{tikzcd}
\CEfam1{A}/\Gamma	\ar[rr,"I/\Gamma"] \ar[dr,"\CEslfun{\Gamma}"{swap}]	&&	\CEcat1{A}/\Gamma
\\
&	\CEslcat1{A}{\Gamma}	\ar[ur]	&
\end{tikzcd}\]
\end{linenomath*}
We take $\CEslfun{\Gamma}$ to be the underlying functor of $\sys{A}/\Gamma$.
The choice of pullback squares is induced by $\sys{A}$.

We shall omit the subscript $\sys{A}$ from $\CEslcat1{A}{\Gamma}$ whenever the CE-system is clear from context.
\end{defi}

\begin{rems}\label{rmk:CEslice}
Let $\sys{A}$ be a CE-system.
\begin{enumerate}
\item
For every object $\Gamma$, the identity $\catid{\Gamma}$
is terminal in $\CEslcat1{A}{\Gamma}$.
It follows that any slice CE-system is \rtdCE.
\item
For every $f\colon\Delta\to\Gamma$ in $\cat{C}$,
the functor $\CEpbf0f \colon \cat{F}/\Gamma \to \cat{F}/\Delta$ lifts to a functor
$\CEpbf0f \colon \CEslcat0{}{\Gamma} \to \CEslcat0{}{\Delta}$ making the square below commute.
\begin{linenomath*}
\begin{equation}\label{eq:square-ce-system}
\begin{tikzcd}
\cat{F}/\Gamma	\ar[d,"\CEslfun{\Gamma}",swap]\ar[r,"\CEpbf0f"]	&	\cat{F}/\Delta	\ar[d,"\CEslfun{\Delta}"]
\\
\CEslcat0{}{\Gamma}	\ar[r,"\CEpbf0f"]	&	\CEslcat0{}{\Delta}
\end{tikzcd}
\end{equation}
\end{linenomath*}
\item\label{rmk:CEslice:hom}
For every $f\colon\Delta\to\Gamma$
the commutative square in \eqref{eq:square-ce-system}
lifts to a CE-homomorphism $\CEpbf0f \colon \CEsl{A}{\Gamma} \to \CEsl{A}{\Delta}$.
\end{enumerate}
\end{rems}

\begin{lem}\label{lem:CEhom-distributive}
Let $F \colon \sys{A} \to \sys{B}$ be a CE-homomorphism.
Then for every $f\colon\Delta\to\Gamma$ in $\CEcat1{A}$
the diagram below commutes.
\begin{linenomath*}
\begin{equation*}
\begin{tikzcd}[column sep=large]
&	\CEfam1{A}/\Gamma	\ar[dl,"\CEslfun{\Gamma}",swap]\ar[dd,swap,"{\CEpbf0f}",near end]\ar[rr,"{\hCEfam{F}/\Gamma}"]
&&	\CEfam1{B}/F\Gamma	\ar[dl,"\CEslfun{F\Gamma}",swap]\ar[dd,"{\CEpbf1{Ff}}",near end]
\\
\CEslcat1{A}{\Gamma} \ar[rr,crossing over,"{\hCEcat{F}/\Gamma}",near end] \ar[dd,swap,"{\CEpbf0f}"]
&& \CEslcat1{B}{F\Gamma}
&\\
&	\CEfam1{A}/\Delta	\ar[dl,"\CEslfun{\Delta}"]\ar[rr,"{\hCEfam{F}/\Delta}",near start]
&&	\CEfam1{B}/F\Delta	\ar[dl,"\CEslfun{F\Delta}"]
\\
\CEslcat1{A}{\Delta} \ar[rr,swap,"{F/\Delta}"]	&& \CEslcat1{B}{\hCEcat{F}\Delta}	\ar[from=uu,crossing over,"{\CEpbf1{Ff}}",near end]	&
\end{tikzcd}
\end{equation*}
\end{linenomath*}
\end{lem}

\begin{proof}
Commutativity of the back face follows from the fact that $\CEpb11{Ff}{FA} \jdeq F(\CEpb00fA)$ for every $A \in \CEfam1{A}/\Gamma$,
commutativity of the front face follows from the universal property of pullbacks,
and commutativity of the other faces is immediate.
\end{proof}

\begin{lem}\label{lem:pb_selfdistributive}
Let $\sys{A}$ be a CE-system. For every $f\colon\Delta\to\Gamma$ in $\cat{C}$
and every $g\colon A\to B$ in $\CEslcat0{}{\Gamma}$ the diagram below commutes.
\begin{linenomath*}
\begin{equation*}
\begin{tikzcd} %
&	\cat{F}/\ctxext{\Gamma}{B}	\ar[dl,"I/\ctxext{\Gamma}{B}",swap]\ar[dd,swap,"{\CEpbf0g}",near end]\ar[rr]
&&	\cat{F}/\ctxext{\Delta}{\CEpb00{f}{B}}	\ar[dl,"I/\ctxext{\Delta}{\CEpb00{f}{B}}",swap]\ar[dd,"{\CEpbf1{\CEpb00fg}}",near end]
\\
\CEslcat0{}{\ctxext{\Gamma}{B}} \ar[rr,crossing over,"{\CEpbf0f/B}",near end] \ar[dd,swap,"{\CEpbf0g}",near end]
&& \CEslcat0{}{\ctxext{\Delta}{\CEpb00fB}}
&\\
&	\cat{F}/\ctxext{\Gamma}{A}	\ar[dl,"I/\ctxext{\Gamma}{A}"]\ar[rr]
&&	\cat{F}/\ctxext{\Delta}{\CEpb00{f}{A}}	\ar[dl,"I/\ctxext{\Delta}{\CEpb00{f}{A}}"]
\\
\CEslcat0{}{\ctxext{\Gamma}{A}} \ar[rr,swap,"{\CEpbf0f/A}",near end]
&& \CEslcat0{}{\ctxext{\Delta}{\CEpb00{f}{A}}}	\ar[from=uu,crossing over,"{\CEpbf1{\CEpb00fg}}",swap,near end]	&
\end{tikzcd}
\end{equation*}
\end{linenomath*}
\end{lem}

\begin{proof}
This is \cref{lem:CEhom-distributive} applied to $\CEpbf0f$
seen as a homomorphism of CE-systems thanks to \cref{rmk:CEslice}.\ref{rmk:CEslice:hom}.
\end{proof}

\begin{prob}\label{problem:CE2E-fun}
To construct a functor $\CEtorE \colon \CEsys \to \rEsys$.
\end{prob}

\begin{construction}{problem:CE2E-fun}\label{constr:CE2E-fun}
Let \sys{A} be a CE-system with underlying functor $I\colon \cat{F} \to \cat{C}$.
The underlying category of the E-system $\CEtorE(\sys{A})$ is $\cat{F}$.
The chosen terminal object is the one in \sys{A}.
To equip $\cat{F}$ with a term structure we define,
for every $A\in\cat{F}/\Gamma$, the set
\begin{linenomath*}
\begin{equation}\label{constr:CEtoE-fun:term}
T(A)\defeq\{x\colon\Gamma\to\ctxext{\Gamma}{A}\mid I(A)\circ x\jdeq\catid{\Gamma}\}.
\end{equation}
\end{linenomath*}
We define for any $A\in\cat{F}/\Gamma$, the functor
\begin{linenomath*}
\begin{equation}\label{constr:CEtoE-fun:weak}
W_A\defeq \CEpbf0A \colon \cat{F}/\Gamma\to\cat{F}/\ctxext{\Gamma}{A}.
\end{equation}
\end{linenomath*}
Likewise, we define for any $x\in T(A)$, the functor
\begin{linenomath*}
\begin{equation}\label{constr:CEtoE-fun:sub}
S_x\defeq \CEpbf0x \colon \cat{F}/\ctxext{\Gamma}{A}\to\cat{F}/\Gamma.
\end{equation}
\end{linenomath*}
These clearly extend to functors with term structure.
We also define $\tfid{A}\colon T(W_A(A))$ by the universal property of pullbacks
as in the diagram below.
\begin{linenomath*}
\begin{equation*}
\begin{tikzcd}[row sep=2em]
\ctxext{\Gamma}{A} \arrow[drrr,bend left=15,"\catid{\ctxext{\Gamma}{A}}"]
	\arrow[ddr,bend right=15,swap,"\catid{\ctxext{\Gamma}{A}}"]
	\arrow[dr,densely dotted,near end,"{\tfid{A}}"] \\
& \ctxext{{\Gamma}{A}}{W_A(A)} \arrow[rr,"{\pi_2(I(A),A)}"{swap}] \arrow[d,fib]
&& \ctxext{\Gamma}{A} \arrow[d,fib] \\
& \ctxext{\Gamma}{A} \arrow[rr,fib,swap,"I(A)"] && \Gamma
\end{tikzcd}
\end{equation*}
\end{linenomath*}
As an immediate consequence of \cref{lem:pb_selfdistributive}, we get that
each functor $W_A$ and $S_x$ is both a weakening functor and a substitution functor.
It follows by the definitions that
weakening and substitution preserve the terms $\tfid{A}$.

It remains to verify the remaining conditions of E-systems.
\begin{itemize}
\item[\ref{def:Esys:SxW}.] To show that substitution in weakened families is constant,
  note that
\begin{linenomath*}
\begin{equation*}
S_x\circ W_A\jdeq x^\ast\circ A^\ast\jdeq (A\circ x)^\ast\jdeq (\catid{\Gamma})^\ast\jdeq \catid{\cat{C}_{\cat{F}}/\Gamma}.
\end{equation*}
\end{linenomath*}
\item[\ref{def:Esys:SIdW}.] The identity terms are neutral for pre-composition:
\begin{linenomath*}
\begin{equation*}
S_{\tfid{A}}\circ W_A/A \jdeq S_{\tfid{A}}\circ \pi_2(A,A)^\ast \jdeq (\pi_2(A,A)\circ\tfid{A})^\ast \jdeq
(\catid{\ctxext{\Gamma}{A}})^\ast \jdeq \catid{\cat{C}_{\cat{F}}/\ctxext{\Gamma}{A}}.
\end{equation*}
\end{linenomath*}
\item[\ref{def:Esys:SxId}.] The identity terms behave like identity functions: by the universal property,
  $S_x(\tfid{A})$ is the unique section of $A$ such that the square
\begin{linenomath*}
\begin{equation*}
\begin{tikzcd}[column sep=6em]
\Gamma \arrow[d,swap,"{S_x(\tfid{A})}"] \arrow[r,"{\pi_2(x,\catid{\ctxext{\Gamma}{A}})}"] & \ctxext{\Gamma}{A} \arrow[d,"{\tfid{A}}"] \\
\ctxext{\Gamma}{A} \arrow[r,swap,"{\pi_2(x,W_A(A))}"] & \ctxext{{\Gamma}{A}}{W_A(A)}
\end{tikzcd}
\end{equation*}
\end{linenomath*}
commutes. Thus, it suffices to show that this square also commutes with $x$ in the place of
$S_x(\tfid{A})$. Note that $\pi_2(x,\catid{\ctxext{\Gamma}{A}})\jdeq x$. 
Since $\ctxext{{\Gamma}{A}}{W_A(A)}$ is itself a pullback, it suffices
and it is straightforward to verify the equalities
\begin{linenomath*}
\begin{align*}
W_A(A)\circ\pi_2(x,W_A(A))\circ x & \jdeq W_A(A)\circ\tfid{A}\circ x
\\
\pi_2(A,A)\circ\pi_2(x,W_A(A))\circ x & \jdeq \pi_2(A,A)\circ\tfid{A}\circ x.
\end{align*}
\end{linenomath*}
\end{itemize}

Let now $F \colon \sys{A} \to \sys{B}$ be a CE-system homomorphism.
The underlying functor of $\CEtorE(F)$ is
$\hCEfam{F}\colon \CEfam1{A} \to \CEfam1{B}$,
which clearly preserves the choice of terminal objects,
while the action on terms is given by $\hCEcat{F}$.
This functor with term structure is both a weakening and a substitution homomorphism
because of \cref{lem:CEhom-distributive}.
Note that commutativity of the front square in the diagram in \cref{lem:CEhom-distributive}
is needed for the equations on the action on terms.
Finally, it is a projection homomorphism since it preserves identities.
\end{construction}

\begin{rem}\label{rmk:CE2Eslim}
It follows immediately from the above construction that,
for every CE-system \sys{A}, the E-system $\CEtorE(\sys{A})$ has the property that
$T(\catid{\Gamma})$ is a singleton set for every $\Gamma \in \cat{F}$.
As we shall see in \cref{cor:Etrm-homiso}, this is true for every E-system.
In fact, it will follow from \cref{thm:E-CE}\eqref{thm:E-CE:adj}
that $\CEtorE$ is essentially surjective on objects.
\end{rem}

\subsection{From E-systems to CE-systems}
\label{ssec:e2ce}

In this section we construct a functor from $\Esys$ to $\CEsys$. We proceed in several steps: In \cref{ssec:e2ce-intmor} we define the \scat of internal morphisms of an E-system. There are two kinds of morphisms in this category: internal morphisms from $A$ to $B$ in context $\Gamma$, and for any internal morphism $f:A\to B$ in context $\Gamma$ there are morphisms over $f$. There are also two kinds of composition, and in \cref{ssec:e2ce-verhorcmp} we prove an interchange law for them. In \cref{ssec:E2CE-functor} we complete the construction of the functor from $\Esys$ to $\CEsys$.

\subsubsection{The \scat of internal morphisms of an E-sytem}
\label{ssec:e2ce-intmor}

In this section we define for every E-system $\mathbb{E}$, and every context $\Gamma$ in $\mathbb{E}$,
a category $\EtrmCat{E}{\Gamma}$. This goal is accomplished in \cref{thm:internal_hom}.
The empty context $\Eroot{}$ of \sys{E}, i.e.\ a terminal object in $\cat{F}$,
allows us to
have a non-trivial category structure on the contexts of $\mathbb{E}$.
In this case, the category structure is inherited from the category
$\EtrmCatcl{E}\defeq \EtrmCat{E}{\Eroot{}}$.

\begin{defi}\label{def:Esys-intmor}
For every $A,B\in\cat{F}/\Gamma$ we define the set 
\begin{linenomath*}
\begin{equation*}
\thom{A}{B}\defeq T(\ctxwk{A}{B}).
\end{equation*}
\end{linenomath*}
An element $f\in \thom{A}{B}$ is called an
\define{internal morphism in context $\Gamma$}.
We sometimes write $\jhom{\Gamma}{A}{B}{f}$ to indicate that $f$ is
an internal morphism over $\Gamma$, or we may draw a diagram of the form
\begin{linenomath*}
\begin{equation*}
\begin{tikzcd}[column sep=tiny]
A \arrow[rr,"f"] \arrow[dr,fib] & & B \arrow[dl,fib] \\
& \Gamma
\end{tikzcd}
\end{equation*}
\end{linenomath*}
or we may omit the arrows down to $\Gamma$ and say instead that we have
a diagram in context $\Gamma$.
Note however that this is not (yet) a diagram in any category:
the double-head arrows are arrows in $\Efam{}$,
but the other ones are just elements in some $\thom{A}{B}$.
\end{defi}

\begin{rem}
Note that $\thom{\catid{\Gamma}}{A}\jdeq T(A)$ for any $A\in\cat{F}/\Gamma$, 
because we have 
$W_{\catid{\ctxext{\Gamma}{A}}}\jdeq \catid{\cat{F}/\ctxext{\Gamma}{A}}$.

Note also that $\thom{\ctxext{A}{P}}{B}\jdeq \thom{P}{\ctxwk{A}{B}}$ 
for any $P\in\cat{F}/\ctxext{\Gamma}{A}$ and $B\in\cat{F}/\Gamma$,
because $W_{\ctxext{A}{P}}\jdeq W_P\circ W_A$.
Once we have established a \scat of which the morphisms are given by 
$\thom{\blank}{\blank}$, we therefore get that
\begin{linenomath*}
\begin{equation*}
\ctxext{A}{(\blank)}\dashv W_A.
\end{equation*}
\end{linenomath*}
The right adjoint to weakening by $A$, if it exists, will be the dependent
product $\Pi_A$. 
\end{rem}

\begin{defi}\label{def:Esys-precmp}
Let $A,B\in\cat{F}/\Gamma$.
For any $f\in\thom{A}{B}$ we define the \define{pre-composition} E-homomorphism
\begin{linenomath*}
\begin{equation*}
f^\ast \defeq S_f\circ W_A/B : \sys{E}/\ctxext{\Gamma}{B}\to \sys{E}/\ctxext{\Gamma}{A}.
\end{equation*}
\end{linenomath*}
We shall denote the action of $\Epcmpf0f$ on a family $Q \in \cat{F}/\ctxext{\Gamma}{B}$ as $\jcomp{}{f}{Q}$.
Similarly, for every $C \in \cat{F}/\Gamma$, we shall write $\jcomp{}{f}{g}$
for the action of $\Epcmpf0f$ on $g \in \thom{B}{C} \jdeq T(W_B(C))$.
\end{defi}

\begin{lem}\label{lem:Ehpres-pcmp}
Let $F \colon \sys{E} \to \sys{D}$ be an E-homomorphism.
Then for every $\jhom{\Gamma}{A}{B}{f}$ in $\sys{E}$,
the square of E-homomorphisms below commutes.
\begin{linenomath*}
\[\begin{tikzcd}[column sep=3em]
\sys{E}/\ctxext{\Gamma}{B}	\ar[r,"F/\ctxext{\Gamma}{B}"]	\ar[d,"\Epcmpf0f",swap]
&	\sys{D}/\ctxext{F\Gamma}{FB}	\ar[d,"\Epcmpf1{Ff}"]
\\
\sys{E}/\ctxext{\Gamma}{A}	\ar[r,"F/\ctxext{\Gamma}{A}"]	&	\sys{D}/\ctxext{F\Gamma}{FA}
\end{tikzcd}\]
\end{linenomath*}
\end{lem}

\begin{proof}
As $F$ is both a weakening and a substitution homomorphism, it is
\begin{linenomath*}
\begin{align*}
(F/\ctxext{\Gamma}{A}) \circ \Epcmpf0f &\jdeq (F/\ctxext{\Gamma}{A}) \circ S_f \circ (W_A/B) \jdeq S_{Ff} \circ (F/\ctxext{\ctxext{\Gamma}{A}}{W_A(B)}) \circ (W_A/B)
\\
&\jdeq S_{Ff} \circ (W_{FA}/FB) \circ (F/\ctxext{\Gamma}{B}) \jdeq \Epcmpf1{Ff} \circ (F/\ctxext{\Gamma}{B}).\qedhere
\end{align*}
\end{linenomath*}
\end{proof}

\begin{defi}
Let $A,B\in\cat{F}/\Gamma$, $Q \in \cat{F}/\ctxext{\Gamma}{A}$ and $R \in \cat{F}/\ctxext{\Gamma}{B}$.
For every $f\in\thom{A}{B}$ we define
\[
\thomd{f}{Q}{R} \defeq \thom{Q}{\jcomp{}{f}{R}}.
\]
\end{defi}

\begin{rems}
\hfill
\begin{enumerate}
\item
The terms $\cprojfstf{A}{P}$ and $\cprojsndf{A}{P}$ from \cref{def:E-proj}
are internal morphisms:
\begin{linenomath*}
\[
\cprojfstf{A}{P} \in \thom{\ctxext{A}{P}}{A}
\qquad \text{and}\qquad
\cprojsndf{A}{P} \in \thomd{\cprojfstf{A}{P}}{\catid{\ctxext{A}{P}}}{P}.
\]
\end{linenomath*}
\item
Note that for $\jhom{\Gamma}{B}{C}{g}$, 
we have $\jcomp{}{f}{g}\in T(S_f(W_A/B(W_B(C))))$, whereas we would like that
$\jcomp{}{f}{g}\in\thom{A}{C}$. More generally, we can show that
\begin{linenomath*}
\begin{equation*}
S_f\circ (W_A/B)\circ W_B\jdeq W_A.
\end{equation*}
\end{linenomath*}
Since weakening is a weakening homomorphism, we have
\begin{linenomath*}
\begin{equation*}
S_f\circ (W_A/B)\circ W_B\jdeq S_f\circ W_{W_A(B)}\circ W_A.
\end{equation*}
\end{linenomath*}
By condition~\ref*{def:Esys:SxW} in \cref{def:Esys} we get that
\begin{linenomath*}
\begin{equation*}
S_f\circ W_{W_A(B)}\circ W_A\jdeq W_A.
\end{equation*}
\end{linenomath*}
\end{enumerate}
\end{rems}

\begin{rem}
Note that condition \ref*{def:Esys:SIdW} in \cref{def:Esys} asserts precisely
that ${(\tfid{A})}^\ast\jdeq \catid{\cat{F}/\ctxext{\Gamma}{A}}$ for any $A\in\cat{F}/\Gamma$. In
particular, it follows that $g\circ\tfid{A}\jdeq g$ for any $g\in\thom{A}{B}$
\end{rem}

\begin{lem}\label{lem:compcomp}
For any $f\in\thom{A}{B}$ and $g\in\thom{B}{C}$ we have $f^\ast\circ g^\ast\jdeq (\jcomp{}{f}{g})^\ast$. 
\end{lem}

\begin{proof}
\begin{linenomath*}
\begin{align*}
f^\ast\circ g^\ast & \jdeq S_f\circ (W_A/B) \circ S_g \circ (W_B/C)\\
& \jdeq S_f\circ S_{W_A(g)}\circ(W_A/\ctxext{B}{W_B(C)})\circ (W_B/C) \\
& \jdeq S_{S_f(W_A(g))}\circ (S_f/W_A(W_B(C)))\circ (W_A/\ctxext{B}{W_B(C)}) \circ W_B/C \\
& \jdeq S_{S_f(W_A(g))}\circ ((S_f\circ (W_A/B) \circ W_B)/C) \\
& \jdeq S_{S_f(W_A(g))}\circ ((S_f\circ W_{W_A(B)}\circ W_A)/C) \\
& \jdeq S_{S_f(W_A(g))}\circ W_A/C \\
& \jdeq (\jcomp{}{f}{g})^\ast.\qedhere
\end{align*}
\end{linenomath*}
\end{proof}

\begin{thm}\label{thm:internal_hom}
\hfill
\begin{enumerate}
\item \label{item:scat-of-internals}
For every E-system $\sys{E}$ and every object $\Gamma$
in its underlying \scat $\cat{F}$,
objects of $\cat{F}/\Gamma$ and internal morphisms of \sys{E} over $\Gamma$
form a \scat $\EtrmCat{E}{\Gamma}$.
\item \label{item:ehom-induces-functor}
Every E-homomorphism $F \colon \sys{E} \to \sys{D}$ induces a functor
$\Ehtrm{F}{\Gamma} \colon \EtrmCat{E}{\Gamma}\to \EtrmCat{D}{F(\Gamma)}$
for every $\Gamma$ in $\Efam{E}$.
\end{enumerate}
\end{thm}

\begin{proof}~
\begin{itemize}
\item [\ref{item:scat-of-internals}.]
  For $A,B \in \cat{F}/\Gamma$,
the set of arrows from $A$ to $B$ is $\thom{A}{B}$.
The fact that composition is associative is a direct corollary of
\cref{lem:compcomp}. The axiom $(\tfid{A})^\ast\jdeq\catid{\ctxext{\Gamma}{A}}$
implies that the identity morphisms satisfy the right identity law.
It remains to show that $\jcomp{}{\tfid{B}}{f}\jdeq f$, which is a simple calculation:
\begin{linenomath*}
\begin{equation*}
\jcomp{}{\tfid{B}}{f}
\jdeq
S_f \circ W_A (\tfid{B})
\jdeq
S_f(\tfid{W_A B})
\jdeq
f.
\end{equation*}
\end{linenomath*}
\item [\ref{item:ehom-induces-functor}.]
The action of $\Ehtrm{F}{\Gamma}$ on arrows is given by the term structure of $F$.
Functoriality of $\Ehtrm{F}{\Gamma}$ follows from \cref{lem:Ehpres-pcmp} and the fact that $F$ is a projection homomorphism.
\qedhere
\end{itemize}
\end{proof}

Now that we have a category structure,
we can state and prove the following consequence of  \cref{thm:pairing}.

\begin{cor}\label{cor:Etrm-homiso}
\hfill
\begin{enumerate}
\item\label{cor:Etrm-homiso:sect}
Let $A \in \cat{F}/\Gamma$ and $Q\in\cat{F}/\ctxext{\Gamma}{B}$,
then for every $f\in\thom{A}{B}$ there is a bijection
\begin{linenomath*}
\[\begin{tikzcd}[row sep=1ex,column sep=3em]
\varphi \colon T(\jcomp{A}{f}{Q})	\ar[r,"\sim"]
&	\left\{h\in\thom{A}{\Efcmp{B}{Q}}\mid \jcomp{A}{h}{\cprojfstf{B}{Q}}\jdeq f\right\}.
\end{tikzcd}\]
\end{linenomath*}
given by $\varphi(t) \jdeq \tmext{f}{t}$.
\item\label{cor:Etrm-homiso:singl}
For every object $\Gamma$,
$T(\catid{\Gamma}) \jdeq \{\tfid{\catid{\Gamma}}\}$.
\end{enumerate}
\end{cor}

\begin{proof}~
  \begin{itemize}
  \item [\ref*{cor:Etrm-homiso:sect}.]
    \Cref{thm:pairing} yields the following bijection:
    \begin{linenomath*}
\begin{align*}
\thom{A}{\ctxext{B}{Q}} 
& \jdeq T(\ctxwk{A}{(\Efcmp{B}{Q})}) \\
& \jdeq T(\Efcmp{\ctxwk{A}{B}}{\ctxwk{A}{Q}}) \\
& \cong \coprod_{f\in T(\ctxwk{A}{B})}T(\subst{f}{\ctxwk{A}{Q}})
\\
& \jdeq \coprod_{f\in\thom{A}{B}}T(\jcomp{A}{f}{Q}).
\end{align*}
\end{linenomath*}
Also, we find $\cprojfst{\ctxwk{A}{B}}{\ctxwk{A}{Q}}{h}\jdeq
\subst{h}{\ctxwk{A}{\cprojfstf{B}{Q}}}\jdeq
\jcomp{A}{h}{\cprojfstf{B}{Q}}$.

\item [\ref*{cor:Etrm-homiso:singl}.]
The above bijection becomes in this case
\begin{linenomath*}
\[
T(\catid{\ctxext{\Gamma}{A}})
\cong
\{h\in\thom{A}{A}\mid
	\jcomp{A}{h}{\cprojfstf{A}{\catid{\ctxext{\Gamma}{A}}}} \jdeq \tfid{A}\}
\jdeq
\{\tfid{A}\}
\]
\end{linenomath*}
where the second equality follows from
$\cprojfstf{A}{\catid{\ctxext{\Gamma}{A}}} \jdeq \tfid{A}$.
Since $\catid{\Gamma} \jdeq W_{\catid{\Gamma}}(\catid{\Gamma})$,
the only element in $T(\catid{\Gamma})$ is $\tfid{\catid{\Gamma}}$.
\qedhere
\end{itemize}
\end{proof}

\begin{thm}\label{thm:precomp_by_proj}
Let $A\in\cat{F}/\Gamma$ and $P\in\cat{F}/\ctxext{\Gamma}{A}$.
Precomposition with $\cprojfstf{A}{P}$ is weakening by $P$, i.e.
\begin{linenomath*}
\[\begin{tikzcd}[column sep=6em]
\sys{E}/\ctxext{\Gamma}{A} \arrow[r,"\inpar1{\cprojfstf{A}{P}}^\ast \jdeq\, W_P"]
&	\sys{E}/\ctxext{\ctxext{\Gamma}{A}}{P}
\end{tikzcd}\]
\end{linenomath*}
\end{thm}

\begin{proof}
\begin{linenomath*}
\begin{align*}
\inpar1{\cprojfstf{A}{P}}^\ast
&\jdeq S_{\cprojfstf{A}{P}} \circ W_{\ctxext{A}{P}}/A
\\&\jdeq S_{\ctxwk{P}{\tfid{A}}} \circ W_P/W_A(A) \circ W_A/A
\\&\jdeq W_P \circ S_{\tfid{A}} \circ W_A/A
\\&\jdeq W_P
\qedhere
\end{align*}
\end{linenomath*}
\end{proof}

We conclude this section
with a description of the projections and the pairing operation of an E-system
in the image of the functor $\CEtorE$ from \cref{constr:CE2E-fun}
in terms of the underlying CE-system structure.

\begin{lem}\label{lem:E-CE-pair-proj}
Let \sys{A} be a CE-system and consider the E-system $\sys{E} :\jdeq \CEtorE(\sys{A})$.
For every object $\Gamma$,
every $A \in \CEfam0{}/\Gamma$, $P \in \CEfam0{}/\ctxext{\Gamma}{A}$,
it is
\begin{linenomath*}
\begin{align*}
\cprojfstf{A}{P} &\jdeq \langle \catid{\ctxext{\ctxext{\Gamma}{A}}{P}} , P \rangle
\in \CEcat1{A}(\ctxext{\ctxext{\Gamma}{A}}{P}
			,\ctxext{\ctxext{\ctxext{\Gamma}{A}}{P}}{\ctxwk{\ctxext{A}{P}}{A}})
\\[1ex]
\cprojsndf{A}{P} &\jdeq \langle \catid{\ctxext{\ctxext{\Gamma}{A}}{P}} , \catid{\ctxext{\ctxext{\Gamma}{A}}{P}} \rangle
\in \CEcat1{A}(\ctxext{\ctxext{\Gamma}{A}}{P}
			,\ctxext{\ctxext{\ctxext{\Gamma}{A}}{P}}{\ctxwk{P}{P}})
\end{align*}
\end{linenomath*}
and, for every
$x \in T(A)$ and $u \in T(S_x P)$, it is
\begin{linenomath*}
\[
\tmext{x}{u} \jdeq \CEqar{x}{P} \circ u \in \CEcat1{A}(\Gamma,\ctxext{\ctxext{\Gamma}{A}}{P}).
\]
\end{linenomath*}
\end{lem}

\begin{proof}
The first two claims follow immediately from \cref{def:E-proj}
and the definitions in \cref{constr:CE2E-fun}.
The third claim follows from commutativity of the front-left face
in the diagram below.
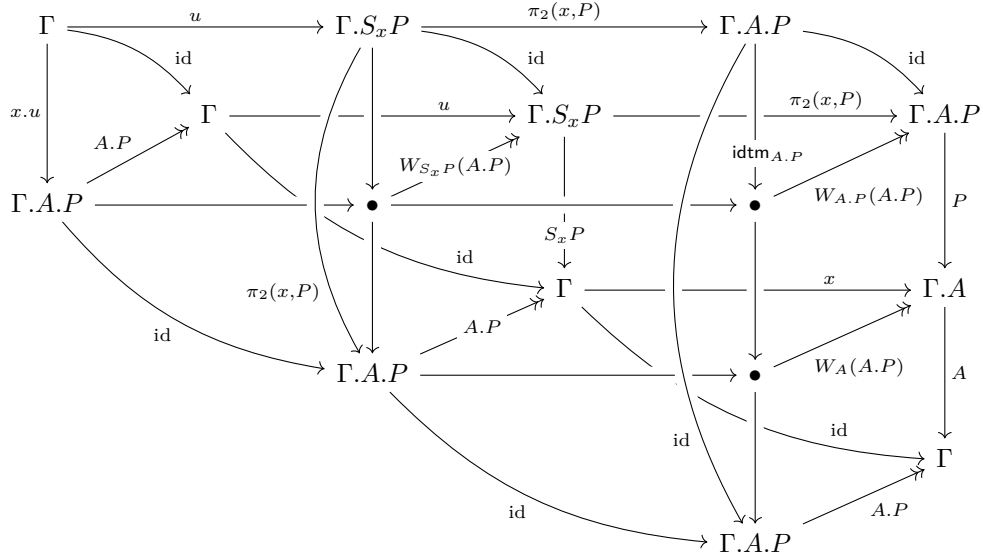
\begin{linenomath*}
\[\begin{tikzcd}[column sep=large]
\Gamma	\ar[dd,"\tmext{x}{u}"{swap}] \ar[rr,"u"] \ar[dr,bend left=2em,"\catid{}"{pos=.8}]
&&	\ctxext{\Gamma}{S_xP}	\ar[rr,"{\CEqar{x}{P}}"] \ar[dr,bend left=2em,"\catid{}"{pos=.8}]
&&	\ctxext{\ctxext{\Gamma}{A}}{P}	\ar[dr,bend left=2em,"\catid{}"{pos=.8}]
&\\
&	\Gamma	\ar[ddrr,bend right=2em,"\catid{}"{near end}] \ar[rr,"u"{near end}]
&&	\ctxext{\Gamma}{S_xP}	\ar[dd,"S_xP"{description,near end}] \ar[rr,"\CEqar{x}{P}"{near end}]
&&	\ctxext{\ctxext{\Gamma}{A}}{P}	\ar[dd,"P"]
\\
\ctxext{\ctxext{\Gamma}{A}}{P}
	\ar[ur,fib,"\ctxext{A}{P}"] \ar[ddrr,bend right=2em,"\catid{}"{swap}]
	\ar[rr,crossing over]
&&	\bullet
	\ar[from=uu,crossing over] \ar[dd,crossing over]
	\ar[ur,fib,"W_{S_xP}(\ctxext{A}{P})"{description}] \ar[rr,crossing over]
&&	\bullet
	\ar[from=uu,crossing over,"\quad\tfid{\ctxext{A}{P}}"{description,near end}] \ar[ur,fib,"W_{\ctxext{A}{P}}(\ctxext{A}{P})"{swap,near start}]
&\\
&&&	\Gamma	\ar[ddrr,bend right=2em,"\catid{}"{near end}] \ar[rr,"x"{near end}]
&&	\ctxext{\Gamma}{A}	\ar[dd,"A"]
\\
&&	\ctxext{\ctxext{\Gamma}{A}}{P}
		\ar[from=uuuu,bend right,crossing over,"{\CEqar{x}{P}}"{swap,near end}]
		\ar[ur,fib,"\ctxext{A}{P}"{description}] \ar[ddrr,bend right=2em,"\catid{}"{swap}] \ar[rr,crossing over]
&&	\bullet	\ar[from=uu,crossing over] \ar[dd,crossing over] \ar[ur,fib,"W_A(\ctxext{A}{P})"{swap,near start}]
&\\
&&&&&	\Gamma
\\
&&&&	\ctxext{\ctxext{\Gamma}{A}}{P}
	\ar[ur,fib,"\ctxext{A}{P}"{swap}]
	\ar[from=uuuuuu,bend right,crossing over,"\catid{}"{swap,pos=.8}]	&
\end{tikzcd}\]
\end{linenomath*}
This diagram commutes by definition, in the sense that
every square not involving the top row is a chosen pullback in $\sys{A}$,
and the remaining part commutes by definition of
$\tfid{\ctxext{A}{P}}$ and $\tmext{x}{u}$ in \cref{constr:CE2E-fun} and \cref{def:E-pair}, respectively.
In this diagram we drop occurrences of the functor $\CEfun0{}$
and freely use notation from the E-system $\CEtorE(\sys{A})$ to increase readability.
\end{proof}

\subsubsection{The interchange laws}
\label{ssec:e2ce-verhorcmp}

We are now in the position to define vertical and horizontal composition, and
prove properties of them.
In particular, we conclude the section showing in \cref{thm:prjsquare}
that every pair $\jhom{\Gamma}{A}{B}{f}$ and $\jhomd{\Gamma}{A}{B}{f}{P}{Q}{F}$ induces a morphism,
\ie a commuting square, from $\cprojfstf{A}{P}$ to $\cprojfstf{B}{Q}$.

\begin{defi}
Let $\jhom{\Gamma}{A}{B}{f}$ and $\jhomd{\Gamma}{A}{B}{f}{P}{Q}{F}$. Then we
define
\begin{linenomath*}
\begin{equation*}
\jhom{\Gamma}{\ctxext{A}{P}}{\ctxext{B}{Q}}{%
	\jvcomp{P}{f}{F} :\jdeq \tmext{(\ctxwk{P}{f})}{F}}.
\end{equation*}
\end{linenomath*}

This is well defined:
we have $\ctxwk{P}{(\jcomp{}{f}{Q})} \jdeq \jcomp{}{(\ctxwk{P}{f})}{Q} \jdeq S_{\ctxwk{P}{f}}\circ W_{\ctxext{A}{P}}/B(Q)$
since $W_P$ is a substitution homomorphism,
therefore $\jvcomp{}{f}{F}\in T(\ctxext{(\ctxwk{\ctxext{A}{P}}{B})}{W_{\ctxext{A}{P}}/B(Q)}) \jdeq T(\ctxwk{\ctxext{A}{P}}{(\ctxext{B}{Q})})$
by \cref{def:E-pair} and functoriality of $W_{\ctxext{A}{P}}$.

Whenever we say that we have a diagram of the form
\begin{linenomath*}
\begin{equation*}
\begin{tikzcd}
R \arrow[r,"f_2"] \arrow[d,fib] &
S \arrow[d,fib] \\
P \arrow[r,"f_1"] \arrow[d,fib] &
Q \arrow[d,fib] \\
A \arrow[r,"f_0"] &
B
\end{tikzcd}
\end{equation*}
\end{linenomath*}
we mean that we have
$f_0\in\thom{A}{B}$, $f_1\in\thomd{f_0}{P}{Q}$ and 
$f_2\in\thomd{\jvcomp{P}{f_0}{f_1}}{R}{S}$.
\end{defi}

\begin{lem}\label{lem:Ehpres-vcomp}
Let $H:\mathbb{E}\to\mathbb{D}$ be an E-homomorphism.
For every $\jhom{\Gamma}{A}{B}{f}$ and $\jhomd{\Gamma}{A}{B}{f}{P}{Q}{F}$ it is
\begin{linenomath*}
\begin{equation*}
H(\jvcomp{P}{f}{F})\jdeq \jvcomp{H(P)}{H(f)}{H(F)}.
\end{equation*} 
\end{linenomath*}
\end{lem}

\begin{proof}
$H(\jvcomp{Q}{f}{F}) \jdeq H(\tmext{\ctxwk{Q}{f}}{F})
\jdeq \tmext{\ctxwk{HQ}{Hf}}{HF}
\jdeq \jvcomp{H(Q)}{H(f)}{H(F)}$.
\end{proof}

\begin{lem}
Vertical composition is associative.
\end{lem}

\begin{proof}
Consider the diagram
\begin{linenomath*}
\begin{equation*}
\begin{tikzcd}
R \arrow[r,"f_2"] \arrow[d,fib] &
S \arrow[d,fib] \\
P \arrow[r,"f_1"] \arrow[d,fib] &
Q \arrow[d,fib] \\
A \arrow[r,"f_0"] &
B
\end{tikzcd}
\end{equation*}
\end{linenomath*}
in context $\Gamma$.
Because weakening distributes over term extension, and term extension is
associative, we have
\begin{linenomath*}
\begin{align*}
\jvcomp{R}{(\jvcomp{P}{f_0}{f_1})}{f_2}
  & \jdeq
\tmext{\ctxwk{R}{(\tmext{\ctxwk{P}{f_0}}{f_1})}}{f_2}
  \\
  & \jdeq
\tmext{(\tmext{\ctxwk{R}{\ctxwk{P}{f_0}}}{\ctxwk{R}{f_1}})}{f_2}
  \\
  & \jdeq
\tmext{\ctxwk{\ctxext{P}{R}}{f_0}}{(\tmext{\ctxwk{R}{f_1}}{f_2})}
  \tag{By \cref{cor:tmext_assoc}}
  \\
  & \jdeq
\jvcomp{{P}{R}}{f_0}{(\jvcomp{R}{f_1}{f_2})}.\qedhere
\end{align*}
\end{linenomath*}
\end{proof}

\begin{defi}
Let $\jhom{\Gamma}{A}{B}{f}$ and $\jhomd{\Gamma}{A}{B}{f}{P}{Q}{F}$. Then we
define the E-homomorphism 
\begin{linenomath*}
\begin{equation*}
F^\bullet\defeq F^\ast\circ (f^\ast/Q):
\mathbb{E}/\ctxext{\ctxext{\Gamma}{B}}{Q}\to\mathbb{E}/\ctxext{\ctxext{\Gamma}{A}}{P}.
\end{equation*}
\end{linenomath*}
The infix notation of $F^\bullet$ is taken to be $\jfcomp{\blank}{\blank}{\blank}{F}{\blank}$.
\end{defi}

\begin{lem}\label{lem:three-composition}
Let $\jhom{\Gamma}{A}{B}{f}$ and $\jhomd{\Gamma}{A}{B}{f}{P}{Q}{F}$. Then we
have the equality
\begin{linenomath*}
\begin{equation*}
F^\bullet\jdeq (\jvcomp{P}{f}{F})^\ast.
\end{equation*}
\end{linenomath*}
\end{lem}

\begin{proof}
\begin{linenomath*}
\begin{align*}
F^\ast\circ (f^\ast/Q)
  & \jdeq
S_F\circ W_P\circ S_f/(W_A(Q))\circ W_A/\ctxext{B}{Q}
  \\
  & \jdeq
S_F\circ S_{W_P(f)}/W_P(W_A(Q))\circ W_P/W_A(\ctxext{B}{Q})\circ W_A/\ctxext{B}{Q}
  \\
  & \jdeq
S_F\circ S_{W_P(f)}/W_P(W_A(Q))\circ W_{\ctxext{A}{P}}/\ctxext{B}{Q}
  \\
  & \jdeq
S_{\tmext{W_P(f)}{F}}\circ W_{\ctxext{A}{P}}/\ctxext{B}{Q}
  \tag{By \cref{subst_by_tmext}}
  \\
  & \jdeq
(\jvcomp{P}{f}{F})^\ast.\qedhere
\end{align*}
\end{linenomath*}
\end{proof}

In the next theorem we prove the interchange law of horizontal and vertical composition.
Its proof uses the following fact.

\begin{lem}\label{lem:compW_W}
Let $f\in\thom{A}{B}$ be an internal morphism in context $\Gamma$. Then one has
\begin{linenomath*}
\begin{equation*}
f^\ast\circ W_B\jdeq W_A.
\end{equation*}
\end{linenomath*}
\end{lem}

\begin{proof}
The proof is a simple calculation:
\begin{linenomath*}
\begin{equation*}
f^\ast\circ W_B\jdeq S_f\circ W_A/B\circ W_B\jdeq S_f\circ W_{W_A(B)}\circ W_A
\jdeq W_A.\qedhere
\end{equation*}
\end{linenomath*}
\end{proof}

\begin{thm}\label{thm:interchange}
Consider the diagram
\begin{linenomath*}
\begin{equation*}
\begin{tikzcd}
P \arrow[r,"F"] \arrow[d,fib] &
Q \arrow[r,"G"] \arrow[d,fib] &
R \arrow[d,fib] \\
A \arrow[r,"f"] &
B \arrow[r,"g"] &
C
\end{tikzcd}
\end{equation*}
\end{linenomath*}
in context $\Gamma$. Then the equality
\begin{linenomath*}
\begin{equation*}
{\jcomp{{A}{P}}{(\jvcomp{P}{f}{F})}{(\jvcomp{Q}{g}{G})}}
  \jdeq
{\jvcomp{P}{(\jcomp{A}{f}{g})}{(\jfcomp{A}{f}{P}{F}{G})}}
\end{equation*}
\end{linenomath*}
of morphisms from $\ctxext{A}{P}$ to $\ctxext{C}{R}$ in context $\Gamma$ holds.
\end{thm}

\begin{proof}
By \cref{lem:three-composition}, we have
\begin{linenomath*}
\begin{align*}
\jcomp{{A}{P}}{(\jvcomp{P}{f}{F})}{(\jvcomp{Q}{g}{G})}
  & \jdeq
F^\ast \circ (f^\ast/Q) (\tmext{(\ctxwk{Q}{g})}{G})
  \tag{By \cref{lem:three-composition}}
  \\
  & \jdeq
F^\ast(\tmext{(\ctxwk{\jcomp{A}{f}{Q}}{\jcomp{A}{f}{g}})}{%
(f^\ast/Q(G))})
  \\
  & \jdeq
\tmext{(F^\ast(\ctxwk{\jcomp{A}{f}{Q}}{\jcomp{A}{f}{g}}))}{(F^\ast\circ f^\ast/Q(G))}
  \\
  & \jdeq
\tmext{(F^\ast(\ctxwk{\jcomp{A}{f}{Q}}{\jcomp{A}{f}{g}}))}{(\jfcomp{A}{f}{P}{F}{G})}
  \\
  & \jdeq
\tmext{(\ctxwk{P}{\jcomp{A}{f}{g}})}{(\jfcomp{A}{f}{P}{F}{G})}
  \tag{By \cref{lem:compW_W}}
  \\
  & \jdeq
\jvcomp{P}{(\jcomp{A}{f}{g})}{(\jfcomp{A}{f}{P}{F}{G})}.
\qedhere
\end{align*}
\end{linenomath*}
\end{proof}

\begin{thm}
Consider the diagram
\begin{linenomath*}
\begin{equation*}
\begin{tikzcd}
P \arrow[r,"F"] \arrow[d,fib] &
Q \arrow[r,"G"] \arrow[d,fib] &
R \arrow[d,fib] \\
A \arrow[r,"f"] &
B \arrow[r,"g"] &
C
\end{tikzcd}
\end{equation*}
\end{linenomath*}
in context $\Gamma$. 
Then $F^\bullet\circ G^\bullet\jdeq (\jfcomp{A}{f}{P}{F}{G})^\bullet$.
In other words the composition $\jfcomp{}{}{}{\blank}{\blank}$ is associative.
\end{thm}

\begin{proof}
\begin{linenomath*}
\begin{align*}
F^\bullet\circ G^\bullet
  & \jdeq
(\jvcomp{P}{f}{F})^\ast\circ(\jvcomp{Q}{g}{G})^\ast 
  \tag{By \cref{lem:three-composition}}\\
  & \jdeq
(\jcomp{\ctxext{A}{P}}{\jvcomp{P}{f}{F}}{\jvcomp{Q}{g}{G}})^\ast 
  \tag{By \cref{lem:compcomp}}\\
  & \jdeq
(\jvcomp{}{(\jcomp{A}{f}{g})}{(\jfcomp{A}{f}{P}{F}{G})})^\ast 
  \tag{By \cref{thm:interchange}} \\
  & \jdeq
(\jfcomp{A}{f}{P}{F}{G})^\bullet.
  \tag{By \cref{lem:three-composition}}
\end{align*}
\end{linenomath*}
\end{proof}

\begin{thm}\label{thm:prjsquare}
Let $\jhom{\Gamma}{A}{B}{f}$ and $\jhomd{\Gamma}{A}{B}{f}{P}{Q}{F}$. Then
$\jvcomp{P}{f}{F}$ is the unique morphism from $\ctxext{A}{P}$ to $\ctxext{B}{Q}$
with the property that both the diagram
\begin{linenomath*}
\begin{equation*}
\begin{tikzcd}
\ctxext{A}{P}
  \ar{r}{\jvcomp{P}{f}{F}}
  \ar{d}[swap]{\cprojfstf{A}{P}}
& \ctxext{B}{Q}
  \ar{d}{\cprojfstf{B}{Q}}
  \\
A \ar{r}[swap]{f}
& B
\end{tikzcd}
\end{equation*}
\end{linenomath*}
commutes and $\jcomp{}{(\jvcomp{P}{f}{F})}{\cprojsndf{B}{Q}}\jdeq F$.
\end{thm}

\begin{proof}
We first note that
\begin{linenomath*}
\begin{align*}
\jcomp{}{(\jvcomp{P}{f}{F})}{\cprojfstf{B}{Q}}
  & \jdeq
F^\ast \circ (f^\ast/Q) (\ctxwk{Q}{\tfid{B}})
  \tag{By \cref{lem:three-composition}} \\
  & \jdeq
F^\ast \circ \ctxwk{\jcomp{A}{f}{Q}}{(\jcomp{A}{f}{\tfid{B}})}
\\
  & \jdeq
F^\ast\circ \ctxwk{\jcomp{A}{f}{Q}}{f} \\
  & \jdeq
\ctxwk{P} f \tag{By \cref{lem:compW_W}} \\
  & \jdeq
\jcomp{\ctxext{A}{P}}{\cprojfstf{A}{P}}{f}.
  \tag{By \cref{thm:precomp_by_proj}}
\end{align*}
\end{linenomath*}
Also, we have
\begin{linenomath*}
\begin{align*}
\jcomp{}{(\jvcomp{P}{f}{F})}{\cprojsndf{B}{Q}}
  & \jdeq
F^\ast\circ (f^\ast/Q)(\tfid{Q})
  \tag{By \cref{lem:three-composition}}
  \\
  & \jdeq
\jcomp{P}{F}{\tfid{\jcomp{A}{f}{Q}}}
  \\
  & \jdeq
F.
\end{align*}
\end{linenomath*}
Thus, we conclude that $\jvcomp{P}{f}{F}$ has indeed the stated property. For
the uniqueness, let $G:\ctxext{A}{P}\to\ctxext{B}{Q}$ be a morphism such that
$\jcomp{}{G}{\cprojfstf{B}{Q}}\jdeq\jcomp{}{\cprojfstf{A}{P}}{f}$ and
$\jcomp{}{G}{\cprojsndf{B}{Q}}\jdeq F$. Then it follows that
\begin{linenomath*}
\begin{equation*}
G \jdeq \tmext{\inpar1{\jcomp{}{\cprojfstf{A}{P}}{f}}}{F}
  \jdeq \tmext{\ctxwk{P}{f}}{F}
  \jdeq \jvcomp{P}{f}{F}.\qedhere
\end{equation*}
\end{linenomath*}
\end{proof}

\subsubsection{The functor from E-systems to CE-systems}\label{ssec:E2CE-functor}

Let $\Esysptd$ be the category of pointed E-systems:
objects are pairs $(\sys{E},\Gamma)$ of an E-systems \sys{E} and an object $\Gamma$ in its underlying \scat,
and arrows are E-homomorphisms that preserve the distinguished object.
There is an evident forgetful functor $\Esysptd \to \Esys$
together with an embedding $\rEtoEp \colon \Esys \hookrightarrow \Esysptd$
which picks out the terminal object of an E-system.

\begin{prob}\label{prob:EtoCE-fun}
To construct a functor $\EptoCE \colon \Esysptd \to \CEsys$.
\end{prob}

\begin{construction}{prob:EtoCE-fun}\label{constr:EtoCE-fun}
Let $(\sys{E},\Gamma)$ be a pointed E-system and
consider the category of terms $\EtrmCat{E}{\Gamma}$ from \cref{thm:internal_hom}.
Define a functor $\Etrmfun{E}{\Gamma}\colon \cat{F}/\Gamma \to \EtrmCat{E}{\Gamma}$ as follows.
It is the identity on objects and maps an arrow
$Q \colon \Efcmp{A}{Q} \to A$ in $\cat{F}/\Gamma$ to $\cprojfstf{A}{Q} \in \thom{\Efcmp{A}{Q}}{A}$.
For functoriality, we compute
$\cprojfstf{A}{\catid{\ctxext{\Gamma}{A}}} \jdeq \ctxwk{\catid{\ctxext{\Gamma}{A}}}{\tfid{A}} \jdeq \tfid{A}$ and
\begin{linenomath*}
\begin{align*}
\cprojfstf{A}{Q.R}
&\jdeq \ctxwk{Q.R}{\tfid{A}} \jdeq \ctxwk{R}{(\ctxwk{Q}{\tfid{A}})}
\\&\jdeq \ctxwk{R}{\cprojfstf{A}{Q}} \jdeq \Epcmp11{\cprojfstf{Q}{R}}{\cprojfstf{A}{Q}}
\\&\jdeq \jcomp{}{\cprojfstf{Q}{R}}{\cprojfstf{A}{Q}}.
\end{align*}
\end{linenomath*}

Next, we show that $\EtrmCat{\sys{E}}{\Gamma}$ admits a functorial choice of pullbacks of arrows in the image of $\Etrmfun{\sys{E}}{}$.
Given $f \in \thom{A}{B}$ and $R$ in $\cat{F}/\ctxext{\Gamma}{B}$,
there is $\jcomp{A}{f}{R}$ in $\cat{F}/\ctxext{\Gamma}{A}$.
We define
\begin{linenomath*}
\begin{equation}\label{eq:EtoCE-qar}
\pi_2(f,R)\defeq \jvcomp{\jcomp{A}{f}{R}}{f}{\tfid{\jcomp{A}{f}{R}}}
  \colon \Efcmp{A}{\jcomp{A}{f}{R}}\to\Efcmp{B}{R}.
\end{equation}
\end{linenomath*}
Then the following diagram in $\EtrmCat{E}{\Gamma}$
\begin{linenomath*}
\begin{equation}\label{cd:intmor-pbsq}
\begin{tikzcd}[column sep=huge]
\Efcmp{A}{(\jcomp{A}{f}{R})} \arrow[r,"{\pi_2(f,R)}"] \arrow[d,swap,"{\cprojfstf{A}{\jcomp{A}{f}{R}}}"]
& \Efcmp{B}{R} \arrow[d,"{\cprojfstf{B}{R}}"]
\\
A \arrow[r,"f"] & B
\end{tikzcd}
\end{equation}
\end{linenomath*}
commutes.
The functoriality conditions follow immediately from the interchange
laws proven in \cref{ssec:e2ce-verhorcmp}.
To show that \eqref{cd:intmor-pbsq} is a pullback square,
consider a morphism $g:X\to A$ in $\EtrmCat{E}{\Gamma}$
and use the isomorphisms
\begin{linenomath*}
\begin{align*}
\{h\in\thom{X}{\Efcmp{B}{Q}}\mid \jcomp{X}{h}{\cprojfstf{B}{Q}}\jdeq\jcomp{X}{g}{f}\}
  &\cong
T(\jcomp{}{(\jcomp{X}{g}{f})}{Q}) \jdeq T(\jcomp{}{g}{(\jcomp{}{f}{Q})})
\\&\cong
\{u\in\thom{X}{\Efcmp{A}{(\jcomp{A}{f}{Q})}}\mid\jcomp{X}{u}{\cprojfstf{A}{\jcomp{A}{f}{Q}}}\jdeq g\}
\end{align*}
\end{linenomath*}
given by \cref{cor:Etrm-homiso,lem:compcomp}.

Therefore, we have constructed a CE-system $\EptoCE(\sys{E},\Gamma)$ on
$\Etrmfun{E}{\Gamma}\colon \cat{F}/\Gamma \to \EtrmCat{E}{\Gamma}$.

Let now $(\sys{D},\Delta)$ be a pointed E-system and let
$F \colon \sys{E} \to \sys{D}$ be an E-homomorphism such that $F\Gamma \jdeq \Delta$.
In particular, for every $A,B \in \Efam{E}/\Gamma$
there is a function $F \colon T(\ctxwk{A}{B}) \to T(\ctxwk{FA}{FB})$.
These functions give the action on arrows of a functor
$F_{\Gamma} \colon \EtrmCat{E}{\Gamma} \to \EtrmCat{D}{F\Gamma}$
whose action on objects is given by
$F/\Gamma \colon \Efam{E}/\Gamma \to \Efam{D}/F\Gamma$.
Functoriality of $F_{\Gamma}$ follows from the fact that $F$ is a projection homomorphism and \cref{lem:Ehpres-pcmp}.
Using \cref{lem:Ehpres-pairproj}, we see that
$F_{\Gamma} \circ \Etrmfun{E}{\Gamma} \jdeq \Etrmfun{D}{F\Gamma} \circ (F/\Gamma)$.
Finally, it follows from \cref{lem:Ehpres-pcmp} and \cref{lem:Ehpres-vcomp}
that $F_{\Gamma}$ preserves the choice of pullback squares.

We have described the action of $\EptoCE$ on objects and arrows.
Its functoriality is straightforward.
\end{construction}

We obtain a functor $\rEtoCE \colon \rEsys \to \CEsys$ defining
$\rEtoCE :\jdeq \EptoCE \circ \rEtoEp$.

\begin{rem}\label{rem:EtoCE-altern}
The CE-system $\rEtoCE(\sys{E})\jdeq\EptoCE(\sys{E},\Eroot{})$ is on the functor $\Etrmfun{\sys{E}}{\Eroot{}}\colon\cat{F}/\Eroot{}\to\EtrmCat{\sys{E}}{\Eroot{}}$.
It is also possible to have a CE-system with category of families given by $\cat{F}$ itself.
Consider the commutative square below,
where the top functor $!$ maps an object $\Gamma$ to the unique arrow $!_{\Gamma}\colon\Gamma\to\Eroot{}$.
\begin{linenomath*}
\begin{equation*} %
\begin{tikzcd}[column sep=7ex]
\Efam{\sys{E}}	\ar[d,"I"'] \ar[r,"!"]
&	\Efam{\sys{E}}/\Eroot{}	\ar[d,"\Etrmfun{\sys{E}}{\Eroot{}}"]
\\
\cat{C}	\ar[r]	&	\EtrmCat{\sys{E}}{\Eroot{}}
\end{tikzcd}
\end{equation*}
\end{linenomath*}
The left and bottom functors are obtained as the factorisation of the composite of the top and right functors into an identity-on-objects functor $I$ followed by a fully faithful one.
The above \cref{constr:EtoCE-fun} can be easily adapted to obtain a CE-system on the functor $I$.

Alternatively, one could rephrase the results in \cref{ssec:e2ce-intmor} leading to \cref{thm:internal_hom}, as happening over the terminal object of the E-system.
In this case, the version of the results over a generic object $\Gamma$ can be recovered using the slice E-system $\sys{E}/\Gamma$ from \cref{cor:sliceEsys}.

Keeping the category of families fixed in the process might seem to be an advantage of this construction.
However, as we discuss in \cref{rem:E-CE-unit-altern},
it does not seem to have any actual useful consequence on the adjunction of which $\rEtoCE$ is the left adjoint.
For this reason, and because it is not an instance of the general construction from \cref{thm:internal_hom},
we prefer to use the one given in \cref{constr:EtoCE-fun}.
\end{rem}

\begin{rem}\label{rmk:EptoCE-rooted}
For every E-system \sys{E} and every $\Gamma$,
the CE-system $\EptoCE(\sys{E},\Gamma)$ is \rtdCE.
The canonical terminal object
$\catid{\Gamma}$ of $\Efam{E}/\Gamma$ is terminal in $\EtrmCat{E}{\Gamma}$
by \cref{cor:Etrm-homiso} since for every $A \in \Efam{E}/\Gamma$
\begin{linenomath*}
\[
\thom{A}{\catid{\Gamma}} \jdeq T(W_A(\catid{\Gamma})) \jdeq T(\catid{\ctxext{\Gamma}{A}}).
\]
\end{linenomath*}
\end{rem}

Next we give the choice of pullbacks in a CE-system in the image of $\rEtoCE$
in terms of the underlying E-system structure.

\begin{lem}\label{lem:E-CE-EC}
For \sys{E} an E-sytem and $\Gamma$ an object in \sys{E},
consider the CE-system $\sys{A} :\jdeq \EptoCE(\sys{E},\Gamma)$.
For every $A \in \Efam{}/\Gamma$ and  $P,Q \in \Efam{}/\ctxext{\Gamma}{A}$ it is
\begin{linenomath*}
\[
\CEpb10{\cprojfstf{A}{P}}{Q} \jdeq \ctxwk{P}{Q}
\in \Efam{}/\ctxext{\ctxext{\Gamma}{A}}{P}
\]
\end{linenomath*}
and
\begin{linenomath*}
\[
\pi_2(\cprojfstf{A}{P},Q) \jdeq \cprojsndf{P}{\ctxwk{P}{Q}}
\in \thom{\ctxext{P}{\ctxwk{P}{Q}}}{Q}.
\]
\end{linenomath*}
\end{lem}
\begin{proof}
The first equality follows from \cref{thm:precomp_by_proj}.
For the second one:
\begin{linenomath*}
\begin{align*}
\pi_2(\cprojfstf{\catid{\Gamma}}{A},B)&\jdeq
\jvcomp{}{\cprojfstf{\catid{\Gamma}}{A}}{\tfid{\ctxwk{A}{B}}}
\\&\jdeq
\tmext{\inpar1{W_{\ctxwk{A}{B}} \ctxwk{A}{\tfid{\catid{\Gamma}}}}}
	  {\tfid{\ctxwk{A}{B}}}
\\&\jdeq
\tmext{\inpar1{W_A \ctxwk{B}{\tfid{\catid{\Gamma}}}}}{\inpar1{W_A\tfid{B}}}
\\&\jdeq
\ctxwk{A}{(\tmext{\ctxwk{B}{\tfid{\catid{\Gamma}}}}{\tfid{B}})}
\\&\jdeq
\ctxwk{A}{\inpar1{\tmext{\cprojfstf{\catid{\Gamma}}{B}}{\cprojsndf{\catid{\Gamma}}{B}}}}
\\&\jdeq
\ctxwk{A}{\tfid{B}} \jdeq \tfid{\ctxwk{A}{B}}
\\&\jdeq
\cprojsndf{A}{\ctxwk{A}{B}}.
\qedhere
\end{align*}
\end{linenomath*}
\end{proof}

\subsection{Equivalence between E-systems and CE-systems}
\label{ssec:eqv-e-ce}

In this section, we show that the functors constructed in~\cref{ssec:e2ce,ssec:ce2e} form an adjunction that, when suitably restricted, yields an equivalence of categories between rooted CE-systems and E-systems.

Specifically, we prove the following results:
\begin{thm}\label{thm:E-CE}\hfill
\begin{enumerate}
\item\label{thm:E-CE:adj}
The functor $\rEtoCE$ is left adjoint to the functor $\CEtorE$.
\begin{linenomath*}
\[\begin{tikzcd}[column sep=6em]
\CEsys	\ar[r,shift right=1ex,"\CEtorE"{swap}]
		\ar[r,phantom,"{\scriptstyle\bot}"]
&	\rEsys	\ar[l,shift right=1ex,"\rEtoCE"{swap}]
\end{tikzcd}\]
\end{linenomath*}
\item\label{thm:E-CE:lff}
The unit of the adjunction is invertible.
In particular, the left adjoint $\rEtoCE$ is full and faithful
and the right adjoint $\CEtorE$ is essentially surjective on objects.
\item\label{thm:E-CE:img}
The counit component at a CE-system $\sys{A}$ is invertible if and only if $\sys{A}$ is rooted.
\end{enumerate}
\end{thm}

\begin{cor}\label{cor:EeqvCE}
The adjoint functors $\rEtoCE$ and $\CEtorE$ induce an (adjoint) equivalence
between the category $\Esys$ of E-systems
and the category $\rCEsys$ of rooted CE-systems.
\end{cor}

\begin{proof}
The equivalence follows from \cref{thm:E-CE},
the observation in \cref{rmk:EptoCE-rooted} that, for every E-system \sys{E}, the CE-system $\rEtoCE(\sys{E})$ is rooted,
and the fact that $\rCEsys$ is a full subcategory of $\CEsys$ by \cref{rmk:CEhom-root}.
\end{proof}

To prove \cref{thm:E-CE} we construct unit and counit and prove the triangular identities.
In this proof we denote as
\begin{linenomath*}
\begin{equation}\label{sl-trm}
\begin{tikzcd}[column sep=4em]
\cat{F}/1	\ar[r,shift left=1ex,"\mathrm{d}"]
&	\cat{F}	\ar[l,shift left=.5ex,"!"]
\end{tikzcd}
\end{equation}
\end{linenomath*}
the canonical isomorphism of \scats,
for any \scat $\cat{F}$ with a terminal object $1$.
We may still leave this isomorphism implicit when doing so creates no confusion.

\begin{prob}\label{prob:E-CE-unit}
To construct, for each E-system \sys{E}, an invertible E-homomorphism $\eta_{\sys{E}} \colon \sys{E} \to \CEtorE \circ \rEtoCE(\sys{E})$,
naturally in \sys{E}.
\end{prob}

\begin{construction}{prob:E-CE-unit}\label{constr:E-CE-unit}
Let \sys{E} be an E-system %
and denote its terminal object by $\Eroot{}$.
In this proof we shall decorate with a hat the constituents
of the E-system structure of
$\hat{\sys{E}} :\jdeq \CEtorE \circ \rEtoCE(\sys{E})$.
The underlying \scat of $\hat{\sys{E}}$ is $\Efam{E}/\Eroot{}$ and,
for every $X \in (\Efam{E}/\Eroot{})/!_{\Gamma}$, we have
\begin{linenomath*}
\[
\hat{T}(X) \jdeq \left\{h \in \thom{!_{\Gamma}}{!_{\ctxext{\Gamma}{X}}} \mid \jcomp{}{h}{\cprojfstf{!_{\Gamma}}{!_{\ctxext{\Gamma}{X}}}} \jdeq \tfid{!_{\Gamma}} \right\}.
\]
\end{linenomath*}
We define $\eta_{\sys{E}}$ as the functor $!\colon \Efam{E} \to \Efam{E}/\Eroot{}$ in \eqref{sl-trm}
with term structure given by the bijections
\begin{linenomath*}
\begin{equation}\label{E-CE-unit:term}
\begin{tikzcd}[column sep=3em]
T(A)	\ar[r,"\varphi"]	&	\hat{T}(!(A))
\end{tikzcd}
\end{equation}
\end{linenomath*}
from \cref{cor:Etrm-homiso},
that is, for $A \in \Efam{E}/\Gamma$ and $t \in T(A)$,
it is $\eta_{\sys{E}}(t) :\jdeq \tmext{\tfid{!_{\Gamma}}}{t}$.
Therefore we have an invertible functor with term structure.

To conclude that this defines an invertible E-homomorphism,
we compute for $A \in \Efam{E}/\Gamma$
\begin{linenomath*}
\begin{align*}
\hat{W}_{!(A)} \circ (!/\Gamma) &\jdeq
\CEpbf1{\cprojfstf{!_{\Gamma}}{A}} \circ (!/\Gamma)
\jdeq W_A \circ (!/\Gamma)
\\&\jdeq
(!/\ctxext{\Gamma}{A}) \circ W_A,
\end{align*}
\end{linenomath*}
and for $t \in T(A)$,
\begin{linenomath*}
\begin{align*}
\hat{S}_{!(t)} \circ (!/\ctxext{\Gamma}{A})) &\jdeq
\CEpbf1{\tmext{\tfid{!_{\Gamma}}}{t}} %
\\&\jdeq
S_{\tmext{\tfid{!_{\Gamma}}}{t}} \circ (W_{!_{\Gamma}}/!_{\ctxext{\Gamma}{A}}) %
\\&\jdeq
S_t \circ (S_{\tfid{!_{\Gamma}}} \circ (W_{!_{\Gamma}}/!_{\Gamma}))/A %
\\&\jdeq
(!/\Gamma) \circ S_t,
\end{align*}
\end{linenomath*}
and finally
\begin{linenomath*}
\begin{align*}
\phi(\tfid{A}) &\jdeq
\tmext{\tfid{!_{\Gamma}}}{\tfid{A}}
\\&\jdeq
\tfid{!_{\ctxext{\Gamma}{A}}}
\\&\jdeq
\hat{\tfid{}}_{!(A)}.
\end{align*}
\end{linenomath*}

Finally, naturality in \sys{E} requires that any
E-homomorphism $F \colon \sys{E} \to \sys{D}$
commutes with $\eta$ as functors with term structures.
This follows from \cref{lem:Ehpres-pcmp,lem:Ehpres-pairproj}.
\end{construction}

\begin{rem}\label{rem:E-CE-unit-altern}
With the alternative construction for $\rEtoCE$ described in \cref{rem:EtoCE-altern},
the underlying category of $\CEtorE\circ\rEtoCE(\sys{E})$ is $\cat{F}$ itself.
In this case we could replace the isomorphism from~\eqref{sl-trm} with an identity.
However, the unit $\eta_{\sys{E}}$ would not become an identity,
as the term structures would still be different (though isomorphic).
\end{rem}

\begin{prob}\label{prob:E-CE-counit}
To construct, for each CE-system \sys{A}, a CE-homomorphism $\epsi_{\sys{A}} \colon \rEtoCE \circ \CEtorE(\sys{A}) \to \sys{A}$,
naturally in \sys{A}.
\end{prob}

\begin{construction}{prob:E-CE-counit}\label{constr:E-CE-counit}
Let \sys{A} be a CE-system and let $\sys{E} :\jdeq \CEtorE(\sys{A})$
be the associated E-system.
The underlying functor of the CE-system
$\hat{\sys{A}} :\jdeq \rEtoCE \circ \CEtorE(\sys{A})$
is $\Etrmfuncl{E} \colon \cat{F}/\Eroot{} \to \EtrmCatcl{E}$
defined in \cref{constr:EtoCE-fun}. %
As before, we decorate with a hat the constituents of
the CE-system structure of $\hat{\sys{A}}$.
For $\Gamma, \Delta$ in $\cat{F}$,
recall that 
$\thom{!_{\Delta}}{!_{\Gamma}} \jdeq
\{\Delta \overset{x}{\longrightarrow} \ctxext{\Delta}{(\CEpb00{!_{\Delta}}{!_{\Gamma}})} \mid
I(\CEpb00{!_{\Delta}}{!_{\Gamma}}) \circ x \jdeq \catid{\Delta}\}$
and let
\begin{linenomath*}
\begin{equation}\label{prob:E-CE-counit:hom}
\begin{tikzcd}[column sep=large]
\thom{!_{\Delta}}{!_{\Gamma}}
\ar[r,"\psi"]
&	\cat{C}(\Delta,\Gamma)
\end{tikzcd}\end{equation}
\end{linenomath*}
be the function that maps $x$ to the arrow
$\CEqar{I(!_{\Delta})}{!_{\Gamma}} \circ x$ of $\CEcat0{}$.
The functions $\psi$ give rise to a functor
$\Psi \colon \EtrmCatcl{E} \to \CEcat0{}$
as follows.
It maps identities to identities since
the identity on $\Gamma$ in $\EtrmCatcl{E}$
is the only $h \in \thom{!_{\Gamma}}{!_{\Gamma}}$
such that $\CEqar{I(!_{\Gamma}}{!_{\Gamma}} \circ h \jdeq \catid{\Gamma}$.
To see that it preserves composites,
consider the commutative diagram below which defines the composite $\jcomp{}{x}{y}$ of
$x \in \thom{!_{\Delta}}{!_{\Gamma}}$ and $y \in \thom{!_{\Gamma}}{!_{\Xi}}$
in $\EtrmCatcl{E}$.
\begin{linenomath*}
\[\begin{tikzcd}[row sep=small]
\Delta	\ar[dddrr,bend right,"\catid{\Delta}"{swap}] \ar[r,"x"]
&[2ex]	\CEctxext{\Delta}{(\CEpb00{!_{\Delta}}{!_{\Gamma}})}
	\ar[dddr,bend right=5ex] \ar[rr,"\CEqar{!_{\Delta}}{!_{\Gamma}}"]
&&[5ex]	\Gamma \ar[dddr,bend right=3ex,"I(!_{\Gamma})"{swap,near end}] \ar[rr,"y"]
&&	\CEctxext{\Gamma}{(\CEpb00{!_{\Gamma}}{!_{\Xi}})}
	\ar[dd] \ar[dl,"\CEqar{!_{\Gamma}}{!_{\Xi}}", {near end}]
\\
&&	\CEctxext{\Delta}{(\CEpb00{!_{\Delta}}{!_{\Xi}})}
	\ar[dd] \ar[rr,crossing over,"\CEqar{!_{\Delta}}{!_{\Xi}}"{pos=.4}]
	\ar[from=ull,crossing over,bend right=3ex,dotted,"\jcomp{}{x}{y}"']
&&	\Xi	\ar[dd,"I(!_{\Xi})"]
\\
&&&&&	\Gamma	\ar[dl,"I(!_{\Gamma})"]
\\
&&	\Delta	\ar[rr,"I(!_{\Delta})"{swap}]	&&	\CEroot0{}	&
\end{tikzcd}\]
\end{linenomath*}
Functoriality of $\Psi$ amounts to the commutativity of the upper face.

To conclude that $(\mathrm{d},\Psi)$ is a CE-homomorphism
it remains to show that it preserves chosen pullbacks,
since the square below commutes by definition of $\Psi$.
\begin{linenomath*}
\[\begin{tikzcd}
\cat{F}/\Eroot{}	\ar[d,"\Etrmfuncl{E}"{swap}] \ar[r,"\mathrm{d}"]
&	\cat{F}	\ar[d,"\CEfun1{A}"]
\\
\EtrmCatcl{E}	\ar[r,"\Psi"]	&	\CEcat0{}
\end{tikzcd}\]
\end{linenomath*}
Let then $x \in \thom{!_{\Delta}}{!_{\Gamma}}$ and $A \in \cat{F}/\Gamma$.
It is
\begin{linenomath*}
\begin{align*}
\CEpbvar01x{!(A)}
&\jdeq	S_x \circ (W_{!_{\Delta}}/!_{\Gamma}) \circ !(A)
\\&\jdeq
\CEpbf0x \circ (\CEpbf0{I(!_{\Delta})}/!_{\Gamma}) \circ !(A)
\\&\jdeq
\mathop{!}\left(\CEpb10{\CEqar{I(!_{\Delta})}{!_{\Gamma}} \circ x}{A}\right)
\\&\jdeq
\mathop{!}\left(\CEpb00{\Psi(x)}{A}\right)
\end{align*}
\end{linenomath*}
whereas
\begin{linenomath*}
\[
\Psi(\CEqarvar{x}{!(A)}) \jdeq
\pi_2(I(!_{\Delta.(\CEpb00{\Psi(x)}{A})}),!_{\Gamma.A}) \circ \CEqarvar{x}{!(A)}
\jdeq \CEqar{\Psi(x)}{A}
\]
\end{linenomath*}
holds by commutativity of the upper face in
\begin{linenomath*}
\[\begin{tikzcd}[sep=3em]
\ctxext{\Delta}{\CEpb00{\Psi(x)}{A}}	\ar[ddd,bend right=3em,"\catid{}"{swap}]
	\ar[d,"\tfid{\CEpb00{\Psi(x)}{A}}"] \ar[drr,bend left=2em,"\catid{}"]
&&&&&
\\
\bullet	\ar[dd,fib] \ar[dr] \ar[rr]
&&	\ctxext{\Delta}{\CEpb00{\Psi(x)}{A}}	\ar[dd,fib,"\CEpb00{\Psi(x)}{A}"{near end}] \ar[dr] \ar[drrr,bend left=2em,"{\pi_2(\Psi(x),A)}"{near end}]	&&&
\\
&	\bullet	\ar[from=uul,bend left=3em,crossing over,"{\hat{\pi_2}(x,!(A))}"{pos=.9}] \ar[rr,crossing over]
&&	\bullet	\ar[rr]	&&	\ctxext{\Gamma}{A}	\ar[dd,fib,"A"]
\\
\ctxext{\Delta}{\CEpb00{\Psi(x)}{A}}	\ar[ddr,bend right,"\catid{}"{swap}] \ar[dr] \ar[rr,"\CEpb00{\Psi(x)}{A}"{near end}]
&&	\Delta	\ar[ddr,bend right,"\catid{}"{swap,near end}] \ar[dr,"x"] \ar[drrr,bend left=2em,"\Psi(x)"{near end}]	&&
\\
&	\bullet	\ar[from=uu,fib,crossing over] \ar[d,fib] \ar[rr,crossing over]
&&	\bullet	\ar[from=uu,fib,crossing over] \ar[d,fib] \ar[rr,"{\pi_2(I(!_{\Delta}),!_{\Gamma})}"{swap}]
&&	\Gamma	\ar[d,fib,"!_{\Gamma}"]
\\
&	\ctxext{\Delta}{\CEpb00{\Psi(x)}{A}}	\ar[rr,"\CEpb00{\Psi(x)}{A}"{swap}]
&&	\Delta	\ar[rr,"!_{\Delta}"{swap}]	&&	\CEroot0{}
\end{tikzcd}\]
\end{linenomath*}
This diagram commutes because all the squares not involving the top-left object are chosen pullback squares in \sys{A},
two of the remaining triangles commute by definition of $\tfid{}$,
and the third one involving $\hat{\pi_2}(x,!(A))$ commutes by~\eqref{eq:EtoCE-qar} and \cref{lem:E-CE-pair-proj}.

The component $\epsi_{\sys{A}} \colon \rEtoCE \circ \CEtorE(\sys{A}) \to \sys{A}$
of the counit at \sys{A} is defined to be the pair $(\mathrm{d},\Psi)$.
To see that this choice is natural in \sys{A} it is enough to show that the square of functors
\begin{linenomath*}
\[\begin{tikzcd}
\EtrmCatcl{\CEtorE(\sys{A})}	\ar[d] \ar[r,"\Psi_{\sys{A}}"]
&	\CEcat1{A}	\ar[d,"\hCEcat{F}"]
\\
\EtrmCatcl{\CEtorE(\sys{B})}	\ar[r,"\Psi_{\sys{B}}"]	&	\CEcat1{B}
\end{tikzcd}\]
\end{linenomath*}
commutes for every CE-homomorphism $F \colon \sys{A} \to \sys{B}$.
Note that the action of the left-hand functor coincide with that of $F$.
Commutativity of the square thus follows from
\[
F(\CEqar{x}{A}) = \CEqar{Fx}{FA}
\]
which holds by definition of CE-homomorphism.
\end{construction}

Next we prove the second claim in \cref{thm:E-CE}.
\begin{lem}\label{lem:E-CE-isos}
For every CE-system \sys{A},
the CE-homomorphism $\epsi_{\sys{A}}$ from \cref{prob:E-CE-counit} is invertible
if and only if \sys{A} is \rtdCE.
\end{lem}

\begin{proof}
Note first that each function $\psi$ in \eqref{prob:E-CE-counit:hom} induces a bijection 
\begin{linenomath*}
\begin{equation}\label{lem:E-CE-isos:counit-sq}
\begin{tikzcd}[column sep=large]
\thom{!_{\Delta}}{!_{\Gamma}}	\ar[r,"\sim"]
&	\{f \in \cat{C}(\Delta,\Gamma) \mid I(!_{\Gamma}) \circ f \jdeq I(!_{\Delta}) \}
\end{tikzcd}
\end{equation}
\end{linenomath*}
with inverse given by the universal property of the canonical pullback square below.
\begin{linenomath*}
\[\begin{tikzcd}
\ctxext{\Delta}{\CEpb00{!_{\Delta}}{!_{\Gamma}}}	\ar[d] \ar[r]
&	\Gamma	\ar[d,"I(!_{\Gamma})"]
\\
\Delta	\ar[r,"I(!_{\Delta})"]	&	\CEroot0{}
\end{tikzcd}\]
\end{linenomath*}
As soon as $\CEroot0{}$ is terminal in $\CEcat0{}$,
the right-hand set in \eqref{lem:E-CE-isos:counit-sq}
coincides with $\CEcat0{}(\Delta,\Gamma)$.
Conversely, if the counit components are invertible it follows
from \eqref{lem:E-CE-isos:counit-sq} that
$\CEcat0{}(\Delta,\CEroot0{}) \jdeq \{!_{\Delta}\}$.
\end{proof}

\begin{proof}[Proof of \cref{thm:E-CE}]
\ref*{thm:E-CE:adj}.
To complete the proof we show that,
for an E-system \sys{E} and a CE-system \sys{A}
\begin{linenomath*}
\[
\CEtorE(\epsi_{\sys{A}}) \circ \eta_{\CEtorE(\sys{A})} = \mathrm{Id}_{\CEtorE(\sys{A})}
\qquad\text{and}\qquad
\epsi_{\rEtoCE(\sys{E})} \circ \rEtoCE(\eta_{\sys{E}}) = \mathrm{Id}_{\rEtoCE(\sys{E})}.
\]
\end{linenomath*}
It is clear that these equations hold between functors on families
by the isomorphism in \eqref{sl-trm}.
It remains to show that they hold also between the term structures in the left-hand one,
and between functors on substitutions in the right-hand one.

For a CE-system \sys{A}, a family $A \in \CEfam0{}/\Gamma$ and
$y \in T(A) = \{ x \colon \Gamma \to \ctxext{\Gamma}{A} \mid I(A) \circ x \jdeq \catid{\Gamma} \}$,
\cref{lem:E-CE-EC} yields
$\eta_{\CEtorE(\sys{A})}(y) \jdeq \pi_2(\tfid{!_{\Gamma}},\CEpb00{\pi_2(!_{\Gamma},!_{\Gamma})}{A}) \circ y$.
It follows that
\begin{linenomath*}
\begin{align*}
\CEtorE(\epsi_{\sys{A}}) \circ \eta_{\CEtorE(\sys{A})}(y)
&\jdeq
\pi_2(!_{\Gamma},!_{\ctxext{\Gamma}{A}}) \circ \pi_2(\tfid{!_{\Gamma}},\CEpb00{\pi_2(!_{\Gamma},!_{\Gamma})}{A}) \circ y
\\&\jdeq
\pi_2(\pi_2(!_{\Gamma},!_{\Gamma}),A) \circ \pi_2(\tfid{!_{\Gamma}},\CEpb00{\pi_2(!_{\Gamma},!_{\Gamma})}{A}) \circ y
\\&\jdeq
\pi_2(\pi_2(!_{\Gamma},!_{\Gamma}) \circ \tfid{!_{\Gamma}}, A) \circ y
\\&\jdeq
y.
\end{align*}
\end{linenomath*}

For an E-system \sys{E}, objects $\Delta$ and $\Gamma$ and $f \in \thom{!_{\Delta}}{!_{\Gamma}}$,
\cref{lem:E-CE-EC,lem:pairproj} yield
\begin{linenomath*}
\[
\epsi_{\rEtoCE(\sys{E})} \circ \rEtoCE(\eta_{\sys{E}})(f)
\jdeq
\cprojsnd{!_{\Delta}}{\ctxwk{!_{\Delta}}{!_{\Gamma}}}{\tmext{\tfid{!_{\Gamma}}}{f}}
\jdeq
f.
\]
\end{linenomath*}

This concludes the proof of the adjunction.

\ref*{thm:E-CE:img}.
This is \cref{lem:E-CE-isos}.
\end{proof}

\subsection{Equivalence between B-systems and C-systems}
\label{ssec:eqv-b-c}

Here we describe the main contribution of our work:
the construction of an equivalence of categories between the category of 
C-systems of \Cref{sec:c-sys} and the category of B-systems of \Cref{sec:b-sys}.

\begin{lem}\label{lem:CEtoE-strat}
The functor $\CEtorE \colon \CEsys \to \Esys$ from \cref{constr:CE2E-fun} restricts to a functor
$\CEtorE \colon \strrCEsys \to \strEsys$ between stratified systems.
\end{lem}

\begin{proof}
To see that the E-system $\CEtorE(\sys{A})$ is stratified
whenever the rooted CE-system $\sys{A}$ is stratified,
note first that the underlying category $\F$ is stratified by assumption.
Weakening and substitution homomorphisms are stratified
since the pullback functor that defines them in
\cref{constr:CE2E-fun}.(\ref*{constr:CEtoE-fun:weak},\ref*{constr:CEtoE-fun:sub})
is stratified.

For a stratified CE-homomorphism $F$, the underlying functor of
the E-homomorphism $\CEtorE(F)$ is the component $F_{\F}$ of $F$
on families, which is stratified by assumption.
\end{proof}

\begin{lem}\label{lem:EtoCE-strat}
The functor $\rEtoCE \colon \rEsys \to \rCEsys$ restricts to a functor
$\rEtoCE \colon \strEsys \to \strrCEsys$ between stratified systems.
\end{lem}

\begin{proof}
Let $\sys{E}$ be a stratified E-system.
In particular, the underlying category $\F$ is stratified.
Since weakening and substitution homomorphisms are also stratified
by assumption,
so is the precomposition homomorphisms from \cref{def:Esys-precmp}.
It follows that the CE-system $\rEtoCE(\sys{E})$ is stratified.

For a stratified E-homomorphism $F$,
the component on families of the CE-homomorphism $\rEtoCE(F)$
is the underlying functor of $F$,
which is stratified by assumption.
\end{proof}

\begin{lem}\label{lem:E-CE-strat}\hfill
\begin{enumerate}
\item\label{lem:E-CE-strat:unit}
For every stratified E-system $\sys{E}$,
the unit component $\eta_{\sys{E}}$ of \cref{constr:E-CE-unit}
is a stratified E-homomorphism.
\item\label{lem:E-CE-strat:counit}
For every stratified CE-system $\sys{A}$,
the counit component $\epsi_{\sys{A}}$ of \cref{constr:E-CE-counit}
is a stratified CE-homomorphism.
\item\label{lem:E-CE-strat:adj}
The adjunction $\rEtoCE \dashv \CEtorE$ from \cref{thm:E-CE}.\ref{thm:E-CE:adj}
restricts to an adjunction
\begin{linenomath*}
\[\begin{tikzcd}[column sep=6em]
\strCEsys	\ar[r,shift right=1ex,"\CEtorE"{swap}]
		\ar[r,phantom,"{\scriptstyle\bot}"]
&	\strEsys.	\ar[l,shift right=1ex,"\rEtoCE"{swap}]
\end{tikzcd}\]
\end{linenomath*}
between subcategories of stratified structures.
\end{enumerate}
\end{lem}

\begin{proof}\hfill
\begin{enumerate}
\item The underlying functor of the unit component $\eta_{\sys{E}}$ is
the functor $! \colon \F\to \F/\Eroot{}$ from~\eqref{sl-trm}.
This functor is stratified since $L(\Eroot{}) \jdeq 0$.

\item 
The underlying functor of the counit component $\epsi_{\sys{A}}$
on families is the inverse $\mathrm{d} \colon \F/\CEroot0{} \to \F$ 
of $! \colon \F\to \F/\CEroot0{}$,
and it is stratified for the same reason.

\item 
This is a consequence of \cref{lem:CEtoE-strat,lem:EtoCE-strat}
and \Cref{lem:E-CE-strat:unit,lem:E-CE-strat:counit} just proved. \qedhere
\end{enumerate}
\end{proof}

Define a functor $\CtoB \colon \Csys \to \Bsys$ as the composite
\begin{linenomath*}
\begin{equation}\label{CtoB}
\begin{tikzcd}[column sep=4em]
\Csys	\ar[r,"\CtoCE"]	&	\strrCEsys	\ar[r,"\CEtorE"]	&	\strEsys	\ar[r,"\EtoB"]	&	\Bsys
\end{tikzcd}
\end{equation}
\end{linenomath*}
where the functors are, in order,
$\CtoCE$ from \cref{constr:C2CE-funct}
$\CEtorE$ from \cref{constr:CE2E-fun}
and $\EtoB$ from \cref{constr:EtoB-fun}.
Similarly, we obtain a functor $\BtoC \colon \Bsys \to \Csys$
in the other direction as the composite
\begin{linenomath*}
\begin{equation}\label{BtoC}
\begin{tikzcd}[column sep=4em]
\Bsys	\ar[r,"\BtoE"]	&	\strEsys	\ar[r,"\rEtoCE"]	&	\strrCEsys	\ar[r,"\CEtoC"]	&	\Csys
\end{tikzcd}
\end{equation}
\end{linenomath*}
where the functors are, in order,
$\BtoE$ from \cref{lem:BtoE-fun},
$\rEtoCE$ from \cref{constr:EtoCE-fun}
and $\CEtoC$ from \cref{def:CEtoC-fun}.

\begin{thm}\label{thm:BeqvC}
The pair of functors $\CtoB$ and $\BtoC$
establish an equivalence between the category of C-systems
and the category of B-systems.
\end{thm}

\begin{proof}
The functors defining $\CtoB$ in~\eqref{CtoB} and $\BtoC$ in~\eqref{BtoC}
are essentially inverse to each other
by \cref{thm:CasCE,cor:EeqvCE,thm:BasE}.
The claim follows since equivalences compose.
\end{proof}

\section{Conclusion}
\label{sec:conclusion}

We have constructed an equivalence between the category of C-systems and the category of B-systems, each equipped with a suitable notion of morphism.
The equivalence does not rely on classical reasoning principles such as the axiom of choice or excluded middle.
This equivalence constitutes a crucial piece in Voevodsky's research program on the formulation and solution of an initiality conjecture.

Some questions that remain open:
\begin{itemize}

\item Voevodsky has studied different type constructions on C-systems, in particular, dependent function types~\cite{MR3584698,MR3607210} and identity types~\cite{1505.06446}.
  The equivalence constructed in the present paper should be extended to type and term constructors on C-systems and B-systems.
  
\item Via Generalized Algebraic Theories, B-systems and C-systems relate to Garner's algebras for a monad on type-and-term systems~\cite{GARNER20151885}, in the form of an equivalence of categories.
  It would be very useful to have an explicit description of the maps back and forth, without passing through GATs.

\item E-systems and CE-systems should be related to other unstratified categorical structures for the interpretation of type theory, such as categories with families~\cite{Dybjer1996}.

\item Voevodsky envisioned a formalization, in a computer proof assistant, of his theory of type theories; some work by Voevodsky towards this goal is available online.%
  \footnote{\url{https://github.com/UniMath/lBsystems}, \url{https://github.com/UniMath/lCsystems}}
A formalization of the equivalence between B- and C-systems is still missing.

\item As remarked in the introduction, B-systems seem more suitable than other semantics to accommodate for modifications of the syntax (either restricting to substructural rules, or extending it with type constructors and operators).
Carrying these modifications over along the equivalence could yield corresponding formulations for C-systems and more traditional semantics.
\end{itemize}

\bibliographystyle{alphaurl}
\bibliography{refs}

\begin{thebibliography}{AENR23}

\bibitem[AENR23]{b-c-systems}
Benedikt Ahrens, Jacopo Emmenegger, Paige~Randall North, and Egbert Rijke.
\newblock {B-systems and C-systems are equivalent}.
\newblock {\em The Journal of Symbolic Logic}, page 1–9, 2023.
\newblock \href {https://doi.org/10.1017/jsl.2023.41} {\path{doi:10.1017/jsl.2023.41}}.

\bibitem[ALV18]{DBLP:journals/lmcs/AhrensLV18}
Benedikt Ahrens, Peter~LeFanu Lumsdaine, and Vladimir Voevodsky.
\newblock Categorical structures for type theory in univalent foundations.
\newblock {\em Logical Methods in Computer Science}, 14(3), 2018.
\newblock \href {https://doi.org/10.23638/LMCS-14(3:18)2018} {\path{doi:10.23638/LMCS-14(3:18)2018}}.

\bibitem[Awo18]{DBLP:journals/mscs/Awodey18}
Steve Awodey.
\newblock Natural models of homotopy type theory.
\newblock {\em Mathematical Structures in Computer Science}, 28(2):241--286, 2018.
\newblock \href {https://doi.org/10.1017/S0960129516000268} {\path{doi:10.1017/S0960129516000268}}.

\bibitem[Bla91]{blanco}
Javier Blanco.
\newblock Relating categorical approaches to type dependency, 1991.
\newblock Master thesis, Radboud University.

\bibitem[Car78]{cartmell_phd}
John~W. Cartmell.
\newblock {\em Generalised Algebraic Theories and Contextual Categories}.
\newblock PhD thesis, Oxford University, 1978.

\bibitem[Car86]{DBLP:journals/apal/Cartmell86}
John Cartmell.
\newblock Generalised algebraic theories and contextual categories.
\newblock {\em Annals of Pure and Applied Logic}, 32:209--243, 1986.
\newblock \href {https://doi.org/10.1016/0168-0072(86)90053-9} {\path{doi:10.1016/0168-0072(86)90053-9}}.

\bibitem[Car14]{Cartmell_metaGATs}
John Cartmell.
\newblock Meta-theory of generalised algebraic theories, 2014.
\newblock URL: \url{https://www.researchgate.net/publication/325763545_Meta-Theory_of_Generalised_Algebraic_Theories}.

\bibitem[Car18]{Cartmell_GAAxiom}
John Cartmell.
\newblock Generalised algebraic axiomatisations of contextual categories, 2018.
\newblock URL: \url{https://www.researchgate.net/publication/325763538_Generalised_Algebraic_Axiomatisations_of_Contextual_Categories}.

\bibitem[dBBLM]{boer-brunerie-lumsdaine-mortberg}
Menno de~Boer, Guillaume Brunerie, Peter~LeFanu Lumsdaine, and Anders Mörtberg.
\newblock Initiality for {M}artin-{L}öf type theory.
\newblock Talk given at HoTTEST seminar, \url{https://www.uwo.ca/math/faculty/kapulkin/seminars/hottest.html}.

\bibitem[Dyb96]{Dybjer1996}
Peter Dybjer.
\newblock Internal type theory.
\newblock In {\em Types for proofs and programs}, volume 1158 of {\em Lecture Notes in Computer Science}, pages 120--134, 1996.
\newblock \href {https://doi.org/10.1007/3-540-61780-9_66} {\path{doi:10.1007/3-540-61780-9_66}}.

\bibitem[FPT99]{DBLP:conf/lics/FiorePT99}
Marcelo~P. Fiore, Gordon~D. Plotkin, and Daniele Turi.
\newblock Abstract syntax and variable binding.
\newblock In {\em 14th Annual {IEEE} Symposium on Logic in Computer Science, Trento, Italy, July 2-5, 1999}, pages 193--202. {IEEE} Computer Society, 1999.
\newblock \href {https://doi.org/10.1109/LICS.1999.782615} {\path{doi:10.1109/LICS.1999.782615}}.

\bibitem[FV20]{1512.08104}
Marcelo Fiore and Vladimir Voevodsky.
\newblock Lawvere theories and {C}-systems.
\newblock {\em Proceedings of the American Mathematical Society}, 148:2297--2315, February, 2020.
\newblock \href {https://doi.org/10.1090/proc/14660} {\path{doi:10.1090/proc/14660}}.

\bibitem[Gar15]{GARNER20151885}
Richard Garner.
\newblock Combinatorial structure of type dependency.
\newblock {\em Journal of Pure and Applied Algebra}, 219(6):1885--1914, 2015.
\newblock \href {https://doi.org/10.1016/j.jpaa.2014.07.015} {\path{doi:10.1016/j.jpaa.2014.07.015}}.

\bibitem[Hof97]{Hofmann_syntax_semantics}
Martin Hofmann.
\newblock Syntax and semantics of dependent types.
\newblock In {\em Semantics and logics of computation ({C}ambridge, 1995)}, volume~14 of {\em Publ. Newton Inst.}, pages 79--130. Cambridge Univ. Press, Cambridge, 1997.
\newblock \href {https://doi.org/10.1017/CBO9780511526619.004} {\path{doi:10.1017/CBO9780511526619.004}}.

\bibitem[Joy17]{Joyal2017clan}
Andr\'e Joyal.
\newblock {Notes on Clans and Tribes}.
\newblock \href{https://arxiv.org/abs/1710.10238}{\texttt{arXiv:1710.10238}}, 2017.

\bibitem[KL21]{kapulkin2012univalence}
Krzysztof Kapulkin and Peter~LeFanu Lumsdaine.
\newblock {The simplicial model of univalent foundations (after Voevodsky)}.
\newblock {\em Journal of the European Mathematical Society}, 23(6):2071--2126, 2021.
\newblock \href {https://doi.org/10.4171/JEMS/1050} {\path{doi:10.4171/JEMS/1050}}.

\bibitem[LA24]{lamiaux2024introduction}
Thomas Lamiaux and Benedikt Ahrens.
\newblock An introduction to different approaches to initial semantics, 2024.
\newblock \href {https://arxiv.org/abs/2401.09366} {\path{arXiv:2401.09366}}.

\bibitem[{Mac}98]{MacLane98}
Saunders {Mac Lane}.
\newblock {\em {Categories for the working mathematician}}.
\newblock Springer, New York, NY, 2nd edition, 1998.

\bibitem[Nor19]{DBLP:journals/mscs/North19}
Paige~Randall North.
\newblock Identity types and weak factorization systems in {C}auchy complete categories.
\newblock {\em Mathematical Structures in Computer Science}, 29(9):1411--1427, 2019.
\newblock \href {https://doi.org/10.1017/S0960129519000033} {\path{doi:10.1017/S0960129519000033}}.

\bibitem[Pit01]{PittsCatLog}
Andrew~M. Pitts.
\newblock Categorical logic.
\newblock In {\em Handbook of Logic in Computer Science: Volume 5: Logic and Algebraic Methods}, page 39–123. Oxford University Press, Inc., USA, 2001.

\bibitem[Shu15]{Shulman2015}
Michael Shulman.
\newblock Univalence for inverse diagrams and homotopy canonicity.
\newblock {\em Mathematical Structures in Computer Science}, 25(5):1203--1277, 2015.
\newblock \href {https://doi.org/10.1017/S0960129514000565} {\path{doi:10.1017/S0960129514000565}}.

\bibitem[Str91]{streicher-semantics-of-tt}
Thomas Streicher.
\newblock {\em {Semantics of Type Theory}}.
\newblock Birkhäuser, 1991.
\newblock \href {https://doi.org/10.1007/978-1-4612-0433-6} {\path{doi:10.1007/978-1-4612-0433-6}}.

\bibitem[Tay99]{Taylor1999}
Paul Taylor.
\newblock {\em Practical Foundations of Mathematics}, volume~59 of {\em Cambridge Studies in Advanced Mathematics}.
\newblock Cambridge University Press, 1999.

\bibitem[Uem21]{uemura_phd}
Taichi Uemura.
\newblock {\em Abstract and concrete type theories}.
\newblock PhD thesis, Institute for Logic, Language and Computation, Amsterdam, 2021.
\newblock URL: \url{https://dare.uva.nl/search?identifier=41ff0b60-64d4-4003-8182-c244a9afab3b}.

\bibitem[Voe]{vv-bonn}
Vladimir Voevodsky.
\newblock Dependent type theories.
\newblock Lectures in the Max Planck Institute in Bonn. February 10-14, 2016.
\newblock URL: \url{https://www.math.ias.edu/vladimir/lectures}.

\bibitem[{Voe}14]{VV_B-systems}
Vladimir {Voevodsky}.
\newblock {B-systems}.
\newblock Preprint available as \href{http://arxiv.org/abs/1410.5389}{\texttt{arXiv:1410.5389}}. Latest version available at \url{https://www.math.ias.edu/Voevodsky/files/files-annotated/Dropbox/Unfinished_papers/Type_systems/Notes_on_Type_Systems/Bsystems/B_systems_current.pdf}, 2014.

\bibitem[Voe15]{MR3402489}
Vladimir Voevodsky.
\newblock A {C}-system defined by a universe category.
\newblock {\em Theory and Applications of Categories}, 30:Paper No. 37, 1181--1215, 2015.
\newblock URL: \url{http://www.tac.mta.ca/tac/volumes/30/37/30-37abs.html}.

\bibitem[Voe16a]{MR3584698}
Vladimir Voevodsky.
\newblock Products of families of types and $(\pi,\lambda)$-structures on {C}-systems.
\newblock {\em Theory and Applications of Categories}, 31:Paper No. 36, 1044--1094, 2016.
\newblock URL: \url{http://www.tac.mta.ca/tac/volumes/31/36/31-36abs.html}.

\bibitem[Voe16b]{MR3475277}
Vladimir Voevodsky.
\newblock Subsystems and regular quotients of {C}-systems.
\newblock In {\em A panorama of mathematics: pure and applied}, volume 658 of {\em Contemporary Mathematics}, pages 127--137. Amer. Math. Soc., Providence, RI, 2016.
\newblock \href {https://doi.org/10.1090/conm/658/13124} {\path{doi:10.1090/conm/658/13124}}.

\bibitem[Voe16c]{TempletonProposal}
Vladimir Voevodsky.
\newblock Mathematical theory of type theories and the initiality conjecture, April, 2016.
\newblock Research proposal to the Templeton Foundation for 2016-2019, project description.
\newblock URL: \url{https://www.math.ias.edu/Voevodsky/other/Voevodsky\%20Templeton\%20proposal.pdf}.

\bibitem[Voe17a]{MR3607209}
Vladimir Voevodsky.
\newblock C-systems defined by universe categories: presheaves.
\newblock {\em Theory and Applications of Categories}, 32:Paper No. 3, 53--112, 2017.
\newblock URL: \url{http://www.tac.mta.ca/tac/volumes/32/3/32-03abs.html}.

\bibitem[Voe17b]{MR3607210}
Vladimir Voevodsky.
\newblock The $(\pi,\lambda)$-structures on the {C}-systems defined by universe categories.
\newblock {\em Theory and Applications of Categories}, 32:Paper No. 4, 113--121, 2017.
\newblock URL: \url{http://www.tac.mta.ca/tac/volumes/32/4/32-04abs.html}.

\bibitem[Voe23a]{Voevodsky_relative}
Vladimir Voevodsky.
\newblock {C-system of a module over a Jf-relative monad}.
\newblock {\em Journal of Pure and Applied Algebra}, 227(6), 2023.
\newblock \href {https://doi.org/10.1016/j.jpaa.2022.107283} {\path{doi:10.1016/j.jpaa.2022.107283}}.

\bibitem[Voe23b]{1505.06446}
Vladimir Voevodsky.
\newblock Martin-{L}\"of identity types in the {C}-systems defined by a universe category.
\newblock {\em Publications mathématiques de l'IHÉS}, 138:1--67, 2023.
\newblock \href {https://doi.org/10.1007/s10240-023-00138-2} {\path{doi:10.1007/s10240-023-00138-2}}.

\end{thebibliography}

\end{document}